\documentclass[11pt,reqno,final]{article}

\usepackage[color,notref,notcite]{showkeys}
\definecolor{labelkey}{rgb}{0.6,0,1}
\usepackage{epsfig}
\usepackage{a4wide,amssymb,amsmath}
\usepackage{amsthm}
\usepackage{color}
\usepackage{url}
\usepackage{mathtools}

\usepackage{graphicx}

\usepackage[english]{babel}
\usepackage[utf8]{inputenc}
\usepackage[T1]{fontenc}
\usepackage{fancyhdr}
\usepackage{dsfont}
\usepackage{indentfirst}
\usepackage{color}
\usepackage{newlfont}
\usepackage{float}
\usepackage{amssymb}
\usepackage{latexsym}
\usepackage{mathrsfs}
\usepackage{listings}
\usepackage{breqn}
\usepackage{caption}
\usepackage{subcaption}
\usepackage{enumerate}
\usepackage{hyperref}

\newcommand{\bd}{\begin{displaymath}}
	
	\newcommand{\ed}{\end{displaymath}}
\newcommand{\be}{\begin{eqnarray}}
	\newcommand{\ee}{\end{eqnarray}}

\newcommand{\ba}{\begin{array}}
	\newcommand{\ea}{\end{array}}

\newtheorem{theorem}{\bf Theorem}[section]
\newtheorem{proposition}[theorem]{\bf Proposition}

\newtheorem{definition}[theorem]{\bf Definition}

\newtheorem{example}[theorem]{Example}

\newtheorem{remark}[theorem]{\bf Remark}
\providecommand{\keywords}[1]{\textbf{\textit{Keywords: }} #1}

\begin{document}
	
	\title{Multidimensional smoothness indicators for first-order Hamilton-Jacobi equations
	}

	\author{Maurizio Falcone\footnotemark[1] \and Giulio Paolucci\footnotemark[1] \and Silvia Tozza\footnotemark[2]}
	
	\maketitle
	
	\renewcommand{\thefootnote}{\fnsymbol{footnote}}
	
	\footnotetext[1]
	{
		Dipartimento di Matematica,  ``Sapienza" Universit{\`a}  di Roma,
		P.le Aldo Moro, 5 - 00185 Rome, Italy
		({\tt e-mail:falcone@mat.uniroma1.it,paolucci@mat.uniroma1.it})
	}
		
	\footnotetext[2]
	{
		Dipartimento di Matematica e Applicazioni “Renato Caccioppoli”,  Universit{\`a} degli Studi di Napoli Federico II, 
		Via Cintia, Monte S. Angelo, I-80126 Napoli, Italy 
		({\tt e-mail: silvia.tozza@unina.it})
	}

\begin{abstract}
The lack of smoothness is a common feature of weak solutions of nonlinear hyperbolic equations and is a crucial issue in their approximation. This has motivated several efforts  to  define appropriate indicators, based on the values of the approximate solutions, in order to detect the most troublesome regions of the domain.
This information helps  to adapt the approximation scheme in order to avoid spurious oscillations when using high-order schemes.
In this paper we propose a genuinely multidimensional extension of the WENO procedure in order to overcome the limitations of indicators based on dimensional splitting. Our aim is to obtain new regularity indicators for problems in 2D and apply them to a class of ``adaptive filtered'' schemes for first order evolutive Hamilton-Jacobi equations.  According to the usual procedure, filtered schemes are obtained by a simple coupling of a high-order scheme and a monotone scheme. The mixture is governed by a filter function $F$ and by a switching parameter $\varepsilon^n=\varepsilon^n({\Delta t,\Delta x})>0$ which  goes to 0 as $(\Delta t,\Delta x)$ is going to 0. 
The adaptivity is related to the smoothness indicators and allows to tune automatically the switching parameter $\varepsilon^n_j$  in time and space.
Several numerical tests on critical situations in 1D and 2D are presented and confirm the effectiveness of the proposed indicators and the efficiency of our scheme.\\

\noindent \keywords{
	High-order Filtered schemes, Hamilton-Jacobi equations, 2D-Smoothness indicators, Front propagation}
\end{abstract}


\section{Introduction}
\label{sec:intro}
The approximation of hyperbolic problems has to deal with non smooth or even discontinuous solutions so high-order accurate schemes are often difficult to implement and the proofs of convergence for accurate approximation schemes are limited to one dimensional cases, when they are available. From the practical point of view, a big effort has been made to develop smoothness indicators to detect the singularities and/or discontinuities of weak solutions in order to adapt the schemes and obtain high-order accuracy in the regions where the solution is regular, see e.g. \cite{JS96, JP00,Shu09}. Most of these indicators are based on 1D arguments and can be applied to multidimensional problems by dimensional splitting \cite{ABM10}. 
However, the schemes based on dimensional splitting work on every dimension by separation of variables and then they collect all the information. In this way, they are not able to detect truly multidimensional singularities that are not alligned with the axis. 
Our goal is to develop genuinely multidimensional smoothness indicators and prove their properties in the 2D setting (for genuinely multidimensional smoothness indicators, see also \cite{CSV19}). Moreover, we will show how these indicators are useful in the approximation of a class of  first order time-dependent Hamilton-Jacobi (HJ) equations, in the form
\begin{equation}
\label{eq:HJ_2}
\left\{
\begin{array}{ll}
v_t+H(x,y,v_x,v_y)=0,\qquad&(t,x,y)\in[0,T]\times\mathbb{R}^2, \\
v(0,x,y)=v_0(x,y),&x \in \mathbb{R}^2,
\end{array}
\right.
\end{equation}
where the Hamiltonian $H$ and the initial data $v_0$ are sufficiently regular functions (usually at least Lipschitz continuous) in order to secure an existence and uniqueness result for the viscosity solution (see e.g. \cite{B94}). \\
It is well-known that in general this problem does not admit classical solutions independently on the regularity of the initial condition, since its solutions may develop discontinuities in the gradient in finite time. This makes the development of accurate approximations rather challenging and the definition of appropriate \emph{smoothness indicators} a topic of primary interest, in order to prevent spurious oscillations near discontinuities for high-order methods and to exploit the good properties of the lower order (and convergent) monotone discretization.

Let us recall that starting from the pioneering works \cite{HEOC86,OS91} and  \cite{JS96,JP00,HS99} on high-order Essentially Non-Oscillatory (ENO) and Weighted ENO (WENO) schemes  respectively, for conservation laws and related HJ equations, a lot of efforts has been made in the last two decades to develop efficient smoothness indicators for grid valued functions. For the numerical approximation of conservation laws, that typically uses cell averages in the reconstructions, many slightly different versions of the indicators have been proposed and analyzed. We cite among the others the study of Arandiga et al.~\cite{ABM10}, focusing on the role of the desingularization parameter and the definition of optimal linear coefficients, and the proposals of Henrick et al.~\cite{HAP05} and of Borges et al.~\cite{BCCD08}, leading respectively to the so-called WENO-M and WENO-Z schemes, where the goal is to increase the accuracy at critical points. Moreover, a \emph{central version} of the reconstruction, which is really an ENO-type construction, has been first proposed in \cite{LPR99} and then improved later in \cite{LPR00} by the same authors in a compact version, which cuts the dependence of nonlinear weights from the point of reconstruction thus simplifying its application to adaptive grids. 

WENO smoothness indicators have been applied to HJ equations 
with minor changes, as for example in the work of Carlini et al.~\cite{CFR05} where an application to semi-lagrangian schemes is presented, but, at least to our knowledge, the theory has not been formally and rigorously extended. This fact has caused some confusion in the related literature, since the scaling factor used in the definition of the indicators, when working with non-differentiable functions, is usually mistaken. In order to fill this gap, in \cite{FPT18a,PaolucciPhD} the authors proposed a rather detailed analysis of the considered indicators and then applied the constructions, in a slightly different setting, to define new Adaptive Filtered (AF)  schemes, renewing the definition of \cite{BFS16}. In our approach the smoothness indicators are necessary in order to detect the regions in which it is safe to use the values of the numerical solution to compute the switching parameter $\varepsilon^n$, which is automatically computed and does not depend on the problem, differently from the scheme in \cite{BFS16}. The idea of filtering a high-order, possibly unstable, scheme with a convergent monotone scheme was probably introduced in \cite{LS95} and formulated in the present setting by Abgrall \cite{A09}, where a blending similar to the approach of this paper is used to obtain a convergence result without error estimates. Filtered schemes belong to the class of $\varepsilon$-monotone schemes (see \cite{AA00} for an example) 
and can be proven convergent thanks to the results of Crandall and Lions \cite{CL84} and Barles and Souganidis \cite{BS91}. We conclude mentioning  also the works by  Oberman and Salvador \cite{OS15} and Bokanowski et al. \cite{BFFGKZ15} on first-order stationary HJ equations and by Froese and Oberman \cite{FO13} on the Monge-Amp\`ere equation.

In this work  we introduce a new definition of multidimensional smoothness indicators devised for functions which may present some discontinuity in the gradient, we analyze their properties and we extend our class of AF schemes \cite{FPT18a} to more spatial dimensions.  In particular, we propose a natural extension of the definition in \cite{JP00} to structured two-dimensional grids. For a first extension of the WENO procedure to multidimensional unstructured grids we cite the work by Zhang and Shu \cite{ZS03}. In our formulation we give a very compact explicit formula for the case $r=2$, where $r$ is the order of the approximation of the function, which can be used for an easy implementation. The construction of the scheme is rather simple and various examples for the composing schemes, up to 4th-order accuracy, are presented in detail. Finally, all the proposed definitions are tested and discussed on a series of benchmarks, which testify the reliability of our indicators and the properties of high-order consistency and convergence of the adaptive filtering method. 

The paper is organized as follows: in Sect. \ref{sec:ind_2D} we describe the construction of our new 2D-indicators, designed to detect the regularity of the gradient of the solution. We prove there that the new indicators are a natural extension of the 1D indicators used in the literature. We continue in Sect.  \ref{sec:AFS_2D} exploiting the properties of the indicators to define our multidimensional AF scheme and present in detail all its elementary blocks. 
In Sect.~\ref{sec:tests} we collect some numerical tests on the newly defined smoothness indicators and on some applications for the AF scheme, comparing our scheme with the state-of-the-art methods. Finally, Sect.~\ref{sec:conclusions} contains some conclusions and future perspectives. 
In addition, \ref{sec:ind_1D} is devoted to a brief review of 1D indicators 
for reader convenience.
\section{Construction and analysis of regularity indicators in high dimension}
\label{sec:ind_2D}
Let us start by giving a multidimensional extension of the smoothness indicators introduced in \cite{JP00} and analyzed in \cite{FPT18a} for the 1D case, which are necessary for the definition of the AF scheme that will be introduced in Sect. \ref{sec:AFS_2D}. 

We consider a uniform grid in space $(x_j,y_i)=(j\Delta x,i\Delta y)$, $j, i \in \mathbb Z$, and a function $f:\mathbb R^2 \to \mathbb R$. 
Before proceeding with the construction, let us recall some important features about multivariate interpolation. 
As it is well known, there are many possibilities to define polynomials in two dimensions. 
For example, we could fix the total degree $r$ of the polynomial and consider a triangular array of points (then using polynomials in $\mathbb P_r(\mathbb R^2)$), or, as we will do in our approach, we can fix the degree $r$ in each variable and define the 2D-polynomial as the tensor product of one-dimensional polynomials ($P\otimes Q \in \mathbb Q_r(\mathbb R^2)$, with $P,Q \in \mathbb P_r(\mathbb R)$). 
Clearly, using the last approach the number of points involved in the stencil of the reconstruction increases exponentially (considering a square grid, if $n$ is the number of points of the 1D stencils, then $n^2$ is the cardinality of the 2D stencil). 
Note that in both cases the problem is well posed. In fact, we can define a unique polynomial interpolating a given function $f(x,y)$ on the points of the stencil with the desired degree. Clearly, some assumptions on the disposition of the points must be made. For example, the points must not lie on the same line, which is trivially the case for uniform Cartesian grids. This is indeed our case of study, in which we are working on structured grids.

Let us consider the general case of a rectangular stencil $\mathcal S=\{x_0,\dots,x_n\}\times\{y_0,\dots,y_m\}$.
Then, using the \emph{Newton form}, we can define the polynomial of degree $n+m$ interpolating a given function $f$ as
\begin{equation}
P(x,y) := \sum_{s=0}^{n}\sum_{t=0}^m \omega^x_{t-1}(x)\omega^y_{s-1}(y) f[x_0,\dots,x_t;y_0,\dots,y_s],
\end{equation}
where $\omega^\rho_k(\rho)=\omega^\rho_{k-1}(\rho-\rho_k)$, $\omega^\rho_{-1}=1$, for $\rho=x,y$, and the two-dimensional divided difference $f[\cdot;\cdot]$ are computed as in the one-dimensional case, keeping each time one of the two variables fixed and computing the divided difference with respect to the free variable, that is, for example
\begin{align*}
\label{div_diff_2D}
&f[x_t,y_s]:=f(x_t,y_s),\qquad \qquad t=0,\dots,m,\ s=0,\dots,n\\
&f[x_0,\dots,x_{t};y_s]:=\frac{f[x_{1},\dots,x_{t};y_s]-f[x_0,\dots,x_{t-1};y_s]}{x_{t}-x_0},\\
&f[x_0,\dots,x_t;y_0,\dots,y_s]:=\frac{f[x_1,\dots,x_t;y_0,\dots,y_{s}]-f[x_0,\dots,x_{t-1};y_0,\dots,y_{s}]}{x_{t}-x_0}.
\end{align*}
The same can be done with respect to the second variable. 
Now, if we want to define a smoothness indicator of a function $f$, as done in the one-dimensional case, we have to focus our attention on a single cell of the grid.  For example, in the cell $[x_{j-1},x_j]\times[y_{i-1},y_i]$ we define the following \emph{smoothness coefficients: }
\begin{equation}\label{beta_2D}
\beta_{k,w} := \sum_{\alpha\in \mathcal{A}} \int_{x_{j-1}}^{x_j}\int_{y_{i-1}}^{y_i} \Delta x^{\gamma_1}\Delta y^{\gamma_2} \left( \partial_x^{\alpha_1}\partial_y^{\alpha_2}P_{k,w}(x,y)\right)^2 dx dy,
\end{equation}
for $k=0,\dots,n-1$ and $w=0,\dots,m-1$, where $P_{k,w}$ is the interpolating polynomial on the stencil $\mathcal S_{k,w}=\{x_{j+k-n},\dots,x_{j+k}\}\times\{y_{i+w-m},\dots,y_{i+w}\}$, $\alpha=(\alpha_1,\alpha_2)$ is a multi-index belonging to the set $\mathcal{A} := \{(\alpha_1,\alpha_2) : \alpha_1=0, \dots, n; \alpha_2=0,\dots, m; |\alpha|\geq 2\}$, and $\gamma_1$, $\gamma_2$ must be chosen (depending on $\alpha_1$ and $\alpha_2$, respectively) 
in order to satisfy  the following properties (see Prop. \ref{prop_beta} for the 1D indicators):
\begin{itemize}
\item $\beta_{k,w}=O(\Delta^2)$ if the function is smooth in $\overline{\mathcal S}_{k,w}$
\item $\beta_{k,w}=O(1)$ if there is a singularity in $\overline{\mathcal S}_{k,w}$,
\end{itemize}
where $\overline{\mathcal S}_{k,w}=[x_{j-n+k},x_{j+k}]\times[y_{i-m+w},y_{i+w}]$ and $\Delta :=\max\{\Delta x, \Delta y\}$. 
Note that in \eqref{beta_2D} we restricted the summation to multi-indices $\alpha$ such that $|\alpha|\geq 2$ since the computation of lower order derivatives can be avoided as they are not useful in detecting discontinuities in the gradient.  

Before giving the proof of a  rigorous result, showing how to correctly choose $\gamma_1$ and $\gamma_2$ in the scaling factor in \eqref{beta_2D}, 
let us first introduce some notations in the case of singularities located along a curve $\Gamma(s)=(x(s),y(s))$, where $s$ is a proper parametrization of the curve. For any point $\xi\in \Gamma$ we consider a ball $B_\delta(\xi)$ centered in $\xi$ and of radius $\delta=O(\Delta)$, such that $\mathcal S_{k,w} \subset B_\delta(\xi)$, for $k=0,\dots,n-1$, $w=0,\dots,m-1$, and the two parts in which is divided by $\Gamma$, namely $B_\delta^-(\xi)$ and $B_\delta^+(\xi)$, such that $B_\delta^-(\xi)\bigcup B_\delta^+(\xi)=B_\delta(\xi)\setminus \Gamma$, as illustrated in Fig. \ref{Gamma_partial_eta}. In the sequel, we will drop the dependence on $\xi$ and $\delta$, since it should not cause any confusion. Moreover, we define the \emph{right (left) gradient with respect to} $\Gamma$ as the vector obtained passing to the limit from the right (left) of $\Gamma$, which is
\begin{equation}
\label{left_right_grad}
\nabla^\pm f(\xi):=\lim_{(x,y)\to \xi\atop (x,y)\in B^\pm} \nabla f(x,y).
\end{equation}
\begin{figure}[t!]
\centering
\includegraphics[trim=50 500 50 120,clip=true,scale=0.65]{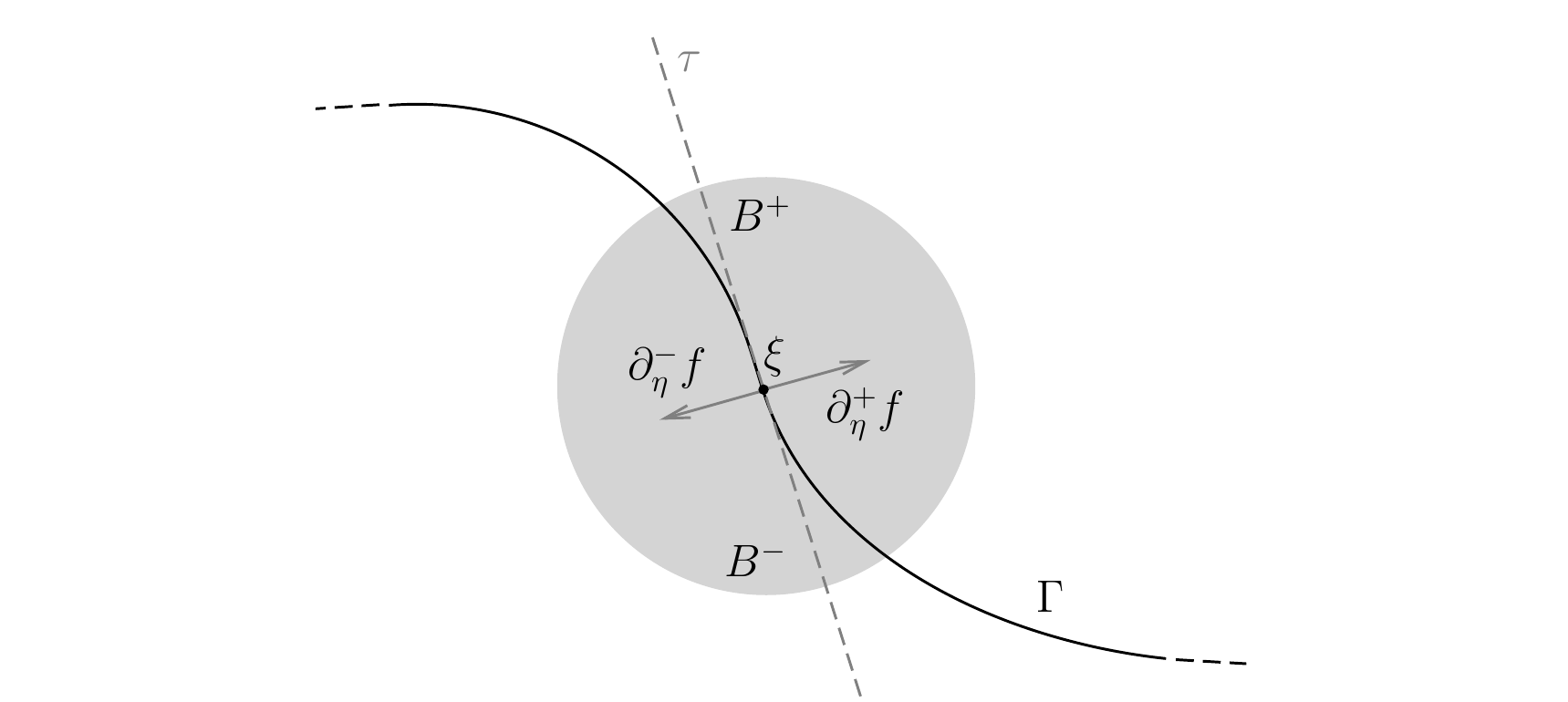}
\caption{\footnotesize{The neighborhood $B$ and the two parts in which it is divided by the curve $\Gamma$.} \label{Gamma_partial_eta}}
\end{figure}
Then, we can give the followings
\begin{definition}
Let us consider a continuous curve $\Gamma$ and a continuous function $f$. We say that $\Gamma$ is a \emph{singularity curve} for $f$ if, for any $\xi \in \Gamma$, $f\in C^k\left(B\setminus\Gamma\right)$, for $k\geq 1$, and we have
\begin{equation}
\mathcal D(\xi):=\nabla^+f(\xi)-\nabla^-f(\xi)\not=(0,0).
\end{equation}
\end{definition}
It is worth to recall that, if the curve $\Gamma \in C^1$, then the tangent and normal unit vectors at a point $\xi$, $\tau(\xi)$ and $\eta(\xi)$, respectively, are well defined and can be computed as
\begin{equation}
\tau(\xi)=\left(\frac{\dot x}{\sqrt{\dot x^2+\dot y^2}},\frac{\dot y}{\sqrt{\dot x^2+\dot y^2}}\right),\qquad \qquad \eta(\xi)=(-\tau_2,\tau_1).
\end{equation}
Consequently, we can rewrite the \emph{jump in the gradient} $\mathcal D(\xi)$ in terms of the discontinuity of the normal derivative of $f$ at $\xi$, that is $\partial_\eta f(\xi)$. More precisely, computing the left and right derivative in the normal direction as
\begin{equation}
\label{partial_eta_pm}
\partial_\eta^\pm f(\xi):=\lim_{\Delta \to 0} \frac{f(\xi\pm\eta \Delta)-f(\xi)}{\Delta}=\pm \nabla^\pm f(\xi) \cdot \eta(\xi),
\end{equation}
we can define the \emph{normal jump} as
\begin{equation}
\mathcal D_\eta(\xi):= \mathcal D(\xi) \cdot \eta(\xi)=\partial_\eta^+f(\xi) + \partial_\eta^- f(\xi),
\end{equation}
which is a scalar quantity and more easily computed, involving only one-dimensional derivatives. Moreover, keeping in mind that we are working with Cartesian structured grids, since the curve is locally linearizable and it can be approximated by its tangent line, we have that the points of the stencil around a specific point $\xi \in \Gamma$ are definitely inside $B^+$ or $B^-$, or at least remain on the tangent direction.

We are now ready to prove the reliability of our definition in the case at hand, stating the different behavior of the indicators $\beta_{k,w}$ depending on the local regularity of the function $f$, as previously discussed, all thanks to the appropriate choice of the scaling factor in \eqref{beta_2D}. For the reader convenience, in \ref{sec:ind_1D} we recalled the 1D counterpart, namely Prop. \ref{prop_beta}, which has been proved in \cite{FPT18a}. To simplify the computations, we will consider the case of a uniform grid with square stencils, that is $\Delta x=\Delta y=\Delta$ and $n=m=r$, with $r>1$. 

\begin{proposition}
\label{prop_beta_2D}
Let us assume that $f$ is a continuous function in its domain and  sufficiently regular outside $\Gamma \in C^1$ that is a singularity curve for $f$. Moreover, let us assume that at least one of the second derivatives of $f$ is non-null for $(x,y) \in B\setminus \Gamma$ and let $\beta_{k,w}$ be given by (\ref{beta_2D}) with $\gamma_i=2(\alpha_i-1)$, for $i=1,2$. Then, for $k=0,\dots,r-1$, $w=0,\dots,r-1$, the followings are true:
\begin{enumerate}[i)]
\item If $\ \Gamma \bigcap \overline {\mathcal S}_{k,w}  = \emptyset \quad \Rightarrow \quad\beta(f)_{k,w}=O(\Delta^2)$;
\item If $\ \Gamma \bigcap \overline{ \mathcal S}_{k,w} \not = \emptyset\quad \Rightarrow \quad \beta(f)_{k,w}=O(1)$.
\end{enumerate}
\end{proposition}
\begin{proof}
Let us consider $r>1$ and a generic square stencil $\mathcal S_{k,w}=\{x_{j+\nu},\dots,x_{j+\nu+r}\}\times\{y_{i+\mu},\dots,y_{i+\mu+r}\}$, for $k=0,\dots,r-1$, $w=0,\dots,r-1$, where $\nu=k-r$ and $\mu=w-r$, around a point $(x_j,y_i)$, $j$, $i\in \mathbb Z$, which is not necessarily the center of the stencil. Before proceeding with the proof, let us rewrite the integrals in (\ref{beta_2D}) in order to highlight and isolate the dependence of the various terms of the summation on the discretization parameters. Let us define 
\begin{equation}
\label{def_fh}
f_h(\theta):=f(x_j+\theta_1\Delta x, y_i+\theta_2 \Delta y),\quad \textrm{ for }\theta=(\theta_1,\theta_2) \in \mathbb Z^2,
\end{equation}
which allows to write
\begin{equation}
\label{divided_undivided}
f_h[\nu,\dots,\nu+t;\mu,\dots,\mu+s]_*:= f[x_{j+\nu},\dots,x_{j+\nu+t};y_{i+\mu},\dots,y_{i+\mu+s}]\;{t!s!\Delta x^t \Delta y^s} \quad\textrm{for }t,s=0,\dots r,
\end{equation}
where $f_h[\cdot]_*$ denotes the undivided difference of $f_h$.
Next, defining the polynomial
\begin{align}
Q_{k,w}(\theta)&:=P_{k,w}(x_j+\theta_1\Delta x,y_i+\theta_2 \Delta y)\nonumber\\
&=\sum_{s,t=0}^r\omega_{t-1}^x(x_j+\theta_1\Delta x)\omega^y_{s-1}(y_i+\theta_2\Delta y)f[x_{j+\nu},\dots,x_{j+\nu+t};y_{i+\mu},\dots,y_{i+\mu+s}]\nonumber\\
&=\sum_{s,t=0}^r q^\nu_{t-1}(\theta_1)q^\mu_{s-1}(\theta_2)  \frac{ f_h[\nu,\dots,\nu+t;\mu,\dots,\mu+s]_*}{t!s!},
\end{align}
where $q_s^\rho(\theta_i)=(\theta_i-\rho)\cdots (\theta_i-\rho-s)$, $q^\rho_{-1}=1$, for $\rho=\nu,\mu$, $i=1,2$, and using the change of variable  $(\theta_1,\theta_2) := ((x-x_j)/\Delta x,(y-y_i)/\Delta y)$, we can compute
\begin{equation}\label{eq:13}
\partial^{\alpha_1}_x\partial^{\alpha_2}_y P_{k,w}(x,y)= \partial^{\alpha_1}_x\partial^{\alpha_2}_y Q_{k,w}\left(\frac{x-x_j}{\Delta x},\frac{y-y_i}{\Delta y}\right)=\frac{1}{\Delta x^{\alpha_1}\Delta y^{\alpha_2}}\partial_{\theta_1}^{\alpha_1} \partial_{\theta_2}^{\alpha_2} Q_{k,w}(\theta_1,\theta_2).
\end{equation}
Putting \eqref{eq:13} in Eq. (\ref{beta_2D}) with the restriction on $\alpha$ before mentioned, i.e. $|\alpha|\geq 2$, we get
\begin{align}
\label{beta_2D_risc}
\beta_{k,w}&=\sum_{{\alpha_1, \alpha_2=0\atop |\alpha|\geq 2}}^r \int_{x_{j-1}}^{x_j}\int_{y_{i-1}}^{y_i} \Delta x^{2(\alpha_1-1)}\Delta y^{2(\alpha_2-1)} \left( \partial_x^{\alpha_1}\partial_y^{\alpha_2}P_{k,w}(x,y)\right)^2 dx dy \nonumber \\
&=\sum_{{\alpha_1, \alpha_2=0\atop |\alpha|\geq 2}}^r \frac{  \Delta x^{2\alpha_1-1}\Delta y^{2\alpha_2-1}}{\Delta x^{2\alpha_1}\Delta y^{2\alpha_2}}\int_{-1}^{0}\int_{-1}^{0}  \left( \partial_{\theta_1}^{\alpha_1}\partial_{\theta_2}^{\alpha_2}Q_{k,w}(\theta_1,\theta_2)\right)^2 d\theta_1 d\theta_2\nonumber \\
&= \frac{ 1}{\Delta x\Delta y}\sum_{{\alpha_1, \alpha_2=0\atop |\alpha|\geq 2}}^r\int_{-1}^{0}\int_{-1}^{0}  \left( \partial_{\theta_1}^{\alpha_1}\partial_{\theta_2}^{\alpha_2}Q_{k,w}(\theta_1,\theta_2)\right)^2 d\theta_1 d\theta_2
\end{align}
where
\begin{equation}
\label{partial_Q}
\partial_{\theta_1}^{\alpha_1}\partial_{\theta_2}^{\alpha_2}Q_{k,w}(\theta_1,\theta_2)=\sum_{t=\alpha_1}^r\sum_{s=\alpha_2}^r \frac{ f_h[\nu,\dots,\nu+t;\mu,\dots,\mu+s]_*}{t!s!}\partial_{\theta_1}^{\alpha_1}q^\nu_{t-1}(\theta_1)\partial^{\alpha_2}_{\theta_2}q^\mu_{s-1}(\theta_2).
\end{equation}
At this point, we divide our proof in several parts taking into account the different possible situations.\\

\noindent
{\em Proof of point (i)}. 
From (\ref{beta_2D_risc})-(\ref{partial_Q}) it is rather straightforward to conclude the thesis. In fact, since the function $f$ is regular by hypothesis, it is enough to apply in (\ref{divided_undivided}) the properties of the (forward) divided differences, that is
\begin{equation}
f[x_{j+\nu},\dots,x_{j+\nu+t};y_{i+\mu},\dots,y_{i+\mu+s}]= \partial_x^t\partial_y^s f(x_j,y_i)+O(\Delta),
\end{equation}
which in turn gives
\begin{equation}
\label{fh_regular}
\frac{f_h[\nu,\dots,\nu+t;\mu,\dots,\mu+s]_*}{t!s!}=\Delta x^t \Delta y^s  \partial_x^t\partial_y^s f(x_j,y_i)+O(\Delta^{t+s+1}).
\end{equation}
Hence,
the first point of the thesis follows by (\ref{beta_2D_risc}) and the hypotheses on the second derivatives of $f$.\\

\noindent
{\em Proof of  point (ii)}. For the second part of the thesis, which clearly requires more effort, we focus on a point $\xi\in \Gamma \bigcap \overline{\mathcal S}_{k,w}$ and consider a circular neighborhood $B$ containing all the points of the considered stencil, together with the two parts in which it is divided by $\Gamma$, $B^+$ and $B^-$. Moreover, in order to lighten the notations, without loss of generality we can take $\xi=(0,0)$ and use only $\Delta=\Delta x=\Delta y$ to denote the discretization parameters. \\
We consider separately the two cases:
\begin{enumerate}[(A)]
\item $\xi \in {\mathcal S}_{k,w}$, i.e. the curve $\Gamma$  has a point on the the grid;
\item $\xi\in \overline{\mathcal S}_{k,w}\setminus {\mathcal S}_{k,w}$.
\end{enumerate}
\underline{\emph{Case A.}} Let us begin by recalling that for the 1D undivided differences the following equality holds
\begin{equation}
f_h[\mu,\dots,\mu+s]_*=\sum_{j=0}^{s-l}\binom{s-l}{j}(-1)^{s-l-j}f_h[\mu+j,\dots,\mu+j+l]_*,\qquad\textrm{ for }l=0,\dots,s,
\label{fh_il2}
\end{equation}
with $s\geq 0$ (see Lemma A.1, p. 37 of \cite{FPT18a}). 
Consequently, since the two-dimensional undivided differences are obtained by successive one-dimensional computations, we can use (\ref{fh_il2}) to rewrite the various terms in (\ref{partial_Q}) in order to focus only on second order undivided differences. More precisely, if $t,s >0$, using (\ref{fh_il2}) with $l=1$ we have
\begin{align}
\label{fh_2D}
f_h[\nu,\dots,&\nu+t;\mu,\dots,\mu+s]_*=\sum_{j=0}^{t-1}\binom{t-1}{j}(-1)^{t-j-1}f_h[\nu+j,\nu+j+1;\mu,\dots,\mu+s]_*\\
&=\sum_{j=0}^{t-1}\binom{t-1}{j}(-1)^{t-j-1}\sum_{i=0}^{s-1}\binom{s-1}{i}(-1)^{s-i-1}f_h[\nu+j,\nu+j+1;\mu+i,\mu+i+1]_*\nonumber \\
&=\sum_{j=0}^{t-1}\binom{t-1}{j}\sum_{i=0}^{s-1}\binom{s-1}{i}(-1)^{s+t-(i+j)}f_h[\nu+j,\nu+j+1;\mu+i,\mu+i+1]_*,\nonumber
\end{align}
whereas, if $s=0$ or $t=0$ we use directly (\ref{fh_il2}) with $l=2$ in the appropriate direction to have, for example if $t>0$,
\begin{equation}
\label{fh_1D_2}
f_h[\nu,\dots,\nu+t;\mu]_*=\sum_{j=0}^{t-2}\binom{t-2}{j}(-1)^{t-j}f_h[\nu+j,\nu+j+1,\nu+j+2;\mu]_*.
\end{equation}
It is now clear that, using (\ref{fh_2D})-(\ref{fh_1D_2}), we can restrict our study to 4-points square stencils $\{\theta_1,\theta_1+1\}\times\{\theta_2,\theta_2+1\}$ for 
$\theta$ such that $\theta \Delta \in \mathcal S_{k,w}$, and to one-directional 3-points stencils. See Fig. \ref{stencil_cases_A} for some examples in the case $r=2$.
Then, keeping in mind that the tangent line definitely partitions the points of the stencil $\mathcal S_{k,w}$ in two sets (or at least some remain in the tangent direction) and considering the possible configurations inside the ball $B$ around $\xi$, we have at most three possible situations:
\begin{enumerate}[{\em A}1)]
\item all the nodes of the stencil are inside $B^+$ (or $B^-$);
\item some of the nodes are in $B^+$ and some in $B^-$;
\item some nodes are on $\Gamma$.
\end{enumerate}
\begin{figure}[t!]
\centering
\includegraphics[trim=150 500 150 120,clip=true,scale=0.6]{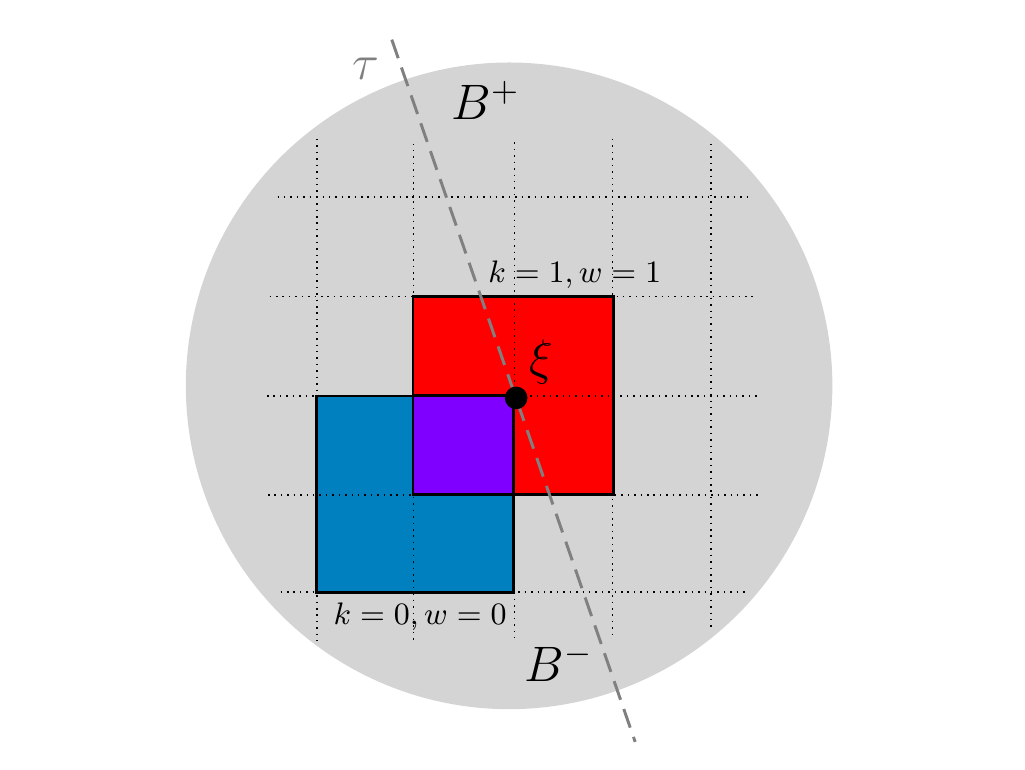}
\includegraphics[trim=150 500 150 120,clip=true,scale=0.6]{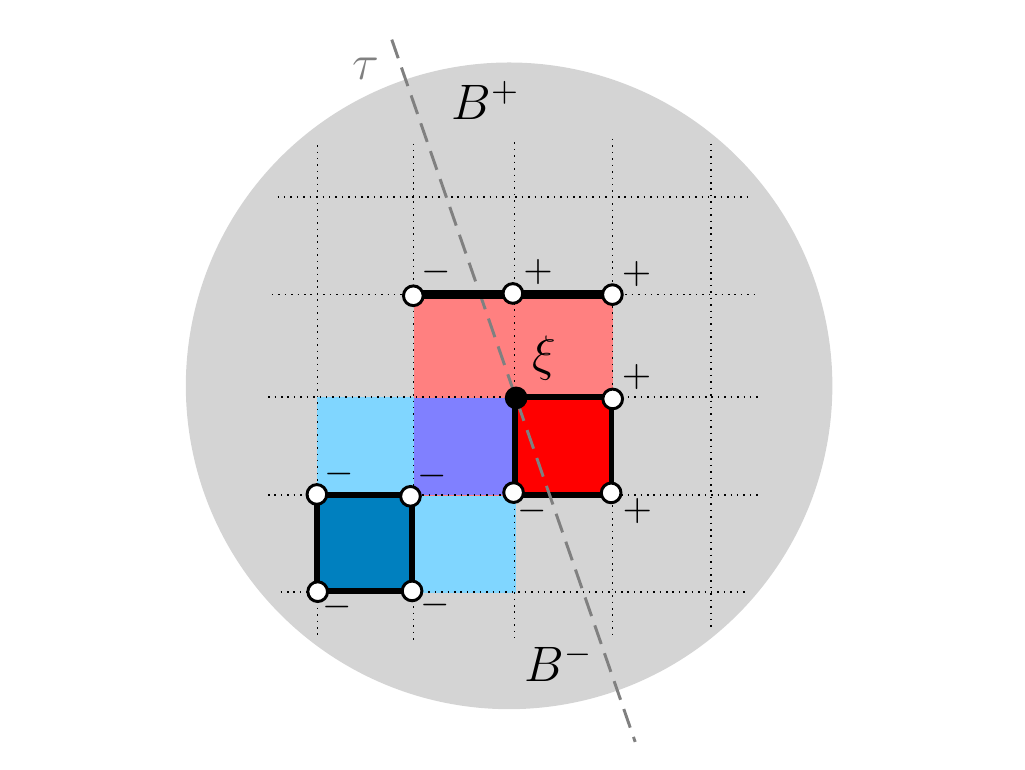}
\caption{\footnotesize{Case A for $r=2$. On the left, two possible stencils around the point $\xi \in \mathcal S_{k,w}$. 
On the right, examples of second order divided differences composing the stencils, for case A1 (darker blue square), case A3 (darker red square) and the 3-points stencil for case A2 (marked with a thicker black line on the top).
\label{stencil_cases_A}}
}
\end{figure}
{\em Proof of case A1:} It is clear that in the first case, since the function is regular by hypothesis, using (\ref{fh_regular}) for $t$, $s$ such that $t+s=2$ we can conclude that the second order undivided differences are of the order $O(\Delta^2)$. 

Next, for the other two cases, the idea of the proof is to project the points of the stencil along the normal direction to $\Gamma$ using Taylor expansions, exploiting the orientation given by the tangent line. More precisely, we project the points in $B^+$ along $\eta$ and those in $B^-$ along $-\eta$, in order to split the different contributions. Then, using the obvious identity
\begin{equation}
\label{develop_theta}
\theta \Delta =\pm\eta \Delta +(\theta\mp \eta)\Delta,\qquad\textrm{ for }\theta: \theta\Delta \in \mathcal S_{k,w},
\end{equation}
and expanding in Taylor series, we have
\begin{equation}
\label{develop_f_theta}
f(\theta \Delta) =f(\pm\eta \Delta) +(\theta\mp \eta)\Delta \cdot \nabla f(\pm \eta \Delta)+O(\Delta^2),\qquad\textrm{ for }\theta\Delta \in B^\pm.
\end{equation}
{\em Proof of  case A2:} Considering the square stencils, the only possible configurations are:
\begin{itemize}
\item 2 points in $B^+$ and 2 in $B^-$;
\item 3 points in $B^\pm$ and 1 in $B^\mp$,
\end{itemize}
whereas, for the 3-points stencils we can only have 2 points in $B^\pm$ and 1 in $B^\mp$. There are clearly some redundant cases due to the symmetry of the construction, so we present the computations only for some representative choices. For the square stencil, if 2 points are in $B^+$ and 2 in $B^-$, e.g. $\{(\theta_1,\theta_2),(\theta_1,\theta_2+1)\}\in B^-$, $\{(\theta_1+1,\theta_2),(\theta_1+1,\theta_2+1)\}\in B^+$, we have 
\begin{align}
\label{fh_square_2_2}
f_h[\theta_1,\theta_1+1;\theta_2,\theta_2+1]_*&=f((\theta_1+1)\Delta,(\theta_2+1)\Delta)-f(\theta_1 \Delta,(\theta_2+1)\Delta )\nonumber \\
&\quad-f((\theta_2+1)\Delta,\theta_2 \Delta)+f(\theta_1 \Delta,\theta_2\Delta)\nonumber \\
&=f(\eta \Delta )+(\theta_1+1-\eta_1)\Delta f_x(\eta \Delta) +(\theta_2+1-\eta_2)\Delta f_y(\eta \Delta) \nonumber \\
&\quad-f(-\eta \Delta)-(\theta_1+\eta_1)\Delta f_x(-\eta \Delta) -(\theta_2+1+\eta_2)\Delta f_y(-\eta \Delta)  \\
&\quad-f(\eta \Delta )-(\theta_1+1-\eta_1)\Delta f_x(\eta \Delta) -(\theta_2-\eta_2)\Delta f_y(\eta \Delta) \nonumber \\
&\quad+f(-\eta\Delta)+(\theta_1+\eta_1)\Delta f_x(-\eta\Delta)+(\theta_2+\eta_2)\Delta f_y(-\eta\Delta) +O(\Delta^2) \nonumber \\
&=\Delta \left(f_y(\eta \Delta )-f_y(-\eta \Delta)\right)+O(\Delta^2).\nonumber
\end{align}
Then, using forward and backward finite difference approximations, we can rewrite the last two terms as
\begin{equation}
\label{forward_backward}
f_y(\eta \Delta)=\frac{f(\eta_1\Delta , \eta_2\Delta)-f(\eta_1\Delta,0)}{\eta_2 \Delta}+O(\Delta^2),\quad f_y(-\eta \Delta)=\frac{f(-\eta_1\Delta , 0)-f(-\eta_1\Delta,-\eta_2\Delta)}{\eta_2 \Delta}+O(\Delta^2),
\end{equation}
and conclude that
\begin{align}
\label{comput_J_eta}
\Delta \left(f_y(\eta \Delta )-f_y(-\eta \Delta)\right)&=\frac{1}{\eta_2}\left(f(\eta_1\Delta , \eta_2\Delta)-f(\eta_1\Delta,0)-f(-\eta_1\Delta , 0)+f(-\eta_1\Delta,-\eta_2\Delta)\right)\nonumber\\
&=\frac{1}{\eta_2}\left(\Delta \partial^+_{\eta}f(0,0)-\eta_1\Delta f_x(0,0)+\Delta \partial^-_{\eta}f(0,0)+\eta_1\Delta f_x(0,0)\right) +O(\Delta^2) \nonumber\\
&=\frac{\Delta}{\eta_2}\left(\partial^+_{\eta} f(0,0)+\partial^-_{\eta}f(0,0)\right)+O(\Delta^2)\nonumber \\
&=\frac{\Delta}{\eta_2}\mathcal D_{\eta}(0,0)+O(\Delta^2) =O(\Delta),
\end{align}
as we wanted. 
Repeating the computation in the case there are 3 points in $B^+$ and 1 in $B^-$, e.g. $\theta \Delta \in B^-$, we have that
\begin{align}
\label{fh_square_3_1}
f_h[\theta_1,\theta_1+1;\theta_2,\theta_2+1]_*&=f(\eta \Delta )+(\theta_1+1-\eta_1)\Delta f_x(\eta \Delta) +(\theta_2+1-\eta_2)\Delta f_y(\eta \Delta) \nonumber \\
&\quad-f(\eta \Delta)-(\theta_1-\eta_1)\Delta f_x(\eta \Delta) -(\theta_2+1-\eta_2)\Delta f_y(\eta \Delta) \nonumber \\
&\quad-f(\eta \Delta )-(\theta_1+1-\eta_1)\Delta f_x(\eta \Delta) -(\theta_2-\eta_2)\Delta f_y(\eta \Delta) \\
&\quad+f(-\eta\Delta)+(\theta_1+\eta_1)\Delta f_x(-\eta\Delta)+(\theta_2+\eta_2)\Delta f_y(-\eta\Delta) +O(\Delta^2) \nonumber \\
&\hspace{-3.6cm}=\underbrace{f(-\eta\Delta)-f(\eta\Delta)+\Delta \eta \cdot \left(\nabla f(-\eta \Delta)+\nabla f(\eta \Delta)\right)}_{\textrm{(I)}} +\underbrace{\Delta \theta \cdot \left(\nabla f(-\eta \Delta )-\nabla f(\eta \Delta)\right)}_{\textrm{(II)}}+O(\Delta^2).\nonumber
\end{align}
For the first part (I), using (\ref{partial_eta_pm}) we have that
\begin{equation}
f(\pm\eta \Delta)=f(0,0) +\Delta \partial^\pm_\eta f(0,0)+O(\Delta^2),
\end{equation}
which, together with the definition of directional derivative, leads to
\begin{align}
f(-\eta\Delta)-f(\eta\Delta)+\Delta \eta \cdot &\left(\nabla f(-\eta \Delta)+\nabla f(\eta \Delta)\right) \nonumber \\
&=\Delta \left(\partial^-_\eta f(0,0)-\partial^+_\eta f(0,0)\right)+\Delta \left(\partial^+_\eta f(\eta \Delta)-\partial^-_\eta f(-\eta\Delta)\right)\nonumber \\
&=O(\Delta^2),
\end{align}
whereas for the second term (II), a computation similar to (\ref{comput_J_eta}) directly gives
\begin{equation}
\Delta \theta \cdot \left(\nabla f(-\eta \Delta )-\nabla f(\eta \Delta)\right)=-\Delta \left(\frac{\theta_1}{\eta_1}+\frac{\theta_2}{\eta_2}\right)\mathcal D_\eta(0,0)+O(\Delta^2)=O(\Delta).
\end{equation}
Using the same technique, we can get the same conclusion also in the case of the 3-points stencil. In fact, if for example 2 points are in $B^-$ and 1 in $B^+$, e.g. $(\theta_1+2,\theta_2)\in B^+$, we have
\begin{align}
f_h[\theta_1,\theta_1+1,\theta_1+2;\theta_2]_*&=f((\theta_1+2)\Delta,\theta_2\Delta)-2f((\theta_1+1) \Delta,\theta_2\Delta )+f(\theta_1 \Delta,\theta_2\Delta)\nonumber \\
&=f(\eta \Delta )+(\theta_1+2-\eta_1)\Delta f_x(\eta \Delta) +(\theta_2-\eta_2)\Delta f_y(\eta \Delta)  \\
&\quad-2\left(f(-\eta \Delta)+(\theta_1+1+\eta_1)\Delta f_x(-\eta \Delta) +(\theta_2+\eta_2)\Delta f_y(-\eta \Delta)\right) \nonumber\\
&\quad+f(-\eta\Delta)+(\theta_1+\eta_1)\Delta f_x(-\eta\Delta)+(\theta_2+\eta_2)\Delta f_y(-\eta\Delta) +O(\Delta^2) \nonumber \\
&\hspace{-4cm}
=\underbrace{f(\eta\Delta)-f(-\eta\Delta)-\Delta \eta \cdot \left(\nabla f(-\eta \Delta)+\nabla f(\eta \Delta)\right)}_{\textrm{(I)}} +\underbrace{\Delta( \theta+2) \cdot \left(\nabla f(\eta \Delta )-\nabla f(-\eta \Delta)\right)}_{\textrm{(II)}}+O(\Delta^2),\nonumber
\end{align}
which analogously to the previous cases, leads to
\begin{equation}
f_h[\theta_1,\theta_1+1,\theta_1+2;\theta_2]_*=\Delta \left(\frac{\theta_1+2}{\eta_1}-\frac{\theta_2+2}{\eta_2}\right)\mathcal D_\eta(0,0)+O(\Delta^2)=O(\Delta).
\end{equation}
Finally, exploiting relations (\ref{fh_2D})-(\ref{fh_1D_2}), the previous computations directly imply that, if the stencil of some undivided difference in (\ref{partial_Q}) intersects the curve $\Gamma$, then the formula (\ref{beta_2D_risc}) gives
\begin{equation}
\beta_{k,w}=\frac{(O(\Delta))^2}{\Delta x \Delta y}=O(1),
\end{equation}
whence the thesis for Case A2.

\noindent
{\em Proof of Case A3: }
To complete the proof for Case A, we have to consider the situation in which some point of the stencil of the second order undivided difference stays along the tangent line (case \emph{A3)} ), which happens in particular when $\xi$ is one of the considered nodes belonging to the stencil. 
For the square stencils we can have:
\begin{itemize}
\item 3 points in $B^\pm$ and 1 on $\Gamma$
\item 2 points in $B^\pm$, 1 in $B^\mp$ and 1 on $\Gamma$
\item 1 point in $B^\pm$, 1 in $B^\mp$ and 2 on $\Gamma$, that is when the curve is diagonal to the grid. 
\end{itemize}
For the 3-points stencil we can only encounter:
\begin{itemize}
\item 2 points in $B^\pm$ and 1 on $\Gamma$
\item 1 points in $B^\pm$, 1 in $B^\mp$ and 1 on $\Gamma$.
\end{itemize}
When the singularity touches the stencil only at some vertex, then the undivided difference cannot see the jump. In fact, computing for example in the case of the square stencil, if $\theta \Delta\in \Gamma$ and the other points are in $B^+$,
\begin{align}
f_h[\theta_1,\theta_1+1;\theta_2,\theta_2+1]_*&=f(\eta \Delta )+(\theta_1+1-\eta_1)\Delta f_x(\eta \Delta) +(\theta_2+1-\eta_2)\Delta f_y(\eta \Delta) \nonumber \\
&\quad-f(\eta \Delta)-(\theta_1-\eta_1)\Delta f_x(\eta \Delta) -(\theta_2+1-\eta_2)\Delta f_y(\eta \Delta) \nonumber \\
&\quad-f(\eta \Delta )-(\theta_1+1-\eta_1)\Delta f_x(\eta \Delta) -(\theta_2-\eta_2)\Delta f_y(\eta \Delta) \nonumber \\
&\quad+f(\eta\Delta) +\Delta (\theta-\eta)\cdot \nabla^+ f(\eta \Delta)+O(\Delta^2) \nonumber \\
&=\Delta (\theta-\eta) \cdot\left( \nabla^+ f(\eta \Delta)-\nabla f(\eta \Delta)\right)+O(\Delta^2)=O(\Delta^2).
\end{align}
Using analogous reasoning it is straightforward to show that for the other two cases of the square stencil we can obtain (\ref{fh_square_3_1}), (\ref{fh_square_2_2}), respectively. In fact, it is enough to develop the points on $\Gamma$ along $\pm\eta\Delta$ using $\nabla^\pm f$, which are always well defined. In the same way, for the 3-points stencil, if $(\theta_1+1,\theta_2)\Delta \in \Gamma$, we get
\begin{align}
f_h[\theta_1,\theta_1+1,\theta_1+2;\theta_2]_*&=f(\eta \Delta )+(\theta_1+2-\eta_1)\Delta f_x(\eta \Delta) +(\theta_2-\eta_2)\Delta f_y(\eta \Delta) \nonumber \\
&\quad-f(\eta \Delta)-\Delta((\theta_1+1,\theta_2)-\eta)\cdot \nabla^+f(\eta \Delta)-f(-\eta \Delta) \nonumber \\
&\quad-\Delta((\theta_1+1,\theta_2)+\eta)\cdot \nabla^-f(-\eta \Delta) \nonumber \\
&\quad+f(-\eta\Delta)+(\theta_1+\eta_1)\Delta f_x(-\eta\Delta)+(\theta_2+\eta_2)\Delta f_y(-\eta\Delta) +O(\Delta^2) \nonumber \\
&=\Delta\left[ ((\theta_1+1,\theta_2)-\eta) \cdot\left( \nabla f(\eta \Delta)-\nabla^+ f(\eta \Delta)\right)\right. \nonumber \\
&\quad\left.+(\theta+\eta) \cdot\left( \nabla f(-\eta \Delta)-\nabla^- f(-\eta \Delta)\right)\right]\nonumber\\
&\quad+\Delta \left(f_x(\eta \Delta )-f^-_x(-\eta \Delta)\right)+O(\Delta^2)\nonumber\\
&=\frac{\Delta}{\eta_1}\mathcal D_{\eta}(0,0)+O(\Delta^2)=O(\Delta).
\end{align}
At this point, as it has been done in the previous case, it is enough to use \eqref{fh_2D}-\eqref{fh_1D_2} in \eqref{beta_2D_risc} to obtain the thesis also for case A3.

\begin{figure}[t!]
\vspace{-0.4 cm}
\centering
\includegraphics[trim=100 370 100 130,clip=true,scale=0.35]{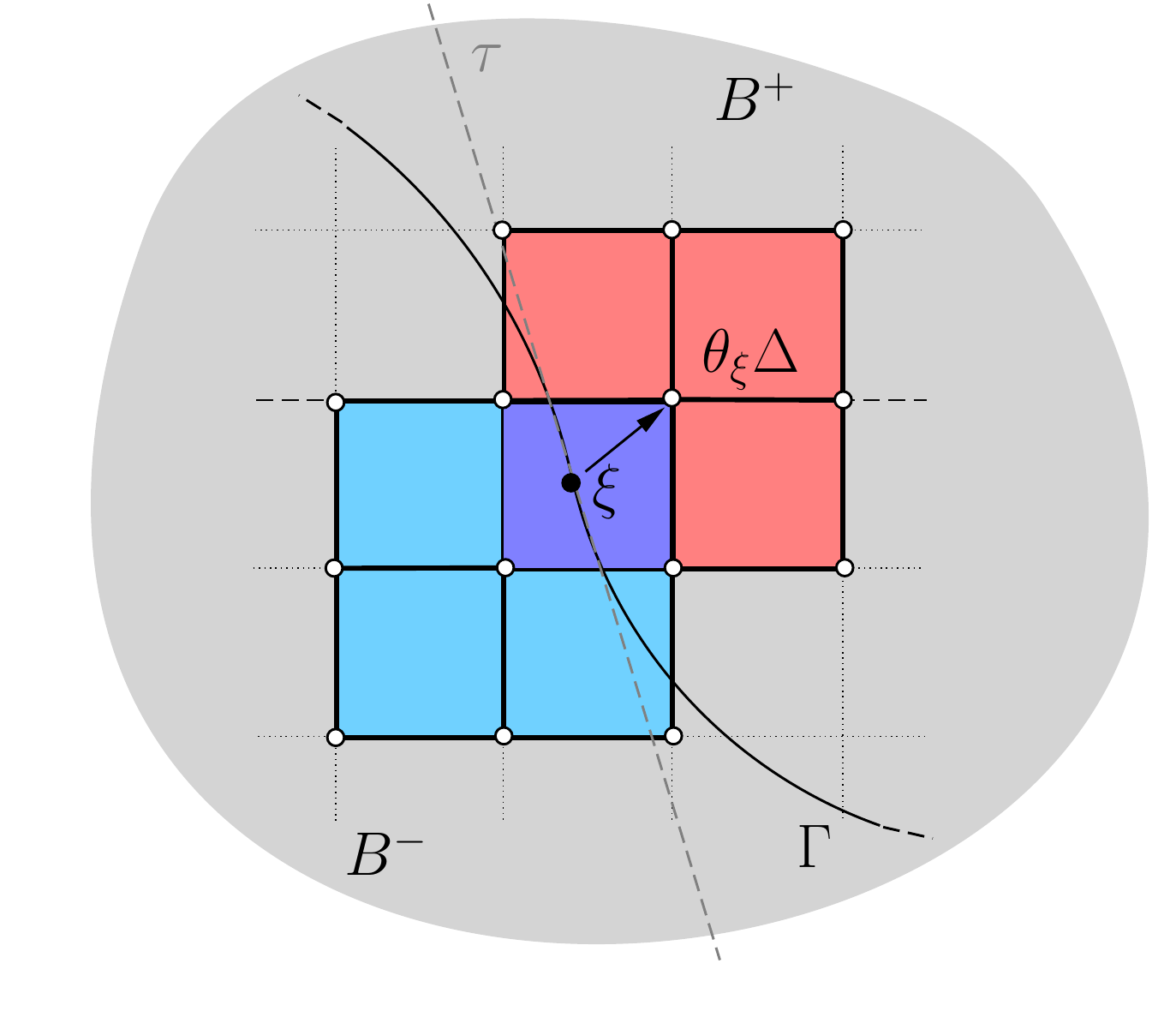}
\quad
\includegraphics[trim=170 530 170 130,clip=true,scale=0.85]{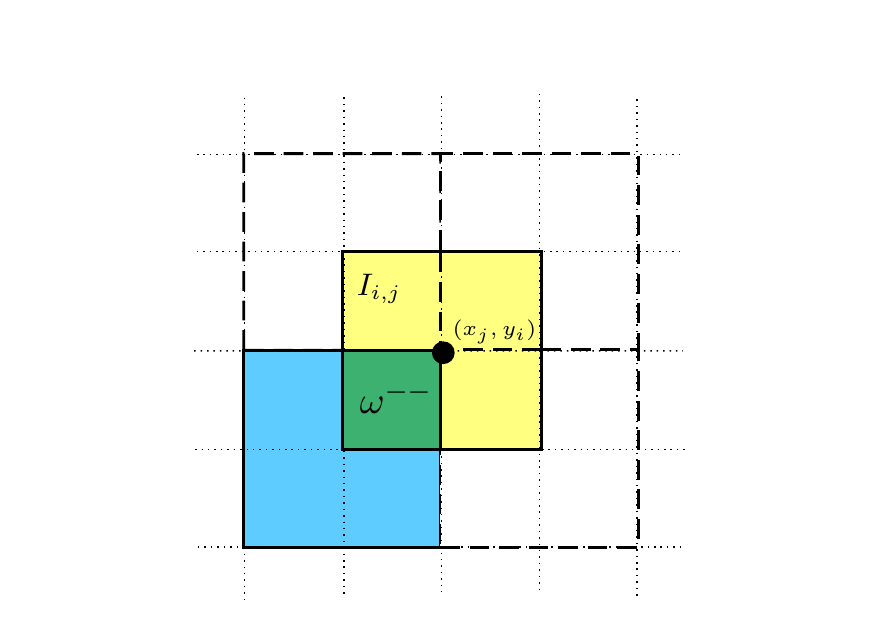}
\caption{\footnotesize{On the left, two possible stencils around the point $\theta_\xi \Delta$ crossed by the singularity curve $\Gamma$ (Case B). On the right, stencils of the polynomials needed to compute $\omega^{--}$. } \label{stencil_cases_B}}
\end{figure}
\noindent
\underline{\emph{Case B.}} If the curve $\Gamma$ does not intersect the points of the stencil $\mathcal S_{k,w}$, but it is such that $ \overline{\mathcal S}_{k,w}\bigcap \Gamma \not = \emptyset$, focusing on a particular point $\xi=(0,0) \in \Gamma$, there always exists a vector $\theta_\xi \in \mathbb R^2$ such that $\theta_\xi\Delta \in \mathcal S_{k,w}$, as shown in Fig. \ref{stencil_cases_B}, on the left. Consequently, we can repeat the same computations of the previous case simply by defining the function
\begin{equation}
f_{h,\xi}(\theta):= f((\theta-\theta_\xi)\Delta),
\end{equation}
using which we have
\begin{equation}
\label{develop_fhxi_theta}
f_{h,\xi}(\theta) =f(\pm\eta \Delta) +(\theta-\theta_\xi\mp \eta)\Delta \cdot \nabla f(\pm \eta \Delta)+O(\Delta^2),\qquad\textrm{ for }\theta\Delta \in B^\pm,
\end{equation}
and the thesis for Case B follows with minor modifications.
\end{proof}
\begin{remark}
\label{def_sigma_2D}
It is good to notice that, if we remove the assumption on the second order derivatives in regions of regularity, we can only obtain the inequality $\beta_{k,w}\leq O(\Delta^2)$ for case i). 
This fact, as pointed out firstly in \cite{ABM10} in the one-dimensional case and then taken up in a multidimensional setting several times, e.g. in the WENOZ schemes \cite{CSV19}, may cause a loss of accuracy at critical points. 
Nevertheless, it can be avoided by simply adding a small quantity $\sigma_h=\sigma \Delta^2 $, for some constant $\sigma>0$, to the smoothness coefficients $\beta_{k,w}$, as we will assume in the sequel, choosing $\sigma=2$.
\end{remark}
\subsection{Construction of a smoothness indicator function in 2D}\label{subsec:phi2D}
We want to construct a \emph{smoothness indicator function} $\phi$ such that
\begin{equation}\label{eq:phi_2D}
\phi_{i,j}=\phi(\omega_{i,j}):=\left\{
\begin{array}{ll}
1\qquad&\textrm{ if the function } f\textrm{ is regular in }I_{i,j},  \\
0&\textrm{ if }I_{i,j} \textrm{ contains a point of singularity,}
\end{array}
\right.
\end{equation}
where $I_{i,j}:=[x_{j-1},x_{j+1}]\times[y_{i-1},y_{i+1}]$ and $\omega_{i,j}$ is the \emph{smoothness indicator} at the node $(x_j,y_i)$ depending on the values of the function $f$. 
From here on, in order to obtain an easy and explicit formula, we will focus on the case $n=m=2$, which is also enough for our needs. Notice that with this assumption we work with polynomials of second degree in each variable, constructed on a stencil of nine points. Then, (\ref{beta_2D}) simply reads 
\begin{equation}
\label{beta_2D_2}
\beta_{k,w}= \sum_{{\alpha_1, \alpha_2=0\atop |\alpha|\geq 2}}^2 \int_{x_{j-1}}^{x_j}\int_{y_{i-1}}^{y_i} \Delta x^{2(\alpha_1-1)}\Delta y^{2(\alpha_2-1)}\left( \partial_x^{\alpha_1}\partial_y^{\alpha_2}P_{k,w}(x,y)\right)^2 dx dy,
\end{equation}
where we have made the choice $\gamma_i=2(\alpha_i-1)$, $i=1, 2$, as suggested by Prop. \ref{prop_beta_2D}. 
From now on, we modify a little bit the notation in order to easily change the integration domain in each subcase. More in detail, referring to Fig. \ref{stencil_cases_B} on the right, we split the domain $I_{i,j}:=[x_{j-1},x_{j+1}]\times [y_{i-1},y_{i+1}]$ into four subcells. For each subdomain we inspect separately the regularity of the function $f$ by comparing only the information given by the ``outer'' (light blue in the figure) and the ``inner'' stencil (yellow in the figure). Consequently, we need only one index to denote the respective stencil, using `$0$' for the inner stencil, and `$1$' for the outer one. 
We refer to these subcells with the superscripts `$\pm\ \pm$, $\pm\ \mp$', according to the sign of the shift between the point $(x_j,y_i)$ and its symmetrical edge with respect to the center of the considered cell, in both $x$ and $y$ directions, respectively. For example, if we focus on the cell $[x_{j-1},x_j]\times[y_{i-1},y_i]$ as in Fig. \ref{stencil_cases_B},
we use the superscript `$-\ -$'.
In this way, we define four indicators $\omega^{\pm\pm}$ and $\omega^{\pm \mp}$ for each point $(x_j,y_i)$ which quantify the regularity of the function in all the subcells around the point. Finally, we take $\omega=\min\{\omega^{\pm \pm}, \omega^{\pm \mp}\}$ as the smoothness indicator for the domain $I_{i,j}$. In order to compute these indicators, we always focus on the polynomial constructed on the central stencil $\mathcal S_0=\{x_{j-1},x_j,x_{j+1}\}\times\{y_{i-1},y_i,y_{i+1}\}$ and compare it with the polynomial constructed on the symmetrical stencil with respect to the considered subcell (where they overlap), denoting the respective smoothness coefficients by $\beta_0$ and $\beta_1$. Note that with this approach even if we are using the ``central'' polynomial for all the four indicators, we have to recompute $\beta_0$ for each case because of the change in the integration domain.
\begin{remark}\label{remark_correct_def}
Notice that the situation is slightly, but fundamentally, different from the one-dimensional case, in which through our procedure we are able to measure the regularity only in the \emph{open} interval $I_j=(x_{j-1},x_{j+1})$. That is indeed natural if we are focusing on the point $x_j$, since the integrals (\ref{beta_k_m})-(\ref{beta_k_p}), with $k=0$, are well-defined if the solution is regular in the open interval $(x_{j-1},x_{j+1})$.
On the contrary, in the two-dimensional case, since the boundary of the domain $I_{i,j}$ is a rectangle, there are at least two possible directions, called $v_1$ and $v_2$, along which we can compute the directional derivative at any point $(x,y)\in I_{i,j}$. Then, if there exists a point $\xi=(x_s,y_s)\in I_{i,j}$ such that
\begin{equation}
\nabla_{v_1} f(\xi)\not = \nabla_{v_2} f(\xi),
\end{equation}
it seems reasonable to consider the interpolating polynomial $P_0(x,y)$ not well defined in the whole domain, also in the case that $\xi$ is one of the points of the boundary.
Therefore, the function $f$ should be considered singular in $I_{i,j}$.
\end{remark}
Now, we complete the construction for the case of $\omega^{--}$, 
the other three follow the exact same lines. Assuming without loss of generality that $(x_j,y_i)=(0,0)$, and using the function $f_h$ defined by \eqref{def_fh}, we can write the polynomials as
\begin{align}
P_0^{--}(x,y)&=f_h(-1,-1)+(x+\Delta x)\frac{ f_h[-1,0;-1]_*}{\Delta x}+(y+\Delta y) \frac{ f_h[-1;-1,0]_*}{\Delta y} \\
&\quad+x(x+\Delta x)\frac{f_h[-1,0,1;-1]_*}{2\Delta x^2}
+y(y+\Delta y)\frac{f_h[-1;-1,0,1]_*}{2\Delta y^2}\nonumber\\
&\quad+(x+\Delta x)(y+\Delta y)\frac{f_h[-1,0;-1,0]_*}{\Delta x\Delta y}
+x(x+\Delta x)(y-\Delta y)\frac{f_h[-1,0,1;-1,0]_*}{2\Delta x^2 \Delta y}\nonumber\\
&\quad+y(x+\Delta x)(y+\Delta y)\frac{f_h[-1,0;-1,0,1]_*}{2\Delta x \Delta y^2}+xy(x+\Delta x)(y+\Delta y)\frac{f_h[-1,0,1;-1,0,1]_*}{4\Delta x^2\Delta y^2},\nonumber
\end{align}
for the reference stencil, and
\begin{align}
\hspace{-0.4cm}P_1^{--}(x,y)&=f_h(0,0)+x\frac{f_h[0,-1;0]_*}{\Delta x}+y\frac{f_h[0;0,-1]_*}{\Delta y}+x(x+\Delta x)\frac{f_h[0,-1,-2;0]_*}{2\Delta x^2} \\
&\,\,+y(y+\Delta y)\frac{f_h[0;0,-1,-2]_*}{2\Delta y^2}
+xy\frac{f_h[0,-1;0,-1]_*}{\Delta x\Delta y}+xy(x+\Delta x)\frac{f_h[0,-1,-2;0,-1]_*}{2\Delta x^2\Delta y}\nonumber\\
&\,\,+xy(y+\Delta y)\frac{f_h[0,-1;0,-1,-2]_*}{2\Delta x\Delta y^2}+xy(x+\Delta x)(y+\Delta y)\frac{f_h[0,-1,-2;0,-1,-2]_*}{4\Delta x^2\Delta y^2}.\nonumber
\end{align}

Whence, we plug these last two expressions in (\ref{beta_2D_2}) and compute directly 
\begin{align}\label{explicit_beta_2D}
\beta_k^{--}
=\frac{1}{\Delta x\Delta y}&\left[ {f{[2,0]}_*}^2+{f{[0,2]_*}}^2+{f{[1,1]_*}}^2+\frac{17}{12}\left({f{[2,1]_*}}^2+{f{[1,2]_*}}^2\right)+\frac{317}{720}{f{[2,2]}_*}^2+\right.\nonumber\\
&\left.+f{[2,0]}_*f{[2,1]_*}
+f{[0,2]}_*f{[1,2]_*}-\frac{1}{6}\left(f{[2,0]}_*f{[2,2]}_*+f{[0,2]}_*f{[2,2]}_*\right)\right. \nonumber\\
&\left.
-\frac{1}{12}\left(f{[2,1]}_*f{[2,2]}_*+f{[1,2]}_*f{[2,2]}_*\right)\right]
\end{align}
where we have used the compact notation $f{[t,s]}_*$ to denote the multivariate undivided difference of $f$ of order $t$ in $x$ and $s$ in $y$. Notice that we avoided to specify the points on which the undivided difference are computed, aiming at writing a unique formulation for all cases. 
Before listing the stencils to be used, it is worth to point out that to obtain this simplification we used the outer stencil in a smart way, writing the Newton form of the polynomial starting from the origin in both directions. More precisely, we have used the \emph{ordered} stencils
$$
\mathcal S_0^{--}=\{x_{j-1},x_j,x_{j+1}\}\times\{y_{i-1},y_i,y_{i+1}\},\quad \mathcal S_1^{--}=\{x_{j},x_{j-1},x_{j-2}\}\times\{y_{i},y_{i-1},y_{i-2}\}.
$$
For all the other smoothness coefficients, $\beta_k^{+ +}$, $\beta_k^{+ -}$, $\beta_k^{- +}$, 
we advise the use of the following \emph{ordered} stencils:
\begin{itemize}
\item $\mathcal S_0^{+-}=\{x_{j+1},x_{j},x_{j-1}\}\times\{y_{i-1},y_i,y_{i+1}\}$, $\mathcal S_1^{+-}=\{x_{j},x_{j+1},x_{j+2}\}\times\{y_{i},y_{i-1},y_{i-2}\}$,
\item $\mathcal S_0^{++}=\{x_{j+1},x_{j},x_{j-1}\}\times\{y_{i+1},y_i,y_{i-1}\}$, $\mathcal S_1^{++}=\{x_{j},x_{j+1},x_{j+2}\}\times\{y_{i},y_{i+1},y_{i+2}\}$,
\item $\mathcal S_0^{-+}=\{x_{j-1},x_{j},x_{j+1}\}\times\{y_{i+1},y_i,y_{i-1}\}$, $\mathcal S_1^{-+}=\{x_{j},x_{j-1},x_{j-2}\}\times\{y_{i},y_{i+1},y_{i+2}\}$.
\end{itemize}
Note that we are changing also the ordering of the reference stencil in each case. Then, if we compute the integrals
\begin{equation}
\label{beta_zeta}
\beta_k^{\zeta_1\zeta_2}=(-1)^{|\zeta|}\sum_{{\alpha_1, \alpha_2=0\atop |\alpha|\geq 2}}^2 \int_{\zeta_1\Delta x}^0\int_{\zeta_2\Delta y}^{0} \Delta x^{2(\alpha_1-1)}\Delta y^{2(\alpha_2-1)} \left(\partial_x^{\alpha_1}\partial_y^{\alpha_2} P_k^{\zeta_1 \zeta_2}(x,y)\right)^2 dx dy,
\end{equation}
where $|\zeta|$ denotes the number of `$-$' in $(\zeta_1,\zeta_2)$, for $\zeta_1, \zeta_2=+,-$, using the previous ordered stencils, we obtain the same formula (\ref{explicit_beta_2D}).
From now on, the construction follows similar steps as in the usual WENO procedure, but with a slightly different aim. 
Firstly, we define
\begin{equation}
\alpha^{--}_k=\frac{1}{(\beta^{--}_k+\sigma_h)^2},
\end{equation}
then we take the information given by the reference polynomial computing
\begin{equation}
\omega^{--}=\frac{\alpha^{--}_0}{\alpha^{--}_0+\alpha^{--}_1},
\end{equation}
which will measure the regularity of the function in the cell $[x_{j-1},x_j]\times[y_{i-1},y_i]$. Once we have computed in the same way the other three indicators, we can finally define
\begin{equation}
\label{omega_2D}
\omega=\min\{\omega^{--},\omega^{+-},\omega^{-+},\omega^{++}\},
\end{equation}
which has properties similar to its one-dimensional counterpart, that is
\begin{equation}
\omega_{i,j}=
\left\{
\begin{array}{ll}
O(\Delta ^4)\qquad &\textrm{ if }(x_s,y_s) \in I_{i,j} \\
\frac{1}{2}+O(\Delta)& \textrm{ otherwise.}
\end{array}
\right.
\end{equation}
Consequently, in order to reduce the oscillations of order $O(\Delta)$ around the optimal value in regions of regularity, we can use one of the construction described in \ref{sec:ind_1D}, that is the mapping (\ref{map2}),
which directly gives $\omega_{i,j}^*=g(\omega_{i,j})=\frac{1}{2}+O(\Delta^3)$, and the `WENO-Z' procedure, which can be directly generalized to the 2D-case. In fact it is sufficient to substitute the superscript `$\pm$' with `$\zeta_1\zeta_2$' in (\ref{tau_z})-(\ref{WENO_Z}).
Finally, what is left is to define a function $\phi$ such that (\ref{eq:phi_2D}) is satisfied, that is $\phi=1$ if $\omega$ is close to $\frac{1}{2}$ and $\phi=0$, otherwise. The simplest choice is to take
\begin{equation}
\label{phi_disc}
\phi(\omega)=\chi_{\{\omega\geq M\}},
\end{equation}
with $M<\frac{1}{2}$, a number possibly dependent on $\Delta x$.
Or we can choose a more regular function
\begin{equation}
\phi(\omega)=\frac{e^{-M\omega}-1}{e^{-M}-1},
\end{equation}
where now $M$ must be big enough to have a quick transition from $1$ to $0$. Of course, in this case $\phi$ will be only an approximation of the values $0$ and $1$ aroud the singularity. \\
It is worth to point out that, according to the described procedure, when using polynomials of degree 2 the stencil of the smoothness indicator function $\phi$ is very compact, requiring only $5\times 5$ points, as shown in Fig. \ref{stencil_cases_B} (right).

\section{Multidimensional Adaptive Filtered Scheme}\label{sec:AFS_2D}
The aim of this section is to present the detailed construction of the multidimensional AF scheme, focusing on the 2D case. 
Information on each component of the scheme will be given, making the implementation of the method rather straightforward, having also provided the explicit formulas for the smoothness indicators in the previous section.  
Let us consider a uniform grid in space $(x_j,y_i)=(j\Delta x,i\Delta y)$, $j, i \in \mathbb Z$, and in time $t_n=t_0+n\Delta t$, $n\in[0,N_T]$, with $(N_T-1)\Delta t< T\leq N_T\Delta t$. A typical feature of a filtered scheme $S^F$  is that at the node $(x_j,y_i)$ it is a mixture of a high-order scheme $S^A$ and a monotone scheme $S^M$ according to a  \emph{filter function} $F$ and a switching parameter $\varepsilon(\Delta t,\Delta x)$. Following this idea, we compute the numerical approximation $u^n_{i,j}=u(t_n,x_j,y_i)$ of the viscosity solution of (\ref{eq:HJ_2}) with the simple formula
\begin{equation}
\label{AFS_2D}
  u^{n+1}_{i,j} \equiv S^{AF}(u^n)_{i,j} := S^{M}(u^n)_{i,j}+\phi^n_{i,j}\varepsilon^n\Delta t F\bigg(\frac{S^{A}(u^n)_{i,j}-S^{M}(u^n)_{i,j}}{\varepsilon^n\Delta t}\bigg), \quad  i, j\in \mathbb{Z},
\end{equation}
where $\varepsilon^n$ is the switching parameter at time $t_n$ and $\phi_{i,j}^n$ is the \emph{smoothness indicator function} at the node $(x_j,y_i)$ and time $t_n$ defined in \eqref{eq:phi_2D} for a general function $f$, now depending on the values of the approximate solution $u^n$.
Filtered schemes are precisely designed in order to be high-order accurate where the solution is smooth and monotone near singularities. This feature allows to naturally prove convergence to the viscosity solution with error estimates, relying on the classical result of Crandall and Lions \cite{CL84}, and to obtain the high-order consistency property with very mild assumptions on the various components, as it has been shown in \cite{FPT18a} for the one-dimensional version of the scheme.
\subsection{Assumptions on the schemes}
Let us describe the assumptions on the basic schemes $S^M$ and $S^A$, which are direct generalizations in two space dimensions of those presented in \cite{FPT18a}. 

\noindent
\textbf{Assumptions on $S^M$.}
\begin{enumerate}[(M1)]
\item The scheme can be written in \emph{differenced form}
\begin{equation}\label{eq:FD_2D}
u^{n+1}_{i,j}\equiv  S^{M}(u^{n})_{i,j} := u^{n}_{i,j} -\Delta t ~ h^M\left(x_j,y_i,D_x^-u^n_{i,j},D_x^+u^n_{i,j},D_y^-u^n_{i,j},D_y^+u^n_{i,j}\right)
\end{equation}
for a function $h^M(x,y,p^-,p^+,q^-,q^+)$, with $D_x^{\pm} u^n_{i,j}:=\pm \frac{u^n_{i,j\pm 1}-u^n_{i,j}}{\Delta x}$ and $D_y^{\pm}u^n_{i,j}:=\pm \frac{u^n_{i\pm 1,j}-u^n_{i,j}}{\Delta y}$;
\item $h^M$ is a Lipschitz continuous function;
\item (Consistency) $\forall u,v$, $\quad h^M(\cdot,\cdot,u,u,v,v)=H(\cdot,\cdot,u,v)$;
\item (Monotonicity) for any functions $u, v$, 
$\qquad u \leq v\quad  \Rightarrow\quad S^M(u) \leq S^M(v)$.
\end{enumerate}

Under assumption (M2), the consistency property (M3) is equivalent to say that for all functions $v\in C^2([0,T]\times\mathbb R)$, there exists a constant $C_M\geq 0$ independent on $\Delta=(\Delta t,\Delta x, \Delta y)$ such that
\begin{equation}
\label{consistenza_M_2D}
\mathcal E_{M}(v)(t,x,y):=\left|\frac{v(t+\Delta t,x,y)-S^M(v(t,\cdot,\cdot))(x,y)}{\Delta t}\right|\leq C_M\left(\Delta t ||v_{tt}||_\infty +\Delta x||v_{xx}||_\infty+\Delta y||v_{yy}||_\infty\right),
\end{equation}
where $\mathcal E_M$ is the consistency error. The last relation highlights the well-known first order bound on the accuracy of the monotone schemes for regular solutions.
\begin{remark}
It is worth to notice that, under the Lipschitz assumption (M2), it can be shown that the monotonicity property (M4) is equivalent to require, for a.e. $(p^-,p^+)\in \mathbb R^2$,
\begin{equation}
\frac{\partial h^M}{\partial p^-}\geq 0,\qquad \frac{\partial h^M}{\partial p^+}\leq 0,\qquad\frac{\partial h^M}{\partial q^-}\geq 0,\qquad \frac{\partial h^M}{\partial q^+}\leq 0,
\end{equation}
and the \emph{Courant-Friedrichs-Lewy (CFL) condition}
\begin{equation}
\label{cond_CFL_2D}
\frac{\Delta t}{\Delta x}\left(\frac{\partial h^M}{\partial p^-}-\frac{\partial h^M}{\partial p^+}\right)+\frac{\Delta t}{\Delta y}\left(\frac{\partial h^M}{\partial q^-}-\frac{\partial h^M}{\partial q^+}\right)\leq 1.
\end{equation}
We define the constant ratios $\lambda_x:=\frac{\Delta t}{\Delta x}$, $\lambda_y:=\frac{\Delta t}{\Delta y}$, such that (\ref{cond_CFL_2D}) is satisfied and we call the \emph{CFL number} as the maximum $\lambda=\max\{\lambda_x,\lambda_y\}$.
\end{remark}
\begin{example}
In this example we recall some monotone schemes in differenced form satisfying (M1)-(M4), which will be used in the numerical tests. Further examples can be found in \cite{PaolucciPhD}.
\begin{itemize}
\item For the \emph{eikonal equation},
\begin{equation}
v_t+\sqrt{v_x^2+v_y^2}=0,
\end{equation}
we can use the simple numerical hamiltonian 
\begin{equation}
\label{hm_eik_2D}
h^M(p^-,p^+,q^-,q^+):=\sqrt{{\max\{p^-,-p^+,0\}}^2+{\max\{q^-,-q^+,0\}}^2}.
\end{equation}
\item For general equations depending also on the space variables, instead, we can use the 2D-version of the \emph{Local Lax-Friedrichs hamiltonian}
\begin{align}
\label{local_lax_fried_2D}
h^M(x,y,p^-,p^+,q^-,q^+)&:=H\left(x,y,\frac{p^-+p^+}{2},\frac{q^-+q^+}{2}\right)\\
&\quad-\frac{\alpha_x(p^-,p^+)}{2}(p^+-p^-)-\frac{\alpha_y(q^-,q^+)}{2}(q^+-q^-), \nonumber
\end{align}
with
\begin{equation}
\label{loc_vel_lf}
\alpha_x(p^-,p^+)=\max_{p\in I(p^-,p^+)}\left|H_p(x,y,p,q)\right|,\qquad \alpha_y(q^-,q^+)=\max_{q\in I(q^-,q^+)}\left|H_q(x,y,p,q)\right|,
\end{equation}
where the maximum are computed uniformly in $q$ and $p$, respectively, and $I(a,b)$ represents the interval with endpoints $a$ and $b$. 
 The scheme is monotone under the restriction $\lambda_x \max|H_p| +\lambda_y \max|H_q|\leq 1$.
\end{itemize}
\end{example}

\noindent
\textbf{Assumptions on $S^A$.}
\begin{enumerate}[(\textrm{A}1)]
\item The scheme can be written in \emph{differenced form}
\begin{align}
\label{eq:HA_2D}
u^{n+1}_j=S^{A}(u^{n})_j :=u^{n}_j-\Delta t  h^{A}&\left(\{x\}_j,\{y\}_i,D_x^{k,-}u_{i,j},\dots, D_x^-u^n_{i,j},D_x^+u^n_{i,j},\dots,D_x^{k,+}u^n_{i,j},\right.\nonumber\\
 &\left.\qquad D_y^{k,-}u_{i,j},\dots, D_y^-u^n_{i,j},D_y^+u^n_{i,j},\dots,D_y^{k,+}u^n_{i,j}\right),
\end{align}
for some function $h^A(x_j,y_i,p^-,p^+,q^-,q^+)$ (in short), with $D^{k,\pm}_x u^n_{i,j}:=\pm \frac{u^n_{i,j\pm k}-u^n_{i,j}}{k\Delta x}$ and $D^{k,\pm}_y u^n_{i,j}:=\pm \frac{u^n_{i\pm k,j}-u^n_{i,j}}{k\Delta y}$, where with $\{x\}_j$, $\{y\}_i$, we denoted a stencil of points around the node $(x_j,y_i)$;
\item $h^A$ is a Lipschitz continuous function.
\item (High-order consistency) Fix $k\geq 2$ order of the scheme (for all the variables), then for all $l=1,\dots,k$ and for all functions $v\in C^{l+1}$, there exists a constant $C_{A,l}\geq 0$ such that
\begin{align}
\mathcal E_{A}(v)(t,x,y)&:=\left|\frac{v(t+\Delta t,x,y)-S^A(v(t,\cdot))(x,y)}{\Delta t}\right| \nonumber\\
&\leq C_{A,l}\left(\Delta t^l ||\partial^{l+1}_t v||_\infty +\Delta x^l||\partial^{l+1}_x v||_\infty +\Delta y^l||\partial^{l+1}_y v||_\infty\right).
\end{align}
\end{enumerate}
In order to restate the consistency property in a more useful form, let us compute the Taylor expansion
\begin{equation}
\label{exp_v_lw_2D}
v(t+\Delta t,x,y)=v(t,x,y)+\Delta t v_t(t,x,y)+\frac{\Delta t^2}{2} v_{tt}(t,x,y)+O\left(\Delta t^{3}\right),
\end{equation}
which gives, dropping the dependence on $(x,y)$ for brevity,
\begin{align}
\label{lax_wen_2D}
v_{tt}&=\frac{\partial}{\partial t}(-H(x,y,v_x,v_y))=-H_p(x,y,v_x,v_y)v_{xt}-H_q(x,y,v_x,v_y)v_{yt}\nonumber\\
&=H_p(x,y,v_x,v_y)\frac{\partial}{\partial x}\left(H(x,y,v_x,v_y)\right)+H_q(x,y,v_x,v_y)\frac{\partial}{\partial y}\left(H(x,y,v_x,v_y)\right)\nonumber\\
&=H_p\left(H_p v_{xx}+H_x \right)+H_q\left(H_q v_{yy}+H_y\right)+2H_p H_q v_{xy},
\end{align}
where in the last equality we dropped the functional dependence for brevity. 
Then, assuming (A1)-(A2), it is straightforward to write the consistency property in terms of the numerical hamiltonian $h^A$, that is
\begin{enumerate}[(\textrm{A}3$^\prime$)]
\item (High-order consistency) Fix $k\geq 2$ order of the scheme (for all the variables), then for all $l=1,\dots,k$ and for all functions $v\in C^{l+1}$, there exists a constant $C_{A,l}\geq 0$ such that
\begin{align}
\mathcal E_{A}(v)&(t,x,y):=\left|h^A(x,y,D_x^-v,D_x^+v,D_y^-v,D_y^+v)-H(x,y,v_x,v_y)\right.\nonumber\\
&\left.+\frac{\Delta t}{2}\left[H_p(x,y,v_x,v_y)\frac{\partial}{\partial x}\left(H(x,y,v_x,v_y)\right)+H_q(x,y,v_x,v_y)\frac{\partial}{\partial y}\left(H(x,y,v_x,v_y)\right)\right]\right|\nonumber\\
&\leq C_{A,l}\left(\Delta t^l ||\partial^{l+1}_t v||_\infty +\Delta x^l||\partial^{l+1}_x v||_\infty +\Delta y^l||\partial^{l+1}_y v||_\infty\right).
\label{cond_nec_ha_2D}
\end{align}
\end{enumerate}
In the following examples we present some simple high-order schemes satisfying (A1)-(A3) with $l=2$, dropping the dependence on $(i,j)$ (and also on $n$) in order to avoid cumbersome notations.
\begin{example}\label{ex:HC_2D}
The easiest way is to consider a \emph{second order in space numerical hamiltonian} $h^A_*$
\begin{equation}\label{cond_ha_space_2D}
h_*^A(x,y,D_x^-v,D_x^+ v,D_y^-v,D_y^+v)=H(x,y,v_x,v_y)+O\left(\Delta x^2\right)+O\left(\Delta y^2\right),
\end{equation}
such as the simple second order \emph{Centered approximation}
\begin{equation}\label{HC_2D}
h_*^A(x,y,D_x^-u,D_x^+u,D_y^-u,D_y^+u)=H\left(x,y,\frac{D_x^-u+D_x^+u}{2},\frac{D_y^-u+D_y^+u}{2}\right),
\end{equation}
and combine it with the second order \emph{Heun method},
\begin{equation}\label{RK2}
\left\{
\begin{array}{l}
u^*=u^n-\Delta t h_*^A(x,y,D_x^{-}u^n,D_x^{+}u^n,D_y^{-}u^n,D_y^{+}u^n)\\
u^{n+1}=\frac{1}{2}u^n+\frac{1}{2}u^*-\frac{\Delta t}{2}h_*^A(x,y,D_x^{-}u^*,D_x^{+}u^*,D_y^{-}u^*,D_y^{+}u^*).
\end{array}
\right.
\end{equation}
\end{example}
\begin{example}
In this example we propose a series of numerical hamiltonians $h^A$ obtained discretizing directly the formula (\ref{lax_wen_2D}). The first is the most direct and simple discretization, named \emph{Lax-Wendroff (LW)  scheme}
\begin{align}\label{LW_2D}
h^A(x,y,D_x^\pm u,D_y^\pm u)&=H(x,y,D_x u,D_y u) \nonumber\\
&-\frac{\Delta t}{2}\left[H_p(x,y,D_x u,D_y u)\left( H_p(x,y,D_x u,D_y u)D^2_{x}u+H_x(x,y,D_x u,D_y u)\right) \right.\nonumber\\
&+H_q(x,y,D_x u,D_y u)\left(H_q(x,y,D_x u,D_y u) D^2_{y} u+H_y(x,y,D_x u,D_y u)\right)\nonumber\\
&\left.+2H_p(x,y,D_x u,D_y u)H_q(x,y,D_x u,D_y u) D^2_{xy}u\right],
\end{align}
where $D^\pm_x u$, $D_x u$, $D^2_{x} u$, $D^\pm_x u$, $D_x u$, $D^2_{x} u$  
are, respectively, the usual one-sided and centered one-dimensional finite difference approximations of the first and second derivative in the $x$ and $y$ direction, whereas for the mixed derivative we use 
\begin{equation}
D^2_{xy} u_{i,j}=\frac{u_{i+1,j+1}-u_{i-1,j+1}-u_{i+1,j-1}+u_{i-1,j-1}}{4\Delta x\Delta y}.
\end{equation}
Notice that the derivatives of $H$ can be computed either analytically or by some second order numerical approximation. In particular, to compute the derivative $H_x$, we can simply use
\begin{equation}
(H_x)_{i,j}=\frac{H(x_{j+1},y_i,D_x u_{i,j},D_y u_{i,j})-H(x_{j-1},y_i,D_x u_{i,j},D_y u_{i,j})}{2\Delta x},
\end{equation}
and analogously for $H_y$. This hamiltonian has been used in \cite{FPT18b}, where the authors propose a new high-order accurate method for image segmentation.
 
Another possibility, which is more closely related to the one-dimensional Lax-Wendroff scheme, is the following Lax-Wendroff (LW2) scheme
\begin{align}\label{LW_2D_2}
h^A(x,y,D_x^\pm u,D_y^\pm u)=&H(x,y,D_x u,D_y u)-\frac{\Delta t}{2}H_p(x,y,D_x u,D_y u)\left(H_x^*+H_x(x,y,D_x u,D_y u)\right)\nonumber\\
&\qquad \qquad-\frac{\Delta t}{2}H_q(D_x u,D_y u)\left(H^*_y+H_y(x,y,D_x u,D_y u)\right),
\end{align}
 where we have defined
\begin{align}\label{Hx_star}
H_x^*=&\frac{1}{\Delta x}\left[H\left(x_j,y_i,\frac{u_{i,j+1}-u_{i,j}}{\Delta x},\frac{u_{i+1,j+1}-u_{i-1,j+1}+u_{i+1,j}-u_{i-1,j}}{4\Delta y}\right)\right.\nonumber \\
&\left.\qquad \qquad\qquad-H\left(x_j,y_i,\frac{u_{i,j}-u_{i,j-1}}{\Delta x},\frac{u_{i-1,j}-u_{i-1,j}+u_{i+1,j-1}-u_{i-1,j-1}}{4\Delta y}\right)\right]
\end{align}
and
\begin{align}\label{Hy_star}
H_y^*=&\frac{1}{\Delta y}\left[H\left(x_j,y_i,\frac{u_{i+1,j+1}-u_{i+1,j-1}+u_{i,j+1}-u_{i,j-1}}{4\Delta x},\frac{u_{i+1,j}-u_{i,j}}{\Delta y}\right)\right.\nonumber\\
&\left.\qquad \qquad\qquad-H\left(x_j,y_i,\frac{u_{i,j+1}-u_{i,j-1}+u_{i-1,j+1}-u_{i-1,j-1}}{4\Delta x},\frac{u_{i-1,j}-u_{i,j}}{\Delta y}\right)\right].
\end{align}
This can be seen as a discretization of the third relation in (\ref{lax_wen_2D}).
\item The last example we propose, which can be used in the case the hamiltonian does not depend on the space variables, is the \emph{Richtmyer} form,
\begin{equation}
h^A(D_x^\pm u,D_y^\pm u)=H\left(D_x u- \frac{\Delta t}{2}H^*_x,D_y u-\frac{\Delta t}{2}H^*_y\right),
\end{equation}
where $H_x^*$ and $H_y^*$ are computed as in (\ref{Hx_star})-(\ref{Hy_star}), without the dependence on $(x_j,y_i)$. In particular, it is worth to notice that this last scheme does not require any computation of $H_p$ or $H_q$.
\end{example}
\begin{example}
Finally, we would like to show a simple way to define a scheme satisfying (A1)-(A3) with $l=4$, reminding that, in our approach, the high-order scheme has no need to be stable, in any sense. Then, generalizing the construction of Example \ref{ex:HC_2D}, we can define a fourth-order scheme by combining the simple fourth-order central approximation
\begin{equation}\label{C4ord}
h_*^A\left(x_j,y_i,D_x^-u,D_x^+u,D_y^-u,D_y^+u\right)=H\left(x_j,y_i,D^*_x u,D^*_y u\right),
\end{equation}
where the approximated partial derivative are computed as
\begin{equation}
D^*_x u=\frac{u_{i,j-2}-8u_{i,j-1}+8u_{i,j+1}-u_{i,j+2}}{12\Delta x},\quad\quad D^*_y u=\frac{u_{i-2,j}-8u_{i-1,j}+8u_{i+1,j}-u_{i+2,j}}{12\Delta y},
\end{equation}
with the classical \emph{fourth-order Runke-Kutta scheme}
\begin{equation}\label{RK4}
\begin{cases}
\ u^{(1)}=u^n-\frac{\Delta t}{2} h^A_*\left(D^\pm u^n\right)\\
\ u^{(2)}=u^n-\frac{\Delta t}{2} h^A_*\left(D^\pm u^{(1)}\right)\\
\ u^{(3)}=u^n-\Delta t h^A_*\left(D^\pm u^{(2)}\right)\\
u^{n+1}=u^n-\frac{\Delta t}{6}\left[ h^A_*\left(D^\pm u^{n}\right)+2h^A_*\left(D^\pm u^{(1)}\right) +2h^A_*\left(D^\pm u^{(2)}\right) +h^A_*\left(D^\pm u^{(3)}\right) \right]\\
\qquad=\frac{1}{3}\left[ 2u^{(2)}+u^{(3)}-\frac{\Delta t}{2} \left(h^A_*\left(D^\pm u^n\right)+h^A_*\left(D^\pm u^{(3)}\right)\right)\right].
\end{cases}
\end{equation}
Note that, differently from the approach usually used when working with Hamilton-Jacobi equations (e.g. WENO schemes of higher order), we do not require the time discretization to be \emph{Total Variation Diminishing (TVD)}, thus we can use the more efficient formula (\ref{RK4}), which is also easier to implement with respect to the TVD version (see \cite{OS91} for more details). It is worth to point out that, due to the Runke-Kutta procedure, the stencil of the scheme is not very compact, requiring $17\times 17$ points in total. 
\end{example}
\subsection{Filter function}
In order to couple the schemes and their properties, we need to define a \emph{filter function F}, such that
\begin{enumerate}[(\textrm{F}1)]
\item $F(\rho) \approx \rho$ for $|\rho|\leq 1$,
\item $F(\rho) =0$ for $|\rho|>1$,
\end{enumerate}
which implies that
\begin{itemize}
\item If $| S^A-S^M|\leq \Delta t \varepsilon^n$ and $\phi_{i,j}^n=1\Rightarrow S^{AF}\approx S^A$
\item If $| S^A-S^M|> \Delta t \varepsilon^n$ or $\phi^n_{i,j}=0 \Rightarrow S^{AF}= S^M$.
\end{itemize}
It is clear that several possible definitions for $F$ satisfy these two requirements. 
In \cite{FPT18a} the authors presented some examples of filter functions satisfying the previous relations with different regularity properties. 
In this work, also for comparison reasons, we use 
the filter function defined in \cite{BFS16} as
\begin{equation}
F(\rho)=\left\{
\begin{array}{ll}
\rho\qquad&\textrm{ if } |\rho|\leq 1 \\
0&\textrm{ otherwise,}
\end{array}
\right.
\end{equation}
which is clearly discontinuous at $\rho=-1,1$ and satisfies trivially the properties (F1)-(F2). 
\subsection{ Tuning of $\varepsilon^n$}
Finally, the last step is to compute the switching parameter $\varepsilon^n$. If we want the scheme (\ref{AFS_2D}) to switch to the high-order scheme when some regularity is detected, we have to choose $\varepsilon^n$ such that
\begin{equation}
\label{ep_ineq_2D}
\left|\frac{S^A(v^n)_{i,j}-S^M(v^n)_{i,j}}{\varepsilon^n \Delta t}\right|\!=\!\left|\frac{h^A\left(D_x^\pm v^n,D_y^\pm v^n\right)_{i,j}\!-\!h^M\left(D_x^\pm v^n,D_y^\pm v^n\right)_{i,j}}{\varepsilon^n}\right| \leq 1, \textrm{ for }(\Delta t,\Delta x,\Delta y)\to 0,
\end{equation}
in the \emph{region of regularity at time} $t_n$, that is $\mathcal R^n=\left\{(x_j,y_i) : \phi^n_{i,j}=1\right\}$, where, for $I_{i,j}=[x_{j-1},x_{j+1}]\times[y_{i-1},y_{i+1}]$, 
\begin{equation}
\label{phi_2D}
\phi^n_{i,j}=\left\{
\begin{array}{ll}
1\qquad&\textrm{ if the solution } u^n\textrm{ is regular in }I_{i,j},  \\
0&\textrm{ if }I_{i,j} \textrm{ contains a point of singularity.}
\end{array}
\right.
\end{equation}

In order to estimate the distance between the numerical hamiltonians, we first proceed by Taylor expansions for the monotone scheme, obtaining
\begin{align}
h^M\left(x,y,D_x^\pm v^n,D_y^\pm v^n\right)=&H\left(x,y,v^n_x,v^n_y\right)+\frac{\Delta x}{2}v^n_{xx}\left(\partial_{p_+}h^M_{i,j}-\partial_{p_-}h^M_{i,j}\right)\\
&+\frac{\Delta y}{2}v^n_{yy}\left(\partial_{q_+}h^M_{i,j}-\partial_{q_-}h^M_{i,j}\right)+O\left(\Delta x^2\right)+O\left(\Delta y^2\right),\nonumber
\end{align}
whereas for the high-order scheme, by using the consistency property,
\begin{align}
h^A\left(x,y,D_x^\pm v^n_{i,j},D_y^\pm v^n_{i,j}\right)\!=&H\left(x,y,v^n_x,v^n_y\right)-\frac{\Delta t}{2}\left[H_p^2\left(x,y,v^n_x,v^n_y\right)v^n_{xx}+H_q^2\left(x,y,v^n_x,v^n_y\right)v^n_{yy}\right.\nonumber \\
&\!\!\!\left.+2H_p\left(x,y,v^n_x,v^n_y\right)H_q\left(x,y,v^n_x,v^n_y\right)v_{xy}\right]\!+\!O\left(\Delta t^2\right)+O\left(\Delta x^2\right)+O\left(\Delta y^2\right).
\end{align}
Whence, from (\ref{ep_ineq_2D}) we obtain
\begin{align}
\label{epn_ineq_2D}
\varepsilon^n\geq& \left|\frac{\Delta x}{2}v^n_{xx}\left[\partial_{p_+}h^M\!-\!\partial_{p_-}h^M\!+\!\lambda_xH^2_p\left(x,y,v^n_x,v^n_y\right)\right]\!+\!
\frac{\Delta y}{2}v^n_{yy}\left[\partial_{q_+}h^M-\partial_{q_-}h^M\!+\!\lambda_yH^2_q\left(x,y,v^n_x,v^n_y\right)\right]\right.\nonumber\\
&\left.+\Delta t v^n_{xy}H_p\left(x,y,v^n_x,v^n_y\right)H_q\left(x,y,v^n_x,v^n_y\right)+O\left(\Delta t^2\right)+O\left(\Delta x^2\right)+O\left(\Delta y^2\right)\right|,
\end{align}
that has to be satisfied in $\mathcal R^n$.
Then, we use a numerical approximation of the lower bound on the right hand side of the previous inequality to obtain the formula for $\varepsilon^n$. In order to devise a simple formula, we introduce the notation
\begin{equation}
\widetilde h^M_{p^+}=h^M\left(x,y,D_x u^n,D^+_x u^n, D_y u^n,D_y u^n\right)-h^M\left(x,y,D_x u^n,D^-_x u^n, D_y u^n,D_y u^n\right),
\end{equation}
with the other cases following analogously. Finally, the simplest discretization, which we use in the numerical examples, is
\begin{align}
\label{eps_2D}
\varepsilon^n=\max_{(x_j,y_i)\in\mathcal R^n}K&\left|\frac{\Delta t}{2}\left[H_p\left(H_x+H_p D^2_x u^n\right)+H_q\left(H_y+H_q D^2_y u^n\right)+2H_p H_q v_{xy})\right]  \right. \nonumber\\
&+\left.\left(\widetilde h^M_{p^+}-\widetilde h^M_{p^-}\right)+\left(\widetilde h^M_{q^+}-\widetilde h^M_{q^-}\right)\right|,
\end{align}
where all the derivatives of $H$ are computed at $(x,y,D_x u^n,D_y u^n)$ and the finite difference approximations around the point $(i,j)$, while $K>\frac{1}{2}$.  
Another possibility, which does not require the computation of the derivatives of $H$ and it is valid in the case the hamiltonian does not depend on the space variables, is the following
\begin{align}
\varepsilon^n=\max_{(x_j,y_i)\in\mathcal R^n}K&\left|H\left(D_x u^n,D_y u^n\right)-H\left(D_x u^n- \frac{\lambda_x}{2}H^*_x,D_y u^n-\frac{\lambda_y}{2}H^*_y\right) \right. \nonumber\\
&+\left.\left(\widetilde h^M_{p^+}-\widetilde h^M_{p^-}\right)+\left(\widetilde h^M_{q^+}-\widetilde h^M_{q^-}\right)\right|,
\end{align}
 where $H_x^*$ and $H_y^*$ are defined by (\ref{Hx_star})-(\ref{Hy_star}).

\section{Numerical tests}\label{sec:tests}

In this section we want to validate, through several interesting numerical tests, the proposed new smoothness indicators defined in Sect. \ref{sec:ind_2D}, showing also their application in the construction of AF schemes introduced in Sect. \ref{sec:AFS_2D}. 
We divided the tests into two parts. In the first part, we focus on the smoothness indicators in one and two space dimensions, considering in both cases singularities that fall on a grid point or inside some cell. 
In the second part, we solve some well-known evolution problems by using AF schemes  composed with different choices of the high-order scheme. Moreover, we compare the results with other state-of-the-art schemes, as the basic filtered scheme \cite{BFS16}, an optimized version of the WENO scheme of second/third order \cite{JP00}, and of the WENO 3/5 with RK3 (TVD) and RK4 (not TVD) as time integrator. 
For each numerical test, we will specify the parameters and the numerical schemes involved in the computations. 
All the numerical tests have been implemented in language C++, with plots 
in MATLAB. The computer used for the simulations is a Notebook Asus F556U Intel Core i7-6500U with speed of 2.59 GHz and 12 GB of RAM. 

\subsection{Tests on the smoothness indicators}\label{sec:tests_ind}

Before illustrating the numerical tests in 2D, we start this section with a 1D test in order to show the good properties of the one-dimensional indicators, stated by Prop. \ref{prop_beta}. For a more extensive presentation of numerical tests, we refer the reader to \cite{PaolucciPhD}. 
For all the tests reported in this subsection related to the smoothness indicators, we will consider the discontinuous function $\phi$ defined in \eqref{phi_disc} with $M=0.2$ and we will use periodic Boundary Conditions (BCs) to implement the indicators. \\
For the two dimensional tests, we compare the results obtained by using our smoothness coefficients 
\eqref{explicit_beta_2D} with two other possible constructions. 
The first comes from  \eqref{beta_2D_2} considering the restriction $|\alpha|=2$ in the summation, arriving to define
\begin{align}\label{explicit_beta_2D_old}
\beta^P_{k,w}&:= \sum_{|\alpha|=2} \int_{x_{j-1}}^{x_j}\int_{y_{i-1}}^{y_i} \Delta x^{2(\alpha_1-1)}\Delta y^{2(\alpha_2-1)} \left( \partial_x^{\alpha_1}\partial_y^{\alpha_2}P_{k,w}(x,y)\right)^2 dx dy \\
&=\frac{1}{\Delta x\Delta y}\left[ {f{[2,0]}_*}^2+{f{[0,2]_*}}^2+{f{[1,1]_*}}^2+\frac{5}{12}\left({f{[2,1]_*}}^2+{f{[1,2]_*}}^2\right)+\frac{17}{720}{f{[2,2]}_*}^2 \right.\nonumber\\
&\left.+f{[2,0]}_*f{[2,1]_*}
+f{[0,2]}_*f{[1,2]_*}-\frac{1}{6}\left(f{[2,0]}_*f{[2,2]}_*+f{[0,2]}_*f{[2,2]}_*\right) \right. \nonumber\\
&\left.-\frac{1}{12}\left(f{[2,1]}_*f{[2,2]}_*+f{[1,2]}_*f{[2,2]}_*\right)\right].\nonumber
\end{align}
Due to the origin, we used the apex 'P', which stands for \emph{Partial}. 
The second construction considered for comparison is the simplest and direct $2D$-extension of the 1D smoothness indicators, which can be obtained by dimensional splitting, that is,
\begin{equation}\label{omega_split}
\omega_{split}=\min\{\omega_x,\omega_y\},
\end{equation}
where $\omega_x$ and $\omega_y$ are the 1D smoothness indicators in $x$ and $y$ direction, respectively, computed as the minimum between $\omega_-$ and $\omega_+$ defined in \eqref{omegaPM}, fixing each time the other variable. More precisely, in the numerical simulations we use the construction with $r=2$ presented in  \ref{sec:ind_1D}, adding the mapping \eqref{map2} to reduce the oscillations in regular regions, as it is done for both 2D indicators. 
Note that for the ``genuine'' 2D indicator which uses the smoothness coefficients  \eqref{explicit_beta_2D} we use polynomials in $\mathbb Q_2(\mathbb R^2)$, whereas for the \emph{splitting indicator} only couples of polynomials in $\mathbb P_2(\mathbb R)$. 
The latter approach, although very simple and fast, has clearly some drawbacks in terms of reliability with respect to the "full" 2D indicator that uses  \eqref{explicit_beta_2D}, especially because of the oscillations around the optimal value in regions of regularity and the problems in localizing singularities which do not fall on grid points, as will be exposed more clearly in the numerical tests. 
In order to distinguish these slight different 2D indicators, we will denote by $\omega^F_{2D}$, and $\omega^P_{2D}$, the indicators that use the smoothness coefficients \eqref{explicit_beta_2D} or \eqref{explicit_beta_2D_old}, respectively. 

\vspace{0.2cm}\noindent
{\bf Test 1.} Let us consider a 1D case where the function, visible in Fig. \ref{fig_f1D}, is defined as  
\begin{equation}\label{ind:ex1}
f(x)=\left\{
\begin{array}{ll}
{\min\left\{(1-x)^2,(1+x)^2\right\}}^2\qquad &\textrm{ if } -1\leq x \leq 1, \\
\sin\left(\frac{\pi}{2} (x-3)\right)& \textrm{ if \ } 2\leq x\leq 4, \\
0 & \textrm{ otherwise},
\end{array}
\right.
\end{equation}
with $\Omega=[-1.5,4.5]$. 
The first part of the function is clearly in $C^2\left([-1,1]\setminus\{0\}\right)$, since it has a discontinuity in the first derivative located at $x=0$, whereas the second is in $C^\infty\left((2,4)\right)$ and it is discontinuous at $x=2$ and $x=4$. We perform the test in two cases: when the singularity falls on a grid point and when it falls inside a cell. In both cases we compare the mapped indicators $g(\omega)$ defined in \eqref{map2} and the new WENO-Z indicators $\omega^Z_{new}$ defined in \eqref{WENO_Z} using $\tau_{new}$ reported in  \eqref{tau_new}, both with $r=2$, for $\phi$ defined by (\ref{phi_disc}) with $M=0.2$. 
\begin{figure}[h!]
\centering
\includegraphics[scale=0.35]{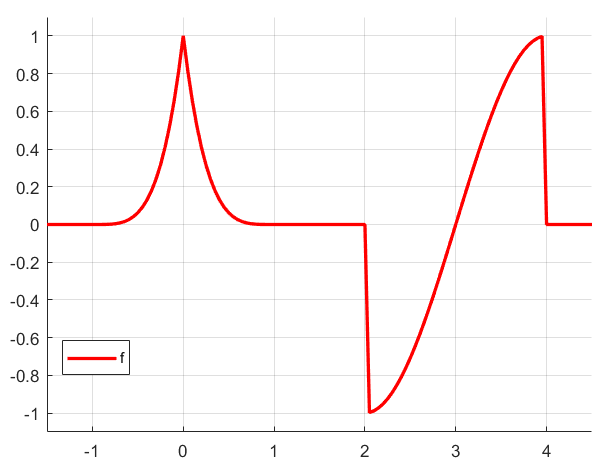}
\caption{\footnotesize{Test 1 in 1D. Piecewise regular function defined in \eqref{ind:ex1} with a singularity in $x=0$ and discontinuities in $x=2$ and $x=4$.}\label{fig_f1D}}
\end{figure}

In Figs. \ref{fig_test1D_1}-\ref{fig_test1D_2} we can see the results obtained by the two indicators in the case the singularity falls on a grid point. It is clear that both indicators perform very well, being able to detect the points of singularity, the discontinuities and the regular regions. We can note that, as better explained in  \ref{sec:ind_1D}, $\omega^Z_{new}$ presents smaller oscillations around the optimal value $\frac{1}{2}$ in regular regions. These oscillations disappear rapidly as the grid is refined. It is good to point out that the stencil of the two indicators is exactly the same, but in the case of $\omega^Z_{new}$ we make a clever use of the full stencil (composed by 5 points) in order to increase the accuracy, reducing also the computational cost with respect to the mappings $g(\omega)$. 
Moreover, we notice that around discontinuities the region detected by the indicator is amplified. This is natural since a jump discontinuity corresponds to an interval where the linear reconstruction has a very high derivative, as can be seen looking at Fig. \ref{fig_f1D}. 

In Figs. \ref{fig_test1D_3}-\ref{fig_test1D_4} we repeat the test shifting the computational domain in order to have the singularity fall inside a cell, precisely on the left with respect to the center of the cell to break the underlying symmetry of the setting. 
In this case we can observe that the mapped indicators have some problems in detecting the right endpoint of the cell for $\Delta x=0.1$, problems which are solved as soon as the grid is refined. On the other hand, the indicators $\omega^Z_{new}$ are able to correctly detect the singular cell and have again smaller oscillations in smooth regions. 

\begin{figure}[h!]
\centering
\begin{subfigure}{0.24\textwidth}
{\includegraphics[width=\textwidth]{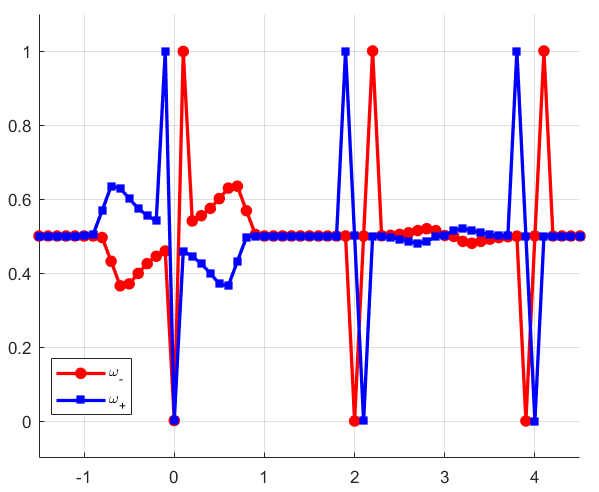}}
\end{subfigure}
\begin{subfigure}{0.247\textwidth}
{\includegraphics[width=\textwidth]{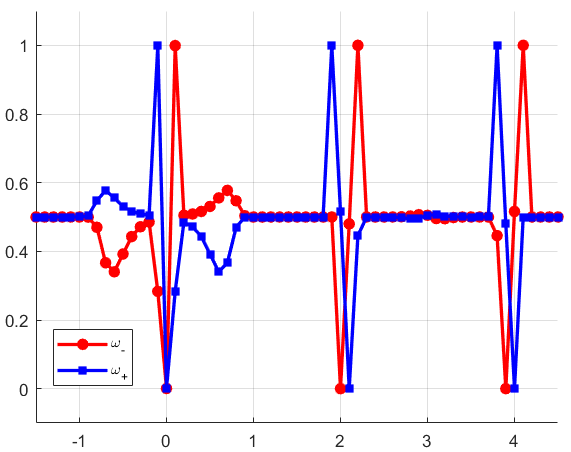}}
\end{subfigure}
\begin{subfigure}{0.24\textwidth}
{\includegraphics[width=\textwidth]{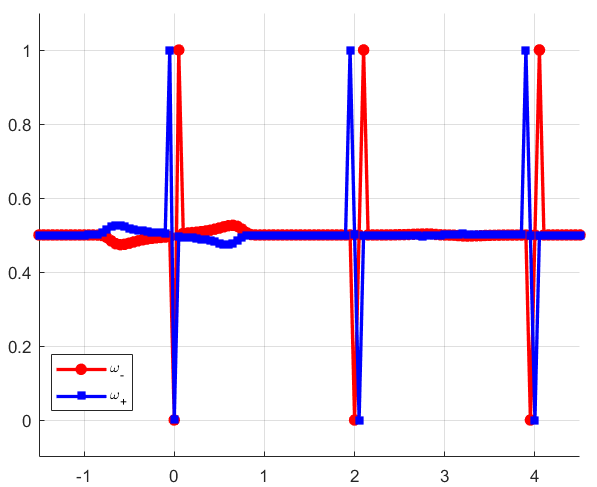}}
\end{subfigure}
\begin{subfigure}{0.245\textwidth}
{\includegraphics[width=\textwidth]{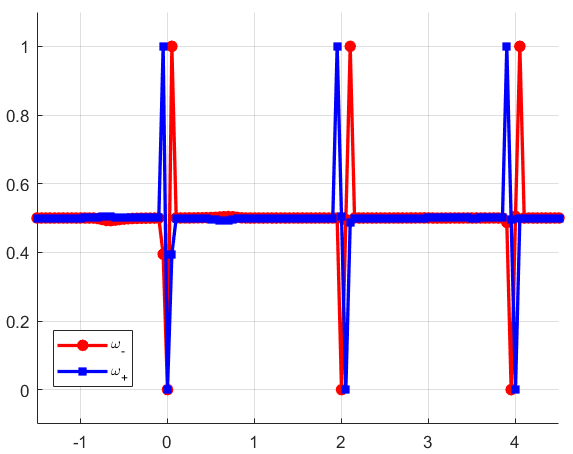}}
\end{subfigure}
\vspace{-0.2cm}
\caption{\footnotesize{Test 1. Singularity on grid points. Results obtained using $g(\omega)$ (left) and $\omega_{new}^Z$ (right) both with $r=2$, for $\Delta x=0.1$ (first two columns) and $\Delta x=0.05$ (last two columns).}\label{fig_test1D_1}}
\end{figure}
\vspace{-0.3cm}
\begin{figure}[h!]
\centering
\begin{subfigure}{0.24\textwidth}
{\includegraphics[width=\textwidth]{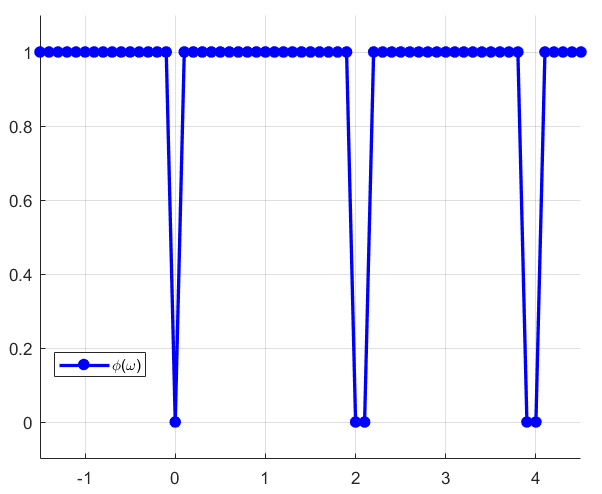}}
\end{subfigure}
\begin{subfigure}{0.247\textwidth}
{\includegraphics[width=\textwidth]{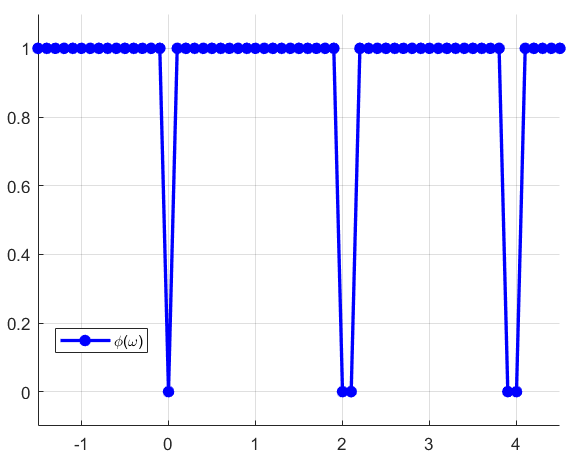}}
\end{subfigure}
\begin{subfigure}{0.24\textwidth}
{\includegraphics[width=\textwidth]{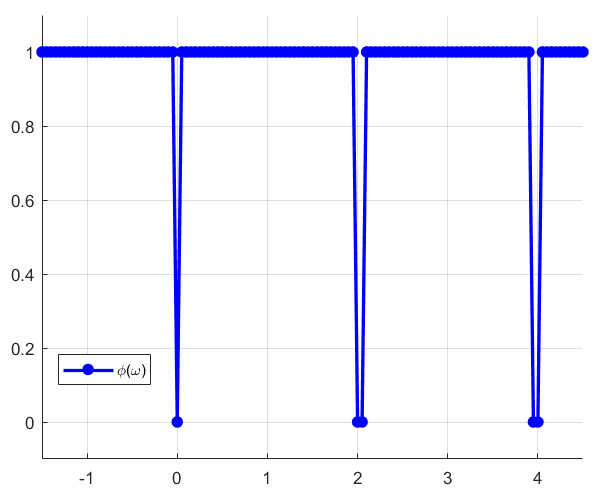}}
\end{subfigure}
\begin{subfigure}{0.245\textwidth}
{\includegraphics[width=\textwidth]{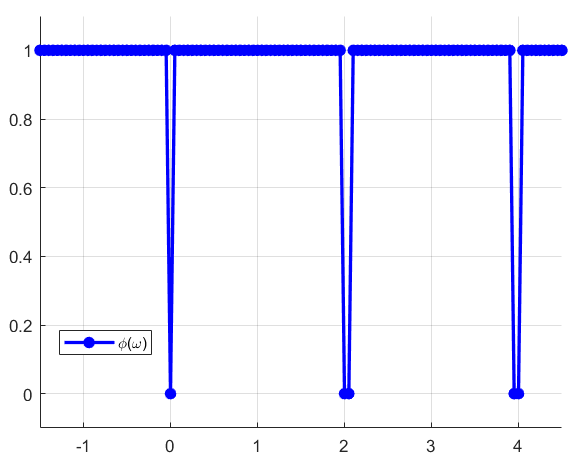}}
\end{subfigure}
\vspace{-0.2cm}
\caption{\footnotesize{Test 1. Singularity on grid points. Results obtained using $\phi(g(\omega))$ (left) and $\phi_{new}^Z$ (right) both with $r=2$, for $\Delta x=0.1$(first two columns) and $\Delta x=0.05$ (last two columns).}\label{fig_test1D_2}}
\end{figure}
\vspace{-0.3cm}
\begin{figure}[h!]
\centering
\begin{subfigure}{0.24\textwidth}
{\includegraphics[width=\textwidth]{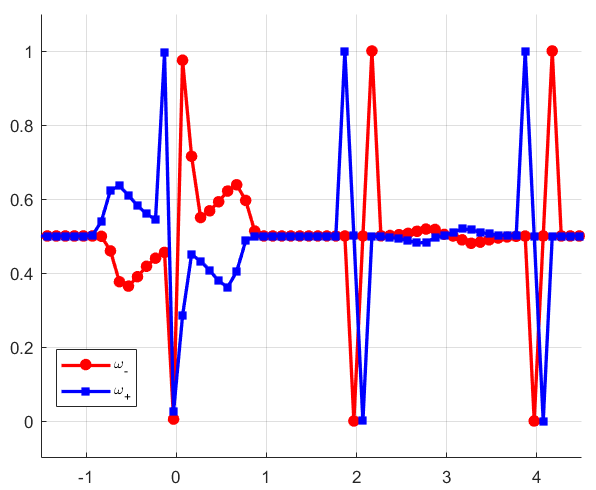}}
\end{subfigure}
\begin{subfigure}{0.247\textwidth}
{\includegraphics[width=\textwidth]{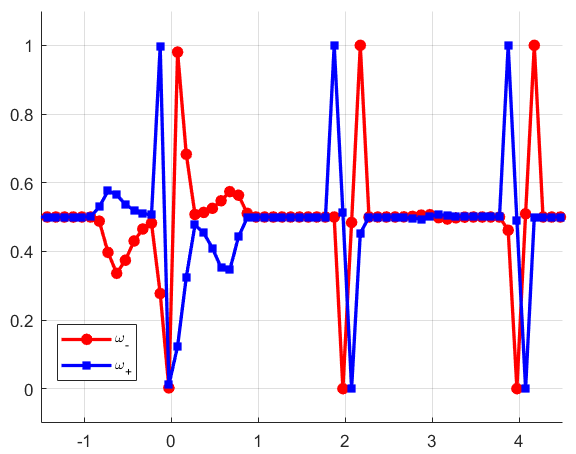}}
\end{subfigure}
\begin{subfigure}{0.24\textwidth}
{\includegraphics[width=\textwidth]{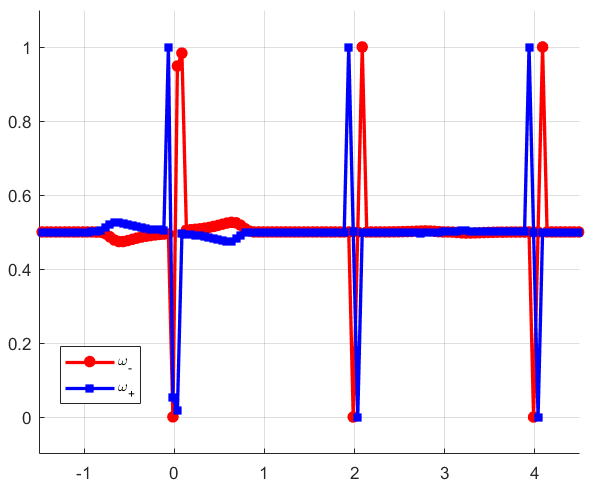}}
\end{subfigure}
\begin{subfigure}{0.245\textwidth}
{\includegraphics[width=\textwidth]{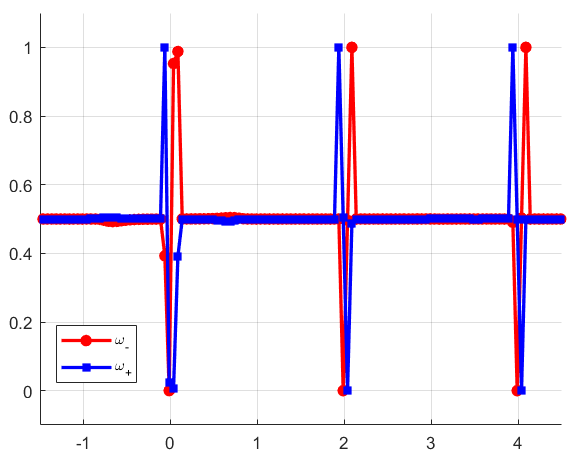}}
\end{subfigure}
\vspace{-0.2cm}
\caption{\footnotesize{Test 1. Singularity inside a cell. Results obtained using $g(\omega)$ (left) and $\omega_{new}^Z$ (right) both with $r=2$, for $\Delta x=0.1$ (first two columns) and $\Delta x=0.05$ (last two columns).}\label{fig_test1D_3}}
\end{figure}
\vspace{-0.3cm}
\begin{figure}[h!]
\centering
\begin{subfigure}{0.24\textwidth}
{\includegraphics[width=\textwidth]{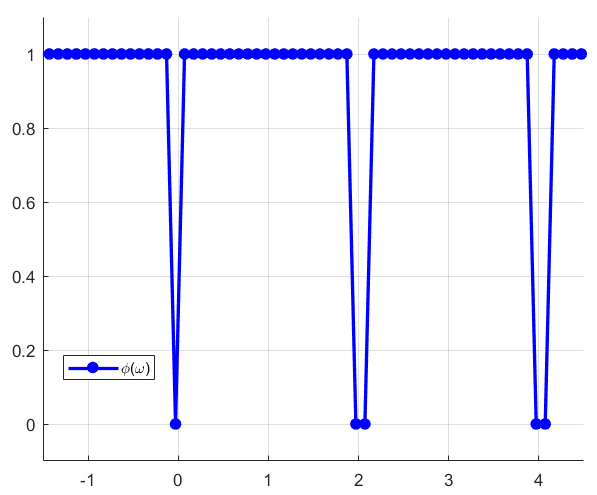}}
\end{subfigure}
\begin{subfigure}{0.245\textwidth}
{\includegraphics[width=\textwidth]{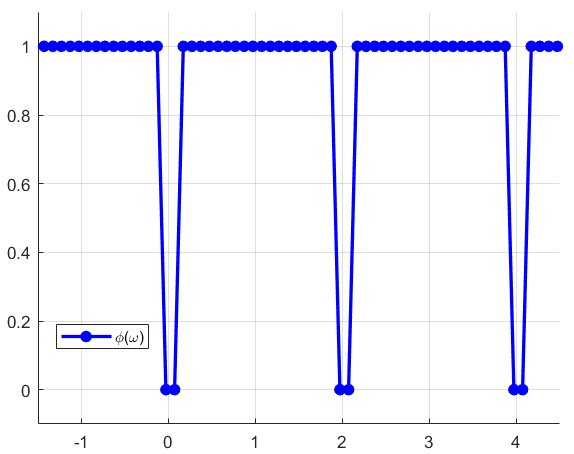}}
\end{subfigure}
\begin{subfigure}{0.238\textwidth}
{\includegraphics[width=\textwidth]{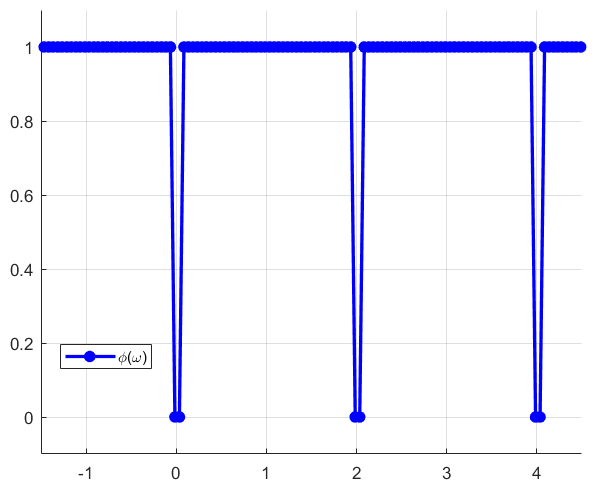}}
\end{subfigure}
\begin{subfigure}{0.247\textwidth}
{\includegraphics[width=\textwidth]{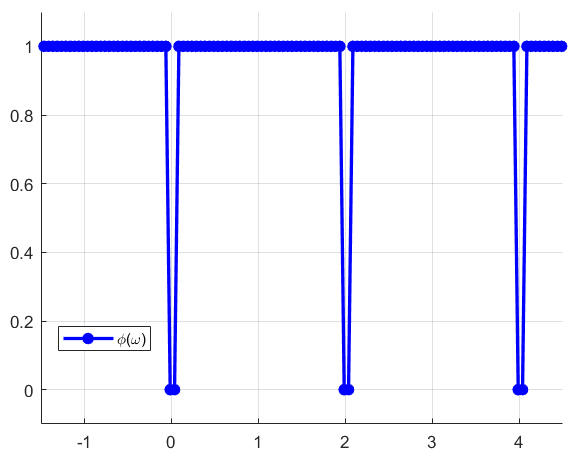}}
\end{subfigure}
\vspace{-0.2cm}
\caption{\footnotesize{Test 1. Singularity inside a cell. Results obtained using $\phi(g(\omega))$ (left) and $\phi_{new}^Z$ (right) both with $r=2$, for $\Delta x=0.1$ (first two columns) and $\Delta x=0.05$ (last two columns).}\label{fig_test1D_4}}
\end{figure}
\begin{figure}[h!]
\centering
\vspace{-0.2cm}
\begin{subfigure}{0.32\textwidth}
	{\includegraphics[width=\textwidth]{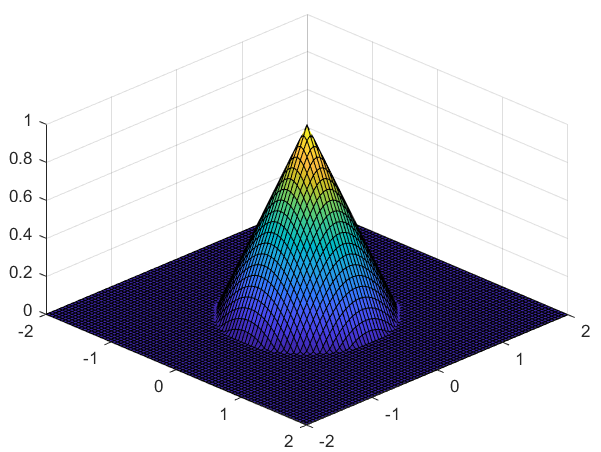}}
	\caption{\footnotesize{}\label{fig:functions_2Da}}	
\end{subfigure}
\begin{subfigure}{0.32\textwidth}
	{\includegraphics[width=\textwidth]{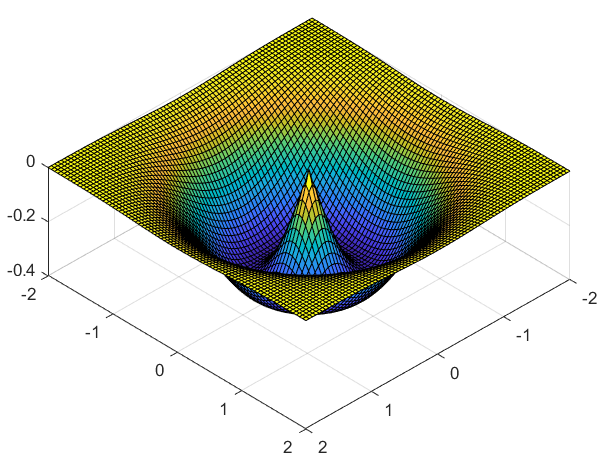}}
	\caption{\footnotesize{}\label{fig:functions_2Db}}
\end{subfigure}
\begin{subfigure}{0.32\textwidth}
	{\includegraphics[width=\textwidth]{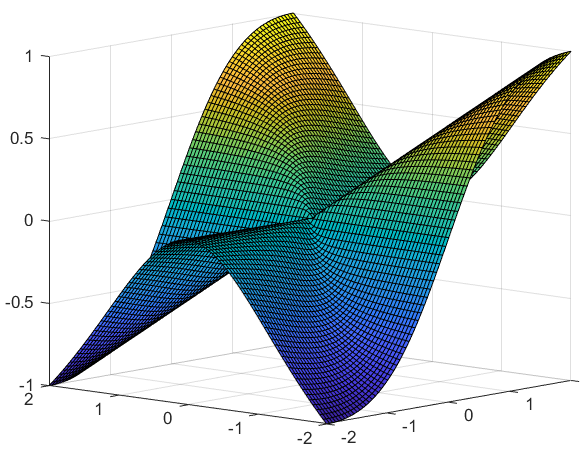}}	
\caption{\footnotesize{}\label{fig:functions_2Dc}}
	\end{subfigure}
\caption{\footnotesize{Tests in 2D.  (a) Singularities at a point and on a circle, (b) A singularity at the origin, (c)  Non differentiable at the origin. }
	\label{fig:functions_2D}}
\end{figure}

\vspace{0.6cm}\noindent
{\bf Test 2.}  Let us start the 2D tests considering a case where the singularities are located at a point and on a circle. 
We thus consider the  cone visible in Fig. \ref{fig:functions_2Da}, defined by the following equation
\begin{equation}
f(x,y)=\left\{
\begin{array}{ll}
1-\sqrt{x^2+y^2 }\qquad& \textrm{ if }x^2+y^2\leq 1\\
0 &\textrm{ otherwise },
\end{array}
\right.
\end{equation}
in the square $[-2,2]^2$, which clearly has a point of singularity located in the origin and a singularity circle at the base of the cone.  
We compare the results obtained using the splitting indicators $\omega_{split}$, and the two indicators defined in (\ref{explicit_beta_2D}) and (\ref{explicit_beta_2D_old}), denoted by $\omega_{2D}^F$ (Full) and $\omega_{2D}^P$ (Partial), respectively, in order to identify the best choice that will be used for the implementation of the AF scheme considered in the numerical simulations reported in Sect. \ref{sec:AFS_tests}. 
We analyze both cases in which the point of singularity in the origin falls on a grid point or inside a cell (again not in the center but shifted toward the south/west node). 
First, we analyze the case of the central singularity falling on a grid point, using the contour plots to show precisely the behavior of the indicators with respect to the location of the singularities. In order to highlight this fact, we plot also the points of singularity, in red in Fig. \ref{fig2_1}, in green in Fig. \ref{fig2_2}.

From Fig. \ref{fig2_1} we can see that all the indicators seem to detect the right regions of singularity and have a very good behavior in regular regions (in light green), although $\omega^F_{2D}$ is evidently more precise. Looking at the neighborhood of the origin we can note that both $\omega^P_{2D}$ and $\omega_{split}$ have fluctuations only in the diagonal direction, with the latter presenting wider oscillations and spreading too much the singular area,  whereas $\omega^F_{2D}$ has a more uniform behavior, also around the singularities on the circle.
\begin{figure}[h!]
	\centering
	\begin{subfigure}{0.3\textwidth}
		{\includegraphics[width=\textwidth]{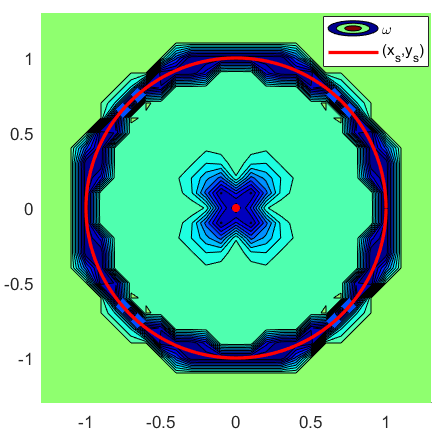}}
	\end{subfigure}
	\begin{subfigure}{0.3\textwidth}
		{\includegraphics[width=\textwidth]{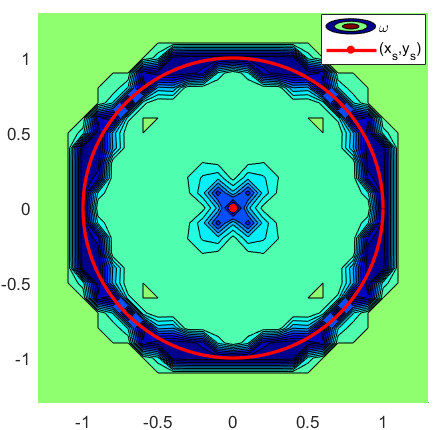}}
	\end{subfigure}
	\begin{subfigure}{0.35\textwidth}
		{\includegraphics[width=\textwidth]{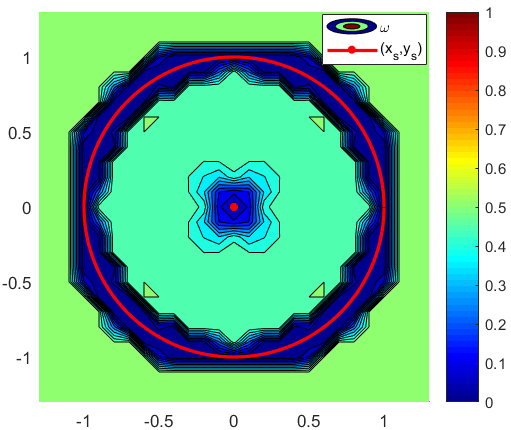}}
	\end{subfigure}\\
	\begin{subfigure}{0.3\textwidth}
		{\includegraphics[width=\textwidth]{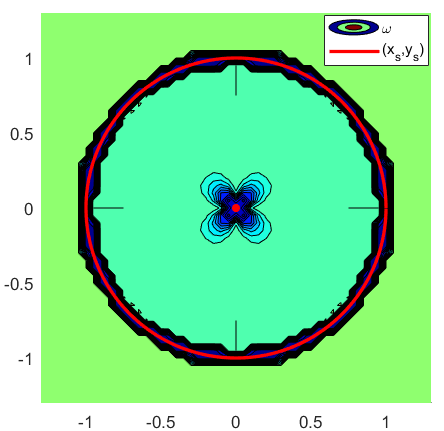}}
	\end{subfigure}
	\begin{subfigure}{0.3\textwidth}
		{\includegraphics[width=\textwidth]{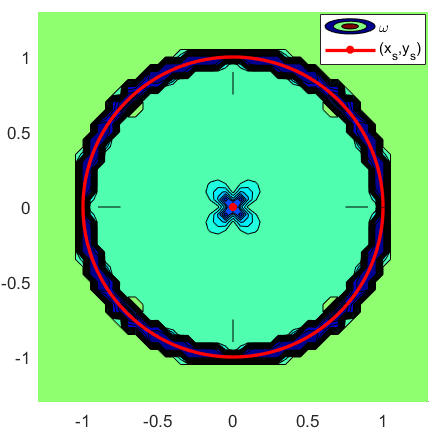}}
	\end{subfigure}
	\begin{subfigure}{0.35\textwidth}
		{\includegraphics[width=\textwidth]{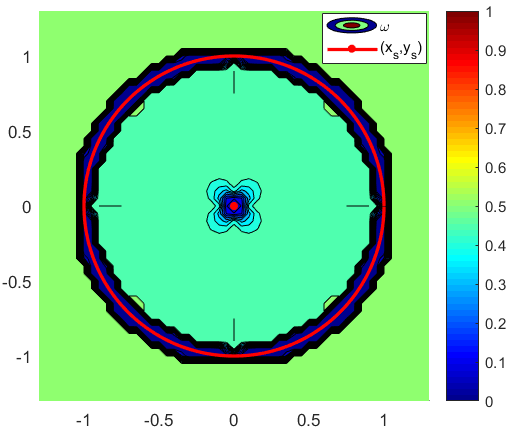}}
	\end{subfigure}
	\caption{\footnotesize{Test 2. Singularity on a grid point. Results obtained using $g(\omega)$ with $\omega_{split}$ (left), $\omega^P_{2D}$ (middle) and $\omega^F_{2D}$ (right), for $\Delta x=\Delta y=0.1$ and $\Delta x=\Delta y=0.05$.}\label{fig2_1}}
\end{figure}
\begin{figure}[h!]
	\centering
	\begin{subfigure}{0.3\textwidth}
		{\includegraphics[width=\textwidth]{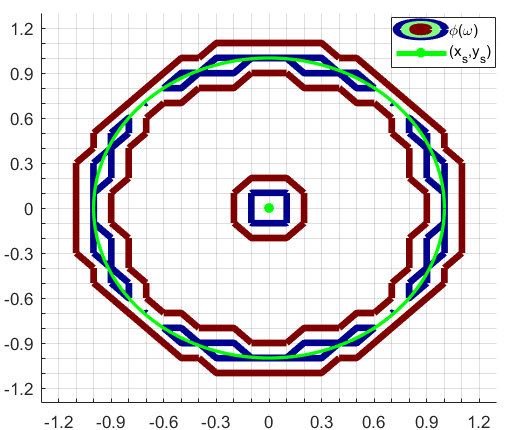}}
	\end{subfigure}
	\begin{subfigure}{0.3\textwidth}
		{\includegraphics[width=\textwidth]{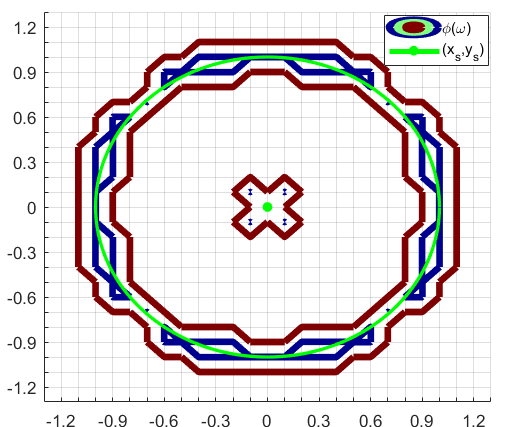}}
	\end{subfigure}
	\begin{subfigure}{0.345\textwidth}
		{\includegraphics[width=\textwidth]{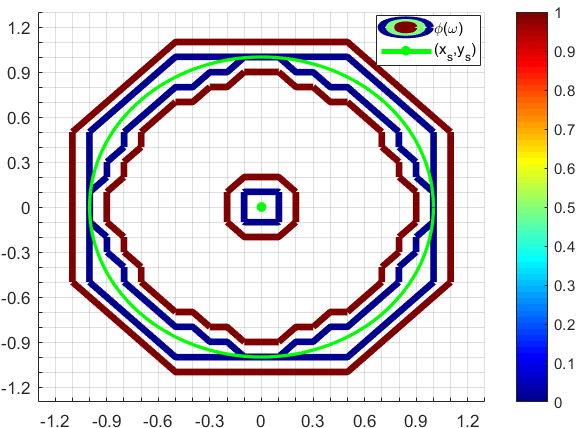}}
	\end{subfigure}\\
	\begin{subfigure}{0.3\textwidth}
		{\includegraphics[width=\textwidth]{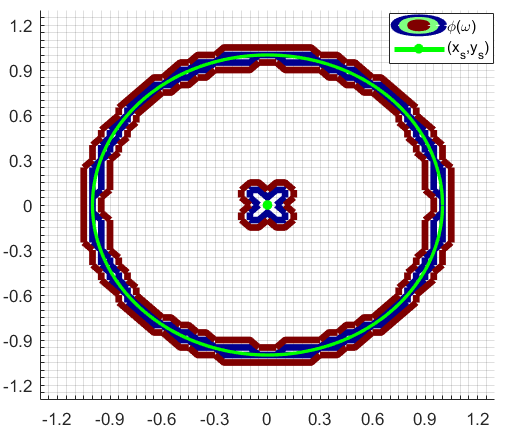}}
	\end{subfigure}
	\begin{subfigure}{0.3\textwidth}
		{\includegraphics[width=\textwidth]{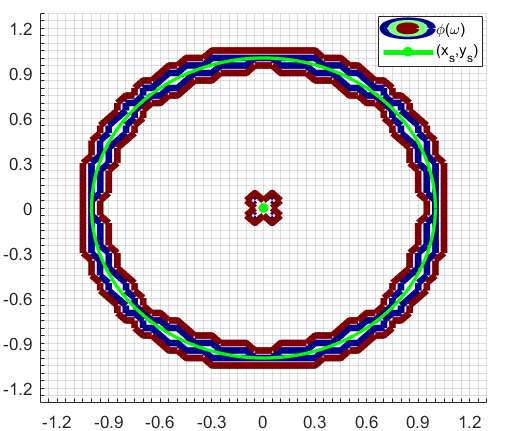}}
	\end{subfigure}
	\begin{subfigure}{0.345\textwidth}
		{\includegraphics[width=\textwidth]{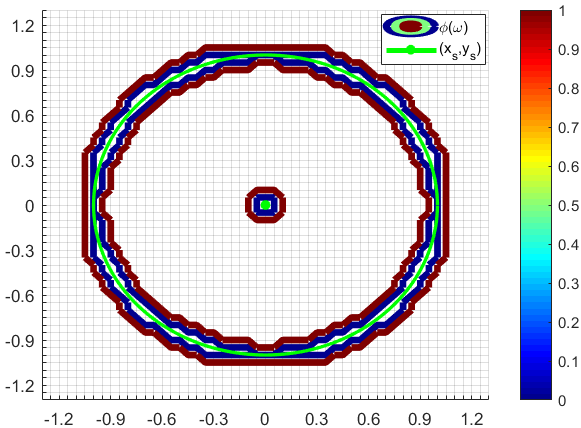}}
	\end{subfigure}
	\caption{\footnotesize{Test 2. Singularity on a grid point. Results obtained using $\phi_{split}$ (left), $\phi^P_{2D}$ (middle) and $\phi^F_{2D}$ (right), for $\Delta x=\Delta y=0.1$ and $\Delta x=\Delta y=0.05$.}\label{fig2_2}}
\end{figure}
\begin{figure}[h!]
	\centering
	\begin{subfigure}{0.3\textwidth}
		{\includegraphics[width=\textwidth]{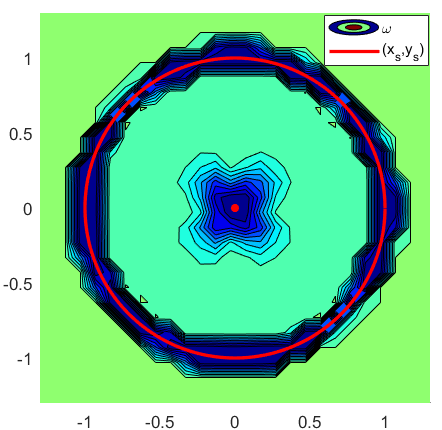}}
	\end{subfigure}
	\begin{subfigure}{0.3\textwidth}
		{\includegraphics[width=\textwidth]{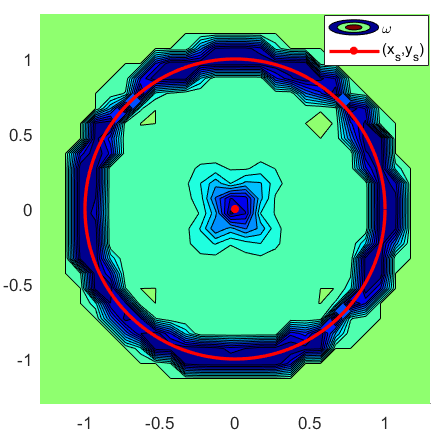}}
	\end{subfigure}
	\begin{subfigure}{0.35\textwidth}
		{\includegraphics[width=\textwidth]{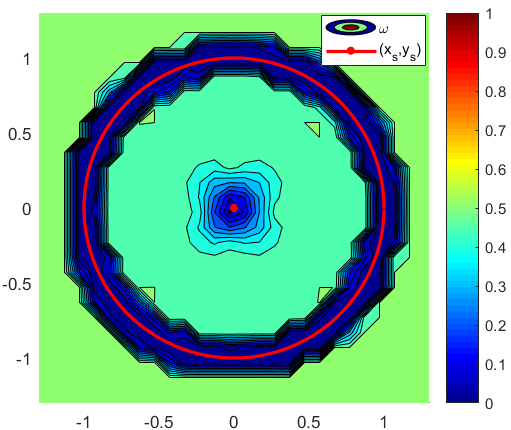}}
	\end{subfigure}\\
	\begin{subfigure}{0.3\textwidth}
		{\includegraphics[width=\textwidth]{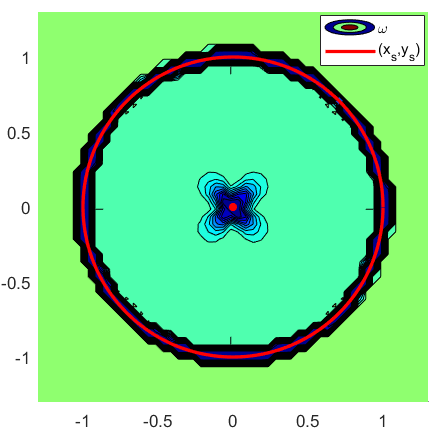}}
	\end{subfigure}
	\begin{subfigure}{0.3\textwidth}
		{\includegraphics[width=\textwidth]{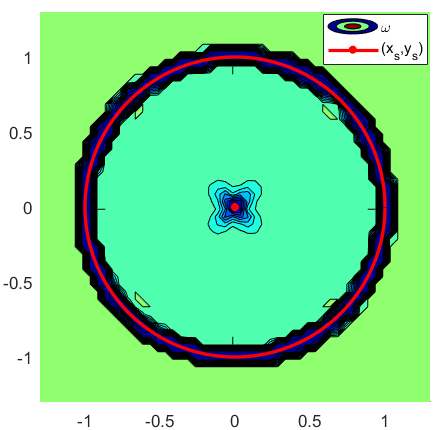}}
	\end{subfigure}
	\begin{subfigure}{0.35\textwidth}
		{\includegraphics[width=\textwidth]{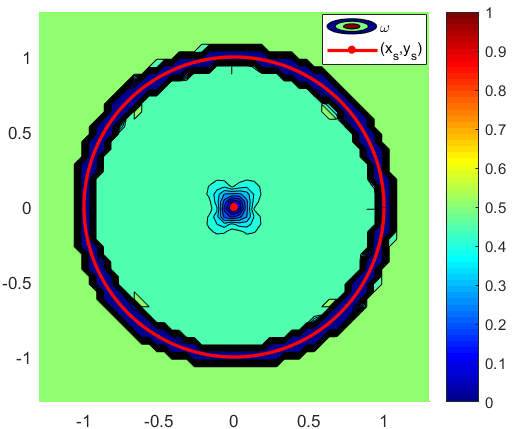}}
	\end{subfigure}
	\caption{\footnotesize{Test 2. Singularity inside a cell. Results obtained using $g(\omega)$ with $\omega_{split}$ (left), $\omega^P_{2D}$ (middle) and $\omega^F_{2D}$ (right), for $\Delta x=\Delta y=0.1$ and $\Delta x=\Delta y=0.05$.}\label{fig2_3}}
	\medskip
	\vspace{0.3 cm}
	\begin{subfigure}{0.3\textwidth}
		{\includegraphics[width=\textwidth]{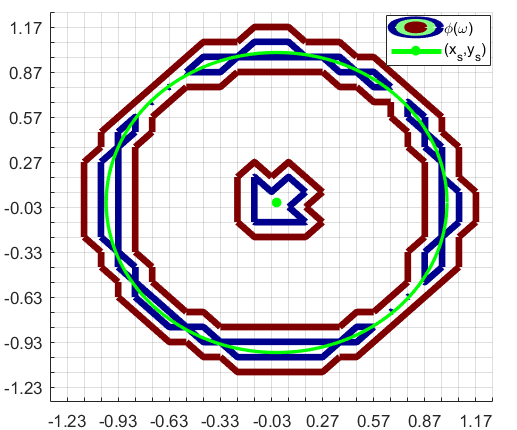}}
	\end{subfigure}
	\begin{subfigure}{0.3\textwidth}
		{\includegraphics[width=\textwidth]{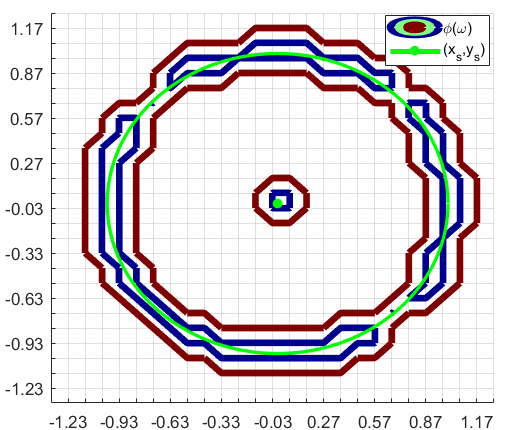}}
	\end{subfigure}
	\begin{subfigure}{0.345\textwidth}
		{\includegraphics[width=\textwidth]{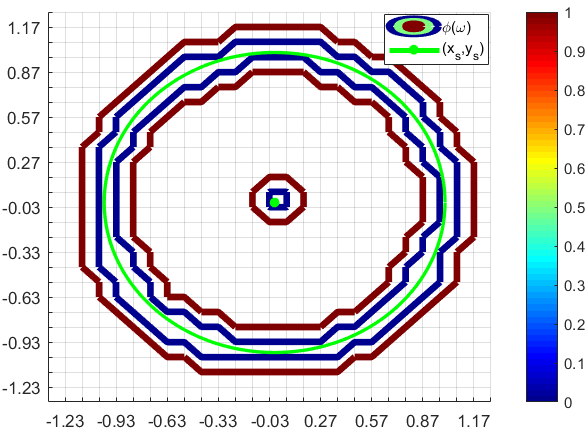}}
	\end{subfigure}\\
	\begin{subfigure}{0.3\textwidth}
		{\includegraphics[width=\textwidth]{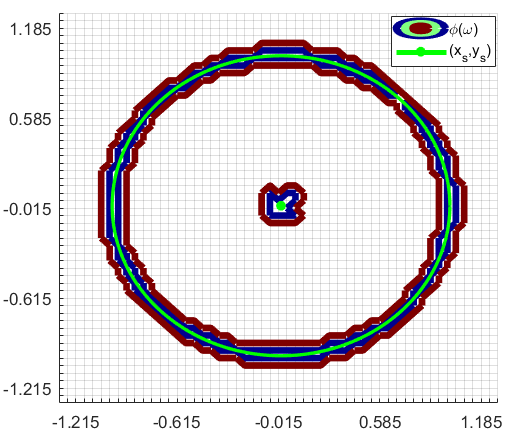}}
	\end{subfigure}
	\begin{subfigure}{0.3\textwidth}
		{\includegraphics[width=\textwidth]{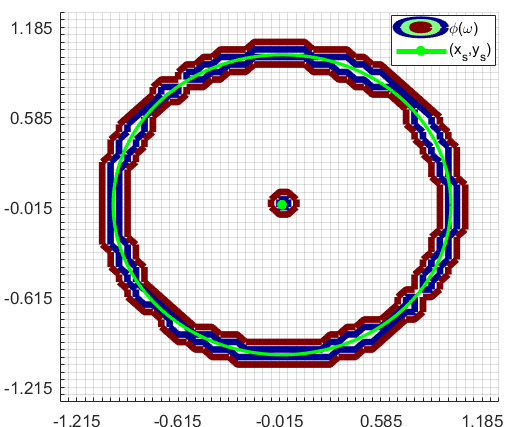}}
	\end{subfigure}
	\begin{subfigure}{0.345\textwidth}
		{\includegraphics[width=\textwidth]{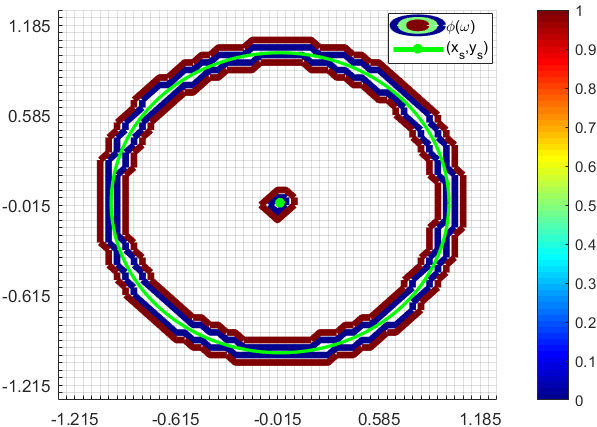}}
	\end{subfigure}
	\caption{\footnotesize{Test 2. Singularity on a grid point. Results obtained using $\phi_{split}$ (left), $\phi^P_{2D}$ (middle) and $\phi^F_{2D}$ (right), for $\Delta x=\Delta y=0.1$ and $\Delta x=\Delta y=0.05$.}\label{fig2_4}}
\end{figure}

Consequently, looking at Fig. \ref{fig2_2}, in which we have highlighted only the $0.1$-level and the $1$-level, we have that $\phi^F_{2D}$ recognizes with extreme precision all the cells and the grid nodes containing a singularity, whereas  $\phi^P_{2D}$ seems to miss some points of the circle and localizes the singularity in the center, consisting in this case of just the origin and the four points in the diagonal direction. Similarly to the latter, $\phi_{split}$ seems to miss some points on the circle, whereas has better precision in detecting the singularity in the origin, at least in the first refinement, although it spreads the singular area in the diagonal direction when the grid is refined. This behavior is rather typical and will be found also in the other simulations. 
The full indicator $\omega^F_{2D}$ is able to detect singular cells even when the singularity just barely intersects the considered region, whereas $\omega^P_{2D}$ and $\omega_{split}$ recognize only ``strong'' singularities, that are close to a grid point or situated around the center of the considered cell. Moreover, looking at the behavior of $\phi^P_{2D}$ around the singularity in the center, enlarging the detected singular region, which instead should consist of only one point, we are led to believe that the results of $\phi^F_{2D}$ are more precise and, evidently, more stable with respect to mesh refinements.

Next, we repeat the test using a grid staggered with respect to the singularity in the origin.
Figs. \ref{fig2_3} and \ref{fig2_4} confirm the impressions given by the previous simulation. In fact, the indicator $\omega^F_{2D}$ and the function $\phi^F_{2D}$ are able to recognize all the cells containing a singularity, in particular those around the circle, which is always inside the 0-level set of $\phi^F_{2D}$ (see Fig. \ref{fig2_4} on the right). Instead, $\omega^P_{2D}$ and $\omega_{split}$ have a rather asymmetrical behavior on the circle. The portion of the circle in the ``South-West'' direction is well detected, whereas on the other three directions the detected regions degenerate into points (see Fig. \ref{fig2_4} on the left and in the middle). Moreover, in Fig. \ref{fig2_4} we can observe that the simple $\phi_{split}$ spreads unnaturally the singular region in the center. Notice that in this case also the full indicator has a particular behavior around the origin in the second refinement, spreading the detected singular region in the direction of the singularity.

\vspace{0.2cm}\noindent
{\bf Test 3.} For this test we consider the nonlinear function visible in Fig. \ref{fig:functions_2Db}, which is regular in the whole domain except in the origin, where it presents a singularity, that is 
\begin{equation}
f(x,y)=-e^{-(x^2+y^2)}\sin{\left(\sqrt{x^2+y^2 }\right)},\qquad (x,y)\in [-2,2]^2.
\end{equation}
Since our aim is mainly to show the good behavior of the indicators also in non trivial regularity regions, we perform the test only considering the point of singularity on a grid node. This is also why, focusing on the singularity in the origin, the results are very similar to those of the previous test, especially with the staggered grid.
\begin{figure}[h!]
\centering
\begin{subfigure}{0.3\textwidth}
{\includegraphics[width=\textwidth]{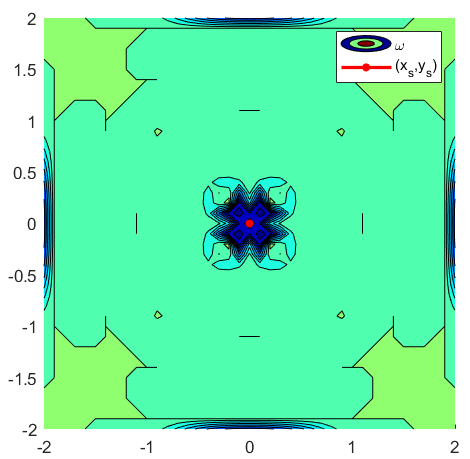}}
\end{subfigure}
\begin{subfigure}{0.3\textwidth}
{\includegraphics[width=\textwidth]{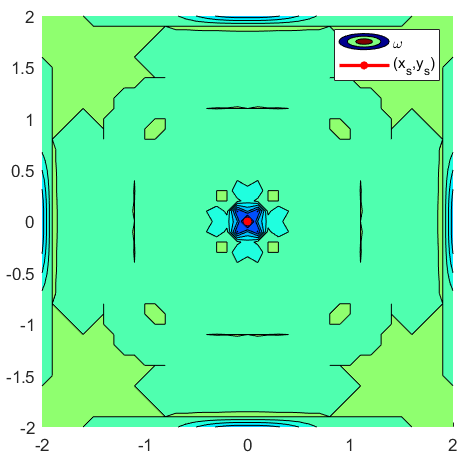}}
\end{subfigure}
\begin{subfigure}{0.35\textwidth}
{\includegraphics[width=\textwidth]{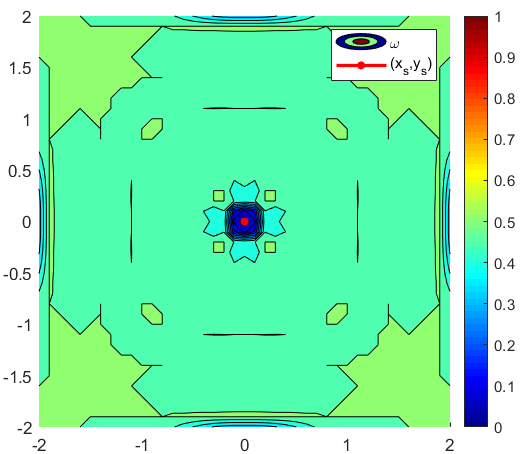}}
\end{subfigure}\\
\begin{subfigure}{0.3\textwidth}
{\includegraphics[width=\textwidth]{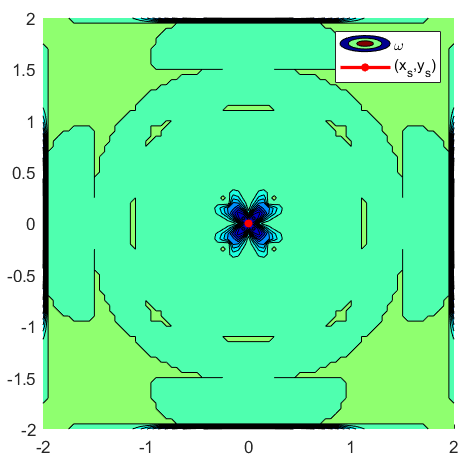}}
\end{subfigure}
\begin{subfigure}{0.3\textwidth}
{\includegraphics[width=\textwidth]{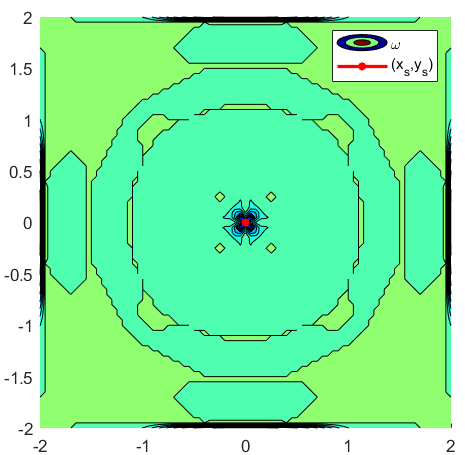}}
\end{subfigure}
\begin{subfigure}{0.35\textwidth}
{\includegraphics[width=\textwidth]{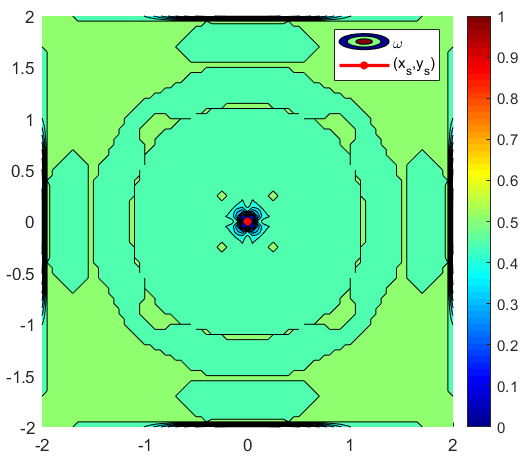}}
\end{subfigure}
\vspace{-0.2cm}
\caption{\footnotesize{Test 3. Singularity on a grid point. Results obtained using $g(\omega)$ with $\omega_{split}$ (left), $\omega^P_{2D}$ (middle) and $\omega^F_{2D}$ (right), for $\Delta x=\Delta y=0.1$ and $\Delta x=\Delta y=0.05$.}\label{fig2_5}}
\end{figure}
\begin{figure}[h!]
\centering
\begin{subfigure}{0.3\textwidth}
{\includegraphics[width=\textwidth]{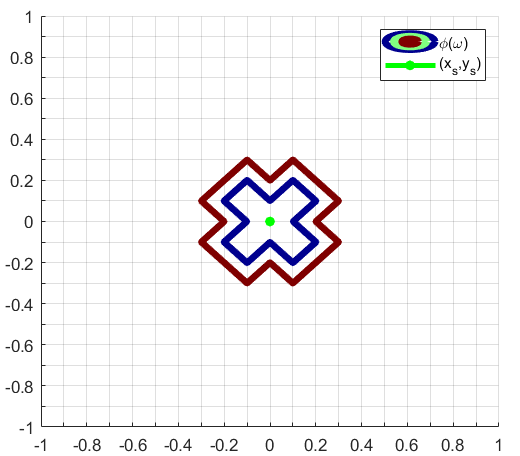}}
\end{subfigure}
\begin{subfigure}{0.3\textwidth}
{\includegraphics[width=\textwidth]{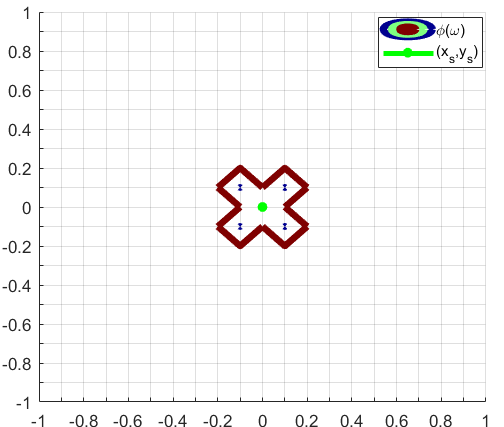}}
\end{subfigure}
\begin{subfigure}{0.345\textwidth}
{\includegraphics[width=\textwidth]{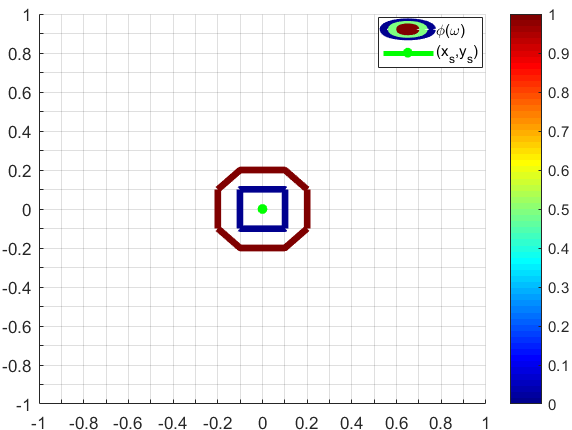}}
\end{subfigure}\\
\begin{subfigure}{0.3\textwidth}
{\includegraphics[width=\textwidth]{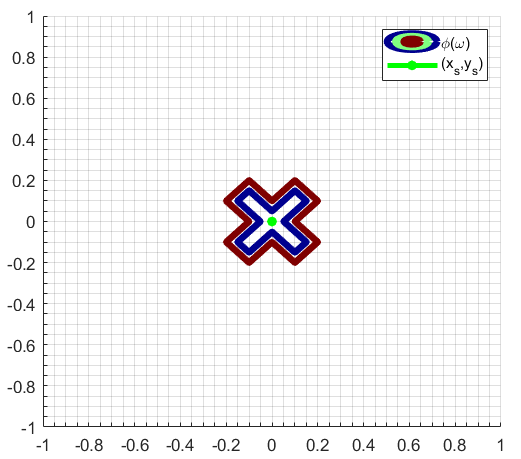}}
\end{subfigure}
\begin{subfigure}{0.3\textwidth}
{\includegraphics[width=\textwidth]{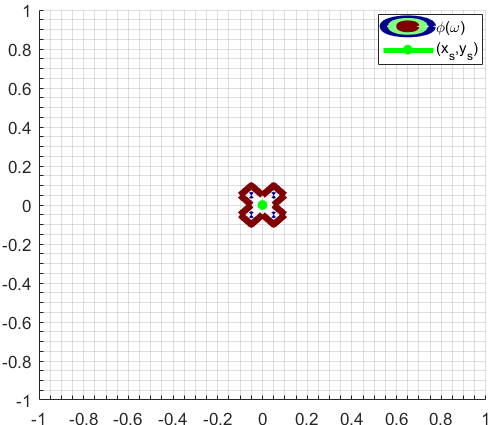}}
\end{subfigure}
\begin{subfigure}{0.344\textwidth}
{\includegraphics[width=\textwidth]{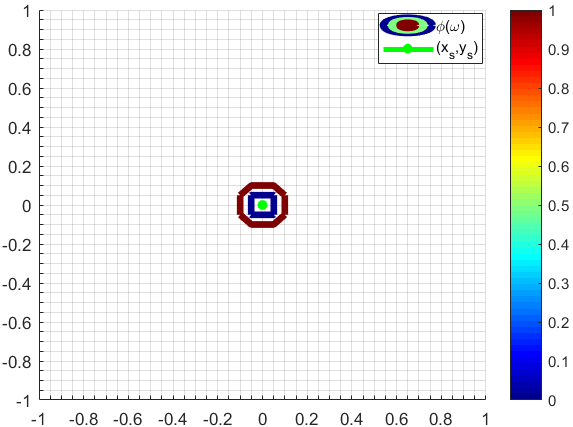}}
\end{subfigure}
\vspace{-0.2cm}
\caption{\footnotesize{Test 3. Singularity on a grid point. Results obtained using $\phi_{split}$ (left), $\phi^P_{2D}$ (middle) and $\phi^F_{2D}$ (right)), for $\Delta x=\Delta y=0.1$ and $\Delta x=\Delta y=0.05$.}\label{fig2_6}}
\vspace*{-0.2cm}
\end{figure}
Figs. \ref{fig2_5} and \ref{fig2_6} (in which we have zoomed the neighborhood of the singularity) follow the same line of the first test, at least regarding the behavior of the multidimensional indicators around the point of singularity in the origin. In fact, we can observe that $\omega^F_{2D}$ has again a more uniform behavior with respect to $\omega_{2D}^P$ in all directions, whereas the results on the regular regions are practically the same. On the other hand, in this situation the splitting indicator $\omega_{split}$ gives evidently worse results, spreading the detected region in the diagonal directions, which moreover does not shrink as the grid is refined 
(see Fig. \ref{fig2_6} on the left). 
Note that, because of the periodic boundary conditions, all the indicators correctly detect more singularities at the borders, as can be seen in Fig. \ref{fig2_5}. 
We avoided to represent the corresponding detected regions in order to highlight the behavior of the indicators around the singularity in the origin, which was our primary interest.

Summarizing the previous observations on Tests 2 and 3, we can deduce that the formula (\ref{explicit_beta_2D_old}) is more suitable to localize singularities with high precision, since the corresponding indicator is able to select the correct grid points or cells characterized by a strong discontinuity in the gradient. Instead the splitting indicator, although faster and simpler to implement, presents various drawbacks which makes it rather unreliable and unfit for our purposes. On the other hand, using the formula (\ref{explicit_beta_2D}), we are able to precisely detect the regularity of the function in the whole domain $I_{i,j}=[x_{j-1},x_{j+1}]\times[y_{i-1},y_{i+1}]$, boundary included. 
Hence, it is clear that the correct indicators for the construction of our AF scheme should be based on (\ref{explicit_beta_2D}), instead of (\ref{explicit_beta_2D_old}), 
since when constructing non splitting 2D-numerical schemes we usually need at least a nine-point stencil $\mathcal S_{i,j}=\{x_{j-1}, x_j, x_{j+1}\}\times \{y_{i-1}, y_i, y_{i+1}\}$.
This is why, when working with second order schemes in 2D, as the ones we have defined as $S^A$ in Sect. \ref{sec:AFS_2D}, in order to verify the high-order consistency property using Taylor expansion, we have to require the regularity of the function in the whole domain $I_{i,j}$. 

\vspace{0.2cm}\noindent
{\bf Test 4.}
We conclude the present analysis on the smoothness indicators by considering a well known example of a function which is differentiable everywhere except at the origin, visible in Fig. \ref{fig:functions_2Dc}, and defined as 
\begin{equation}
f(x,y)=\left\{
\begin{array}{ll}
\frac{y x^2}{x^2+y^2}\qquad& \textrm{ if }(x,y)\not = (0,0),\\
0 &\textrm{ otherwise, }
\end{array}
\right.
\end{equation}
again in the square $[-2,2]^2$. 
\begin{figure}[h!]
\centering
\begin{subfigure}{0.3\textwidth}
{\includegraphics[width=\textwidth]{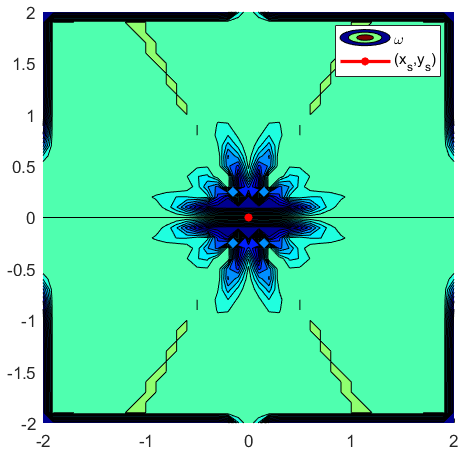}}
\end{subfigure}
\begin{subfigure}{0.3\textwidth}
{\includegraphics[width=\textwidth]{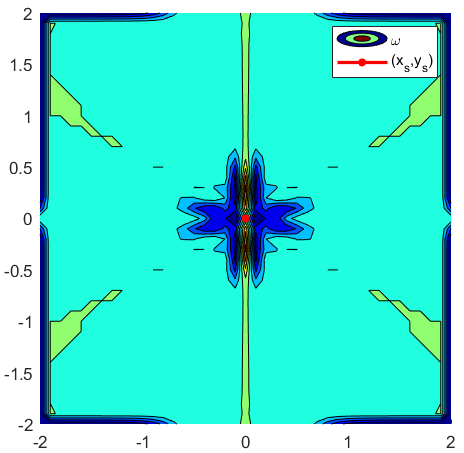}}
\end{subfigure}
\begin{subfigure}{0.35\textwidth}
{\includegraphics[width=\textwidth]{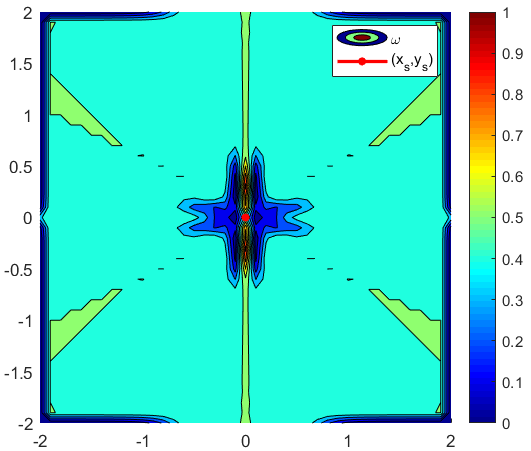}}
\end{subfigure}\\
\begin{subfigure}{0.3\textwidth}
{\includegraphics[width=\textwidth]{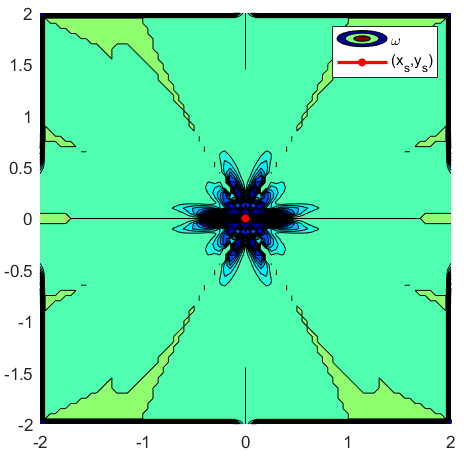}}
\end{subfigure}
\begin{subfigure}{0.3\textwidth}
{\includegraphics[width=\textwidth]{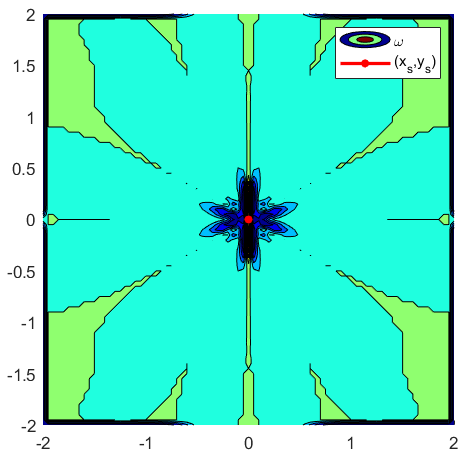}}
\end{subfigure}
\begin{subfigure}{0.35\textwidth}
{\includegraphics[width=\textwidth]{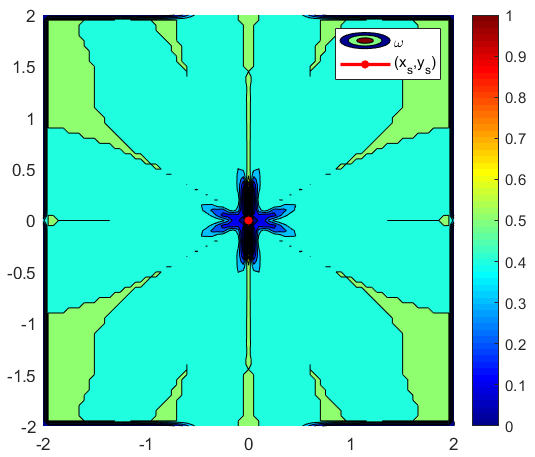}}
\end{subfigure}
\vspace{-0.2cm}
\caption{\footnotesize{Test 4. Singularity on a grid point. Results obtained using $g(\omega)$ with $\omega_{split}$ (left), $\omega^P_{2D}$ (middle) and $\omega^F_{2D}$ (right), for $\Delta x=\Delta y=0.1$ and $\Delta x=\Delta y=0.05$.}\label{fig2_9}}
\end{figure}
\begin{figure}[h!]
\centering
\begin{subfigure}{0.3\textwidth}
{\includegraphics[width=\textwidth]{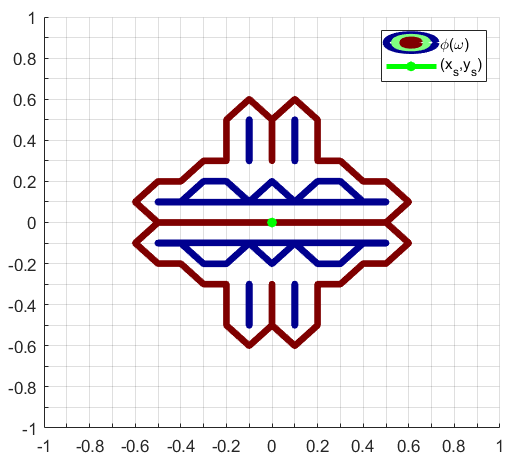}}
\end{subfigure}
\begin{subfigure}{0.3\textwidth}
{\includegraphics[width=\textwidth]{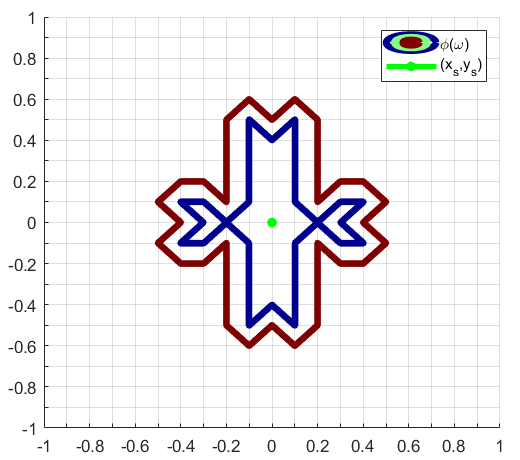}}
\end{subfigure}
\begin{subfigure}{0.345\textwidth}
{\includegraphics[width=\textwidth]{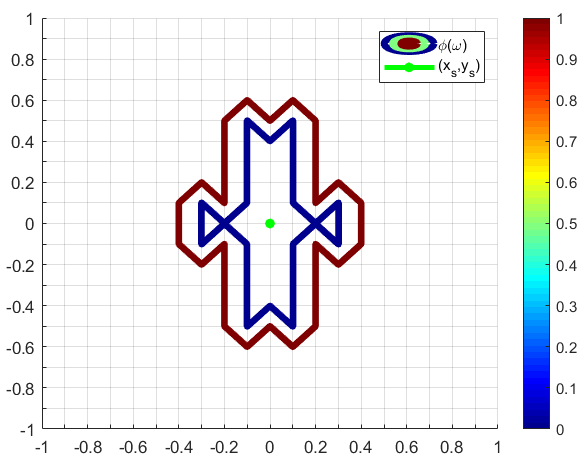}}
\end{subfigure}\\
\begin{subfigure}{0.3\textwidth}
{\includegraphics[width=\textwidth]{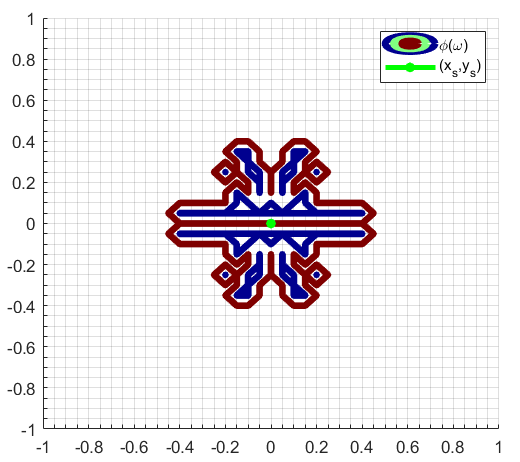}}
\end{subfigure}
\begin{subfigure}{0.3\textwidth}
{\includegraphics[width=\textwidth]{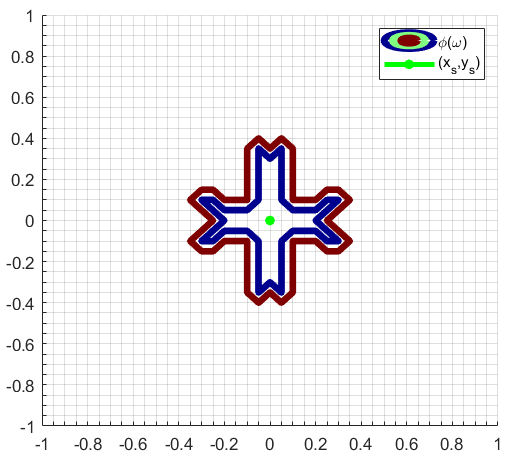}}
\end{subfigure}
\begin{subfigure}{0.344\textwidth}
{\includegraphics[width=\textwidth]{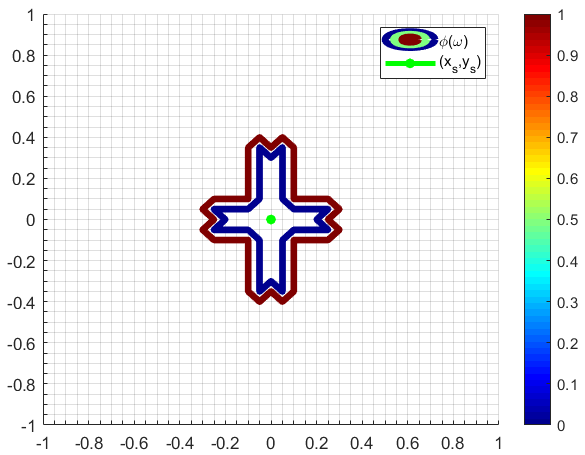}}
\end{subfigure}
\vspace{-0.2cm}
\caption{\footnotesize{Test 4. Singularity on a grid point. Results obtained using $\phi_{split}$ (left), $\phi^P_{2D}$ (middle) and $\phi^F_{2D}$ (right)), for $\Delta x=\Delta y=0.1$ and $\Delta x=\Delta y=0.05$.}\label{fig2_10}}
\end{figure}
In order to highlight the major drawback of the splitting indicator $\omega_{split}$, based on one-directional reconstructions, we perform the test with the point of non-differentiabilty falling on a grid point. The results obtained are collected in Figs. \ref{fig2_9}-\ref{fig2_10} and show that both indicators, $\omega^F_{2D}$ and $\omega^P_{2D}$ are able to detect the singular point in the origin, although the region is spread along the axes directions, especially in the case of $\phi_{2D}^P$, whereas $\phi^F_{2D}$ presents more precise and symmetric results. On the other hand, as could be expected, the splitting indicator does not recognize the low regularity in the origin, since both directional derivatives exist and are equal to $0$. Moreover, $\omega_{split}$ shows a rather peculiar behavior in the region around the origin, detecting more false positive than the other indicators.

These last results are rather promising and suggest that our definition is the correct one for identifying the lack of regularity of functions, also when the gradient exists but it is not continuous. Evidently, deeper investigations seem necessary in order to justify the enlargement of the detected singular area along the axis directions and, in general, the differences between the two versions  $\omega^F_{2D}$ and $\omega^P_{2D}$.

\subsection{Tests on the Adaptive Filtered schemes with 2D smoothness indicators}\label{sec:AFS_tests}
Here we present some two-dimensional tests of evolution problems designed to show the properties of our AF scheme when the switching from one scheme to the other is regulated by the smoothness indicators studied in Sect. \ref{sec:ind_2D}.
Our goal is also to compare the performance of our AF schemes $S^{AF}$ with those of the Filtered Schemes $S^F$ proposed in \cite{BFS16}, and of the WENO scheme of second/third order proposed in \cite{JP00}. 
Regarding the WENO 2/3 scheme, here we use the same efficient implementation suggested in \cite{JP00} (Remark 1 on page 2130), adding the improvement suggested in \cite{ABM10}, that consists in  choosing $\sigma=\Delta x^2$, instead of $\sigma=10^{-8}$ as done in the original paper by \cite{JP00}. 
Concerning the basic Filtered Scheme, we use the implementation suggested in \cite{BFS16}, but avoiding the use of the limiter correction in all the numerical tests. This is mainly because we want to show that the unstability problems that have been fixed through the introduction of the limiter in \cite{BFS16}, can be solved by the adaptive procedure and the function $\phi$ proposed.  
Moreover, since the main aim is to use even higher order schemes (fourth order schemes) and to show the reliability of the adaptive tuning, the use of the limiter would be counterproductive, since it would inevitably limit also the full accuracy of the resulting scheme. 
This fact has been proved in various forms in literature (the scheme is TVD) and can be easily confirmed through some easy numerical tests, such as the transport of a regular function, considered as Test 5. 
For further comparisons with higher order schemes, we implemented also WENO 3/5 with RK3 (TVD) and RK4 (not TVD) as time integrator, presenting the results in the first next test which considers a regular function since in this case we can better appreciate the differences. 

For each test we specify the monotone and high-order schemes composing the filtered schemes, as well as the CFL number $\lambda:=\max\{\lambda_x,\lambda_y\}$. 
In particular, the CFL number will be chosen to satisfy the CFL condition
\begin{equation}
\max\{\lambda_x \max|H_p(p,q)|,\lambda_y \max|H_q(p,q)|\}\leq \frac{1}{2},
\end{equation}
which is more easy to implement than \eqref{cond_CFL_2D}. 
We also compute the errors and orders in  $L^\infty$ and $L^1$ norm. 
For all the following tests, we use homogeneous Neumann boundary conditions, except for Test 9, for which periodic boundary conditions are needed.

\vspace{0.2cm}\noindent
{\bf Test 5.} As a first example of evolution problem, we consider the transport of a very regular function at constant velocity, that is 
\begin{equation}
\label{ex6:eq1}
\left\{
\begin{array}{ll}
v_t+v_x+v_y=0,\qquad &\textrm{ in } (0,T)\times\Omega, \\
v(0,x,y)=v_0(x,y),&\textrm{ in }\Omega
\end{array}
\right.
\end{equation}
where $\Omega=[-2,2]^2$ and $T=0.9$, with the regular initial condition
\begin{equation}
v_0(x,y)=\max\left\{0, 1-x^2-y^2\right\}^5.
\end{equation}
The chosen CFL number is $\lambda=0.2<\frac{1}{2}$. 
For this regular test, we start considering the \emph{Heun-Centered} (HC) scheme, as defined in \eqref{HC_2D}-\eqref{RK2}, in addition to the WENO scheme of second/third order proposed in \cite{JP00}, and the AF schemes which use, as high-order method $S^A$, the HC scheme or the \emph{fourth-order central Runge-Kutta} method  defined in \eqref{C4ord}-\eqref{RK4} and from now on denoted by RKC4 for brevity. 
Looking at Tab. \ref{ex1:table1} we can note that the AF-HC basically coincides with the simple HC scheme when the solution is regular, as expected, and the results clearly testify the success of the filtering process, with both AF schemes achieving the optimal order of the high-order scheme in both norms. 
\begin{table}[h]
\caption{Test 5. Errors and orders in $L^\infty$ and  $L^1$ norms. \label{ex1:table1}}
\centering
\begin{tabular}{c c|c c |c c |c c  | c c}
& & \multicolumn{2}{|c|}{\bf HC }&\multicolumn{2}{|c|}{\bf AF-HC }&\multicolumn{2}{|c|}{\bf AF-RKC4}&\multicolumn{2}{|c}{\bf WENO 2/3}\\
\hline
\hline
$N_x$ & $N_t$ & $L^\infty$ Err & Ord & $L^\infty$ Err & Ord  & $L^\infty$ Err & Ord  & $L^\infty$ Err & Ord  \\
\hline
$40$ &	$30$		& $9.39$e-$02$	& 	&	$9.53$e-$02$	&	&  	$6.32$e-$03$	& 	& $5.30$e-$02$ &  	\\

$80$ &	$60$		& $2.23$e-$02$ &	$2.07$	& 	$2.23$e-$02$	 &	$2.10$ & 	 $4.26$e-$04$ &	$3.89$ 	& $4.85$e-$03$ &	$3.45$	  \\

$160$ &	$120$	& $5.52$e-$03$ &	$2.02$	& $5.52$e-$03$ &	$2.02$ & 	 $2.65$e-$05$ &	$4.01$  & $5.77$e-$04$ &	$3.07$	\\

$320$ &	$240$	& $1.38$e-$03$ &	$2.00$	&  $1.38$e-$03$ &	$2.00$ & 	 $1.96$e-$06$ &	$3.76$ & $7.22$e-$05$ &	$3.00$	\\
\hline
\hline
$N_x$ & $N_t$ & $L^1$ Err & Ord  & $L^1$ Err & Ord&	 $L^1$ Err & Ord &	 $L^1$ Err & Ord 	\\
\hline
$40$ &	$30$		&	$7.11$e-$02$	 &	 &  $6.90$e-$02$	&	& $7.97$e-$03$	&  	& $3.26$e-$02$  	&  	\\

$80$ &	$60$		&	$1.71$e-$02$	& 	$2.06$ & $1.71$e-$02$	& 	$2.02$ & $5.43$e-$04$	& 	$3.88$  	& $3.81$e-$03$	& 	$3.10$   \\

$160$ &	$120$	&	$4.18$e-$03$	& 	$2.03$ &$4.18$e-$03$	& 	$2.03$ & $3.48$e-$05$	& 	$3.96$  	& $4.51$e-$04$	&  $3.08$	\\

$320$ &	$240$	&	$1.04$e-$04$	& 	$2.01$ & $1.04$e-$04$	& 	$2.01$ & $2.14$e-$06$	& 	$4.03$  	 &  $5.52$e-$05$	& 	$3.03$	\\
\hline
\end{tabular}
\end{table}
\begin{table}[H] 
\caption{Test 5. CPU times in seconds. \label{ex1:table2}}
\centering
\begin{tabular}{c c| p{1.7cm}  | p{1.7cm} | p{1.7cm}  | c }
$N_x$ & $N_t$ & \centering {\bf  HC } & \centering {\bf AF-HC }& \centering {\bf AF-RKC4}& {\bf WENO 2/3}\\
\hline
\hline
$40$ &	$30$	&\centering  	$0.009\ s$ 	&\centering 	$0.052\ s$ 	&\centering  $0.057\ s$  &$0.119\ s$   \\

$80$ &	$60$ &\centering 	$0.049\ s$  &\centering 	$0.376\ s$   &\centering $0.446\ s$  	& $0.784\ s$    \\

$160$ &	$120$ &\centering  	$0.392\ s$  &\centering  	$3.097\ s$   &\centering $3.630\ s$ 	& $6.266\ s$ \\

$320$ &	$240$ &\centering 	$2.377\ s$  &\centering  	$26.28\ s$   &\centering  $28.52\ s$ 	& $51.65\ s$ 	\\
\hline
\end{tabular}
\end{table}
\begin{table}[h]
\caption{Test 5. Errors and orders in $L^\infty$ and  $L^1$ norms. \label{ex1:table3}}
\centering
\begin{tabular}{c c|c c |c c  | c c}
& & \multicolumn{2}{|c|}{\bf WENO 3/5-RK3 }&\multicolumn{2}{|c|}{\bf WENO 3/5-RK4}&\multicolumn{2}{|c}{\bf AF-RKC4}\\
\hline
\hline
$N_x$ & $N_t$ &  $L^\infty$ Err & Ord  & $L^\infty$ Err & Ord  & $L^\infty$ Err & Ord  \\
\hline
$40$ &	$30$		&	$1.00$e-$02$	&	&  	$9.10$e-$03$	& 	& $6.32$e-$03$ &  	\\

$80$ &	$60$		& 	$5.27$e-$04$ &	$4.25$ & 	 $5.34$e-$04$ &	$4.09$ 	& $4.26$e-$04$ &	$3.89$	  \\

$160$ &	$120$	& $2.74$e-$05$ &	$4.27$ & 	 $2.27$e-$05$ &	$4.55$  & $2.65$e-$05$ &	$4.01$	\\

$320$ &	$240$	&  $2.28$e-$06$ &	$3.59$ & 	$8.97$e-$07$ &	$4.66$ & $1.96$e-$06$ &	$3.76$	\\
\hline
\hline
$N_x$ & $N_t$ & $L^1$ Err & Ord&	 $L^1$ Err & Ord &	 $L^1$ Err & Ord 	\\
\hline
$40$ &	$30$		&  $1.19$e-$02$	&	& $1.15$e-$02$	&  	& $7.97$e-$03$  	&  	\\

$80$ &	$60$		& $5.65$e-$04$	& 	$4.40$ & $5.00$e-$04$	& 	$4.53$  	& $5.43$e-$04$	& 	$3.88$   \\

$160$ &	$120$	 &$2.72$e-$05$	& 	$4.38$ & $1.69$e-$05$	& 	$4.89$  	& $3.48$e-$05$	& 	 $3.96$	\\

$320$ &	$240$	 & $2.11$e-$06$	& 	$3.68$ & $5.80$e-$07$	& 	$4.86$  	 &  $2.14$e-$06$	& 	$4.03$	\\
\hline
\end{tabular}
\end{table}
\begin{table}[H] 
\caption{Test 5. CPU times in seconds. \label{ex1:table4}}
\centering
\begin{tabular}{c c| c  | c   |  c   }
$N_x$ & $N_t$ & {\bf  WENO 3/5-RK3 } &  {\bf WENO  3/5-RK4}&\hspace{0.3cm} {\bf AF-RKC4 \hspace*{0.3cm}}\\
\hline
\hline
$40$ &	$30$	 	&	$0.098\ s$ 	& $0.182\ s$  &$0.057\ s$   \\

$80$ &	$60$  &	$0.786\ s$   & $1.327\ s$  	& $0.446\ s$    \\

$160$ &	$120$   &	$6.555\ s$   & $10.32\ s$ 	& $3.630\ s$ \\

$320$ &	$240$   &  	$49.57\ s$   &  $81.23\ s$ 	& $28.52\ s$ 	\\
\hline
\end{tabular}
\end{table}
\noindent From Tab. \ref{ex1:table2} we can see that the AF-HC and the AF-RKC4 schemes increase considerably the computational cost of the simple HC scheme, due to the computation of the 2D-indicators which is the heavier procedure, but they are both faster than the WENO scheme, requiring almost half time for all refinements. This is because in our approach the computation of the indicators must be done once for each iteration, independently on the chosen high-order scheme, whereas for the WENO scheme it has to be repeated at each step of the Runge-Kutta integration. 
This imply that in our case the stencil of the indicators is always $5\times5$ points. 
This is confirmed by the slight increase in the computational cost of the AF-RKC4 scheme with respect to the AF-HC scheme, despite its much more involved definition. \\
For this regular case, we also compare the performance of our fourth-order  AF-RKC4 scheme with those of the two WENO 3/5 schemes. As visible in Tab. \ref{ex1:table3}, the two WENO 3/5 schemes implemented give comparable results with respect to our AF-RKC4, but from Tab. \ref{ex1:table4} we can note the advantages of the filtered scheme, which obtains comparable results in less time with respect to both WENO 3/5 schemes, especially refining the grid. 

\vspace{0.2cm}\noindent
{\bf Test 6.}
In this second test we consider a slightly more challenging situation, in which the hamiltonian depends also on the space variables, that is
\begin{equation}
\label{ex2:eq1}
\left\{
\begin{array}{l}
v_t-yv_x+xv_y=0\qquad \textrm{ in } (0,2\pi)\times\Omega, \\
v(0,x,y)=\max\left\{0,\frac{r_0-\left(x+1\right)^2-y^2}{r_0}\right\}^4,  \textrm{ in } \Omega,
\end{array}
\right.
\end{equation}
with $r_0=0.5$ and $\Omega=[-2.5,2.5]^2$. This problem models the rotation of a regular (but rather steep) function around the origin. The initial datum is a regularized paraboloid centered in $(-1,0)$, which is equal to $1$ in the center and $0$ on the circle of radius $\sqrt{r_0}$. The CFL number is taken as $\lambda=\frac{\pi}{16}\approx 0.196\leq \frac{1}{2}\max\{H_p^{-1},H_q^{-1}\}=0.2$. 
For this test, we consider the basic filtered scheme of \cite{BFS16} with the HC method as $S^A$, called in the following as F-HC, and the AF-HC, AF-RKC4 and WENO 2/3 schemes defined in the previous test. 
\begin{figure}[h!]
\centering
\includegraphics[width=0.4\textwidth]{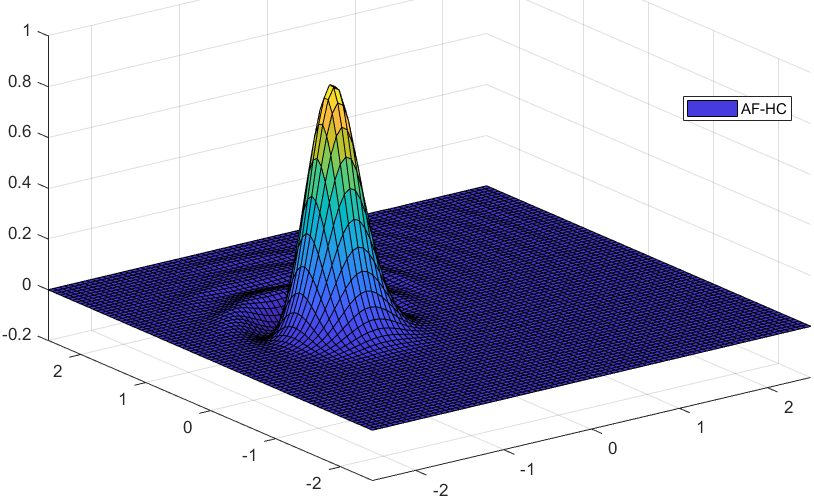}
\qquad
\includegraphics[width=0.4\textwidth]{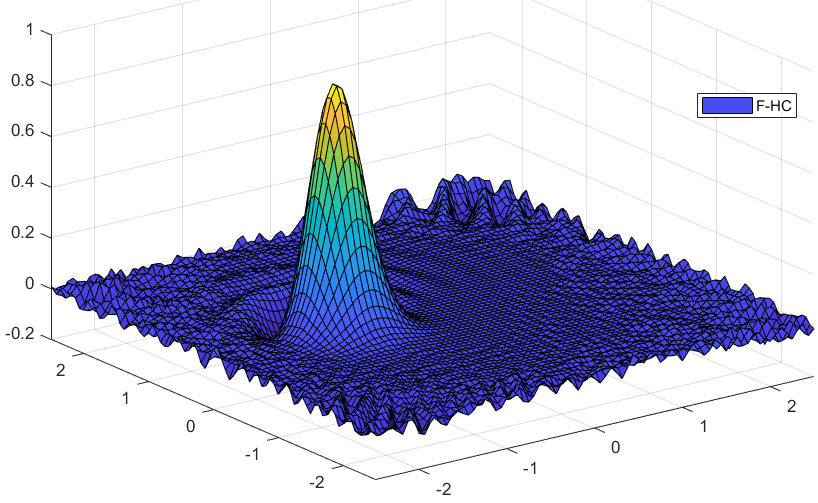}\\
\includegraphics[width=0.4\textwidth]{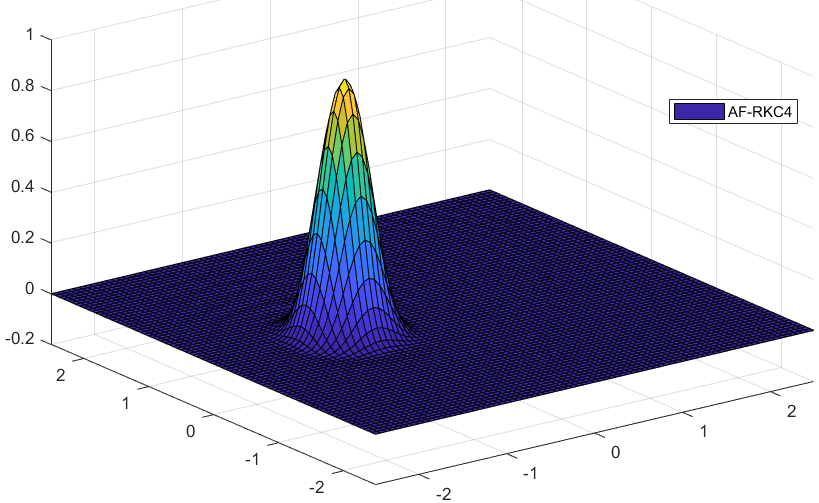}
\qquad
\includegraphics[width=0.4\textwidth]{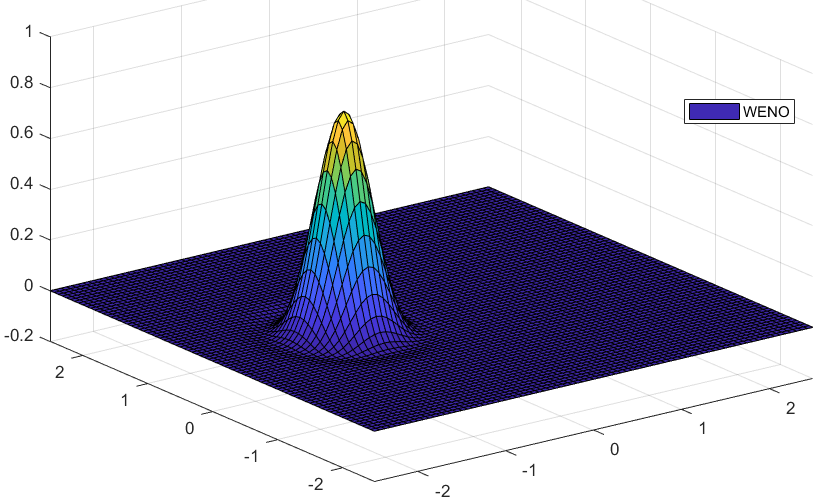}
\caption{\small{Test 6.} Plots of the computed solutions at time $T=2\pi$ with $\Delta x=0.05$. First row: AF-HC (left), F-HC (right). Second row: AF-RKC4 (left), WENO 2/3 (right).
\label{ex2:fig1}}
\end{figure}
\begin{table}[h!]
\caption{Test 6. Errors and orders in $L^\infty$ and  $L^1$ norms. \label{ex2:table1}}
\centering
\begin{tabular}{c c|c c |c c |c c  | c c}
& & \multicolumn{2}{|c|}{\bf F-HC ($20\Delta x$) }&\multicolumn{2}{|c|}{\bf AF-HC }&\multicolumn{2}{|c|}{\bf AF-RKC4}&\multicolumn{2}{|c}{\bf WENO 2/3}\\
\hline
\hline
$N_x$ & $N_t$ & $L^\infty$ Err & Ord & $L^\infty$ Err & Ord  & $L^\infty$ Err & Ord  & $L^\infty$ Err & Ord  \\
\hline
$20$ &	$128$	& $9.24$e-$01$	& 	&	$8.31$e-$01$	&	&  	$6.36$e-$01$	& 	& $7.73$e-$01$ &  	\\

$40$ &	$256$	& $5.88$e-$01$ &	$0.65$	& $6.27$e-$01$ &	$0.41$ & 	$4.09$e-$01$ &	$0.64$ & $4.98$e-$01$ &	$0.64$	 \\

$80$ &	$512$	& $3.70$e-$01$ &	$0.67$	& $3.71$e-$01$ &	$0.76$ & 	$2.74$e-$02$ &	$3.90$  & $1.39$e-$01$ &	$1.84$	\\

$160$ &	$1024$	& $1.02$e-$01$ &	$1.86$	& $1.01$e-$01$ &	$1.88$ & 	 $1.63$e-$03$ &	$4.07$ & $1.38$e-$02$ &	$3.34$	\\
\hline
\hline
$N_x$ & $N_t$ & $L^1$ Err & Ord  & $L^1$ Err & Ord&	 $L^1$ Err & Ord &	 $L^1$ Err & Ord 	\\
\hline
$20$ &	$128$	& $1.47$e+$00$ &	 &  $9.25$e-$01$	&	& $4.66$e-$01$	&  	& $3.85$e-$01$  	&  	\\

$40$ &	$256$	& $7.00$e-$01$	& 	$1.07$ 	& $4.29$e-$01$	& 	$1.11$ & $1.58$e-$01$	& 	$1.56$  	& $1.92$e-$01$	& 	$ 1.00$   \\

$80$ &	$512$	& $3.56$e-$01$	& 	$0.98$	 & $1.70$e-$01$	& 	$1.33$ & $1.83$e-$02$	& 	$3.12$  & $5.24$e-$02$	& 	$1.87$	\\

$160$ &	$1024$	& $1.08$e-$01$	& 	$1.72$ 	& $5.02$e-$02$	& 	$1.76$ & $ 1.20$e-$03$	& 	$3.93$  &  $7.38$e-$03$	& 	$2.83$	\\
\hline
\end{tabular}
\end{table}
\begin{table}[h!] 
\caption{Test 6. CPU times in seconds. \label{ex2:table2}}
\centering
\begin{tabular}{c c| p{1.7cm}  | p{1.7cm} | p{1.7cm}  | c }
$N_x$ & $N_t$ & \centering {\bf  F-HC } & \centering {\bf AF-HC }& \centering {\bf AF-RKC4}& {\bf WENO 2/3}\\
\hline
\hline
$20$ &	$128$ &\centering 	$0.009\ s$  &\centering 	$0.048\ s$   &\centering $0.057\ s$ 	& $0.124\ s$   \\

$40$ &	$256$ &\centering 	$0.052\ s$  &\centering 	$0.394\ s$   &\centering $0.457\ s$  	& $0.845\ s$    \\

$80$ &	$512$ &\centering 	$0.365\ s$   &\centering  	$3.348\ s$   &\centering $3.793\ s$   	& $6.779\ s$ \\

$160$ &	$1024$ &\centering 	$2.806\ s$  &\centering  	$26.20\ s$   &\centering $30.40\ s$   	 & $55.87\ s$ 	\\
\hline
\end{tabular}
\end{table}

Looking at Fig. \ref{ex2:fig1} we can clearly see the advantages provided by the automatic tuning of the parameter $\varepsilon^n$ and the stabilizing properties of the $\phi$ function. In fact, our AF schemes, especially AF-RKC4, are able to almost completely nullify the oscillations caused by the unstable HC scheme, whereas the simple F-HC scheme with $\varepsilon=20\Delta x$ is not able to do that, producing significant visible oscillations. Moreover, for the F-HC scheme the oscillations keep on being amplified as time goes on, reducing the effective accuracy of the scheme. Despite this qualitative graphically evident improvement, the errors and orders of the filtered schemes that use the HC method are rather close (see Tab.  \ref{ex2:table1}), with the basic F-HC scheme obtaining equal or slightly different errors in $L^\infty$, but losing evidently with respect to the adaptive version AF-HC in $L^1$ norm, which gives a better measure of the overall approximation. 
Looking at the results regarding the WENO 2/3 scheme in the same Tab. \ref{ex2:table1}, we can note that it performs better than the second order filtered schemes in both norms, in terms of orders and errors, as it could be expected from the regularity of the evolving function, but it performs worse when compared to the fourth order scheme AF-RKC4 in terms of errors and resolution. This is visible also from Fig. \ref{ex2:fig1}, second row, since the WENO 2/3 scheme presents wider oscillations and flattens  
more the profile of the solution with respect to the AF-RKC4. 
Regarding the CPU times, looking at Tab. \ref{ex2:table2} we can see that the basic filtered scheme is much faster than the other proposed schemes, requiring computational times similar to the simple high-order scheme. This fact is rather natural, since it does not require any computation of smoothness indicators. Comparing the AF schemes with respect to WENO 2/3, we have a confirmation of what already noted in the previous example, with almost half-time of AF schemes with respect to WENO 2/3. 

\vspace{0.2cm}\noindent
{\bf Test 7.}
In this test,  we consider the following eikonal equation in two dimensions with constant velocity
\begin{equation}
\label{ex3:eq1}
\left\{
\begin{array}{l}
v_t+\sqrt{v_x^2+v_y^2}=0\qquad \textrm{ in } (0,T)\times\Omega, \\
v(0,x,y)=v_0(x,y),
\end{array}
\right.
\end{equation}
where $\Omega=[-3,3]^2$. 
This equation appears in front propagation problems through the level set method (see \cite{S85,OS88} for details). 
Here we focus on a simple expansion with constant velocity in the case of a merging of two separate fronts and we compare the same schemes used in Test 6 in terms of  error and resolution of the $0$-level set, varying the regularity of the representation function $v_0$. 
The CFL number is set to $\lambda=0.25<\frac{1}{2}$ for both cases. In the first case (Case a) we consider two collapsing regular representations, with the $0$-level set composed by two circles, that is 
\begin{equation}
v_0(x,y)=0.5-0.5\max\left(\max(0,f_-)^4,\max(0,f_+)^4\right), 
\end{equation}
with
\begin{equation*}
f_\pm=\frac{1-\left(x\pm\frac{\sqrt{2}}{2}\right)^2-\left(y\pm\frac{\sqrt{2}}{2}\right)^2}{1-r_0^2},  \qquad  r_0=0.5.
\end{equation*}
In the second case (Case b), we consider the evolution and then the merging of a sharper representation function, which initial $0$-level set is composed by two squares, i.e. 
\begin{align}
v_0(x,y)=\min\left\{f_1-r_0,f_2-r_0,\frac{1}{2}r_0^2\right\}, \quad \textrm{with}\quad f_1=\max\left\{\left|x- \frac{\sqrt{2}}{2}\right|,\left|y- \frac{\sqrt{2}}{2}\right|\right\},\\
f_2=\max\left\{\left|\left(\sqrt{r_0}x+\frac{\sqrt{2}}{2}\right)+\left(\sqrt{r_0}y+\frac{\sqrt{2}}{2}\right)\right|,\left|\left(\sqrt{r_0}x+\frac{\sqrt{2}}{2}\right)-\left(\sqrt{r_0}y+\frac{\sqrt{2}}{2}\right)\right|\right\}, \nonumber
\end{align}
where now $r_0=0.5$ is a parameter needed to control magnitude of the square $0$-level front. 
For \emph{Case a} the final time is set to $T=0.6$, whereas for \emph{Case b} the solution is computed at $T=0.7$, in order to have the two fronts merge. 

Our goal is to inspect the behavior of the schemes when varying the ``number of singularities'' in the evolution. One is directly provided  by the hamiltonian, since it is only Lipschitz continuous and presents a saddle point in the origin, others may be already present in the initial datum or caused by some merging. 
In this context, looking at Tab. \ref{ex3a:table1} for Case a, we can see that the better results are given by AF-RKC4 scheme in $L^1$ norm in most situations, even with respect to the  WENO 2/3 results, whereas the more compact AF-HC scheme gives the best results in $L^\infty$ among the three filtered schemes, but still worse compared to the WENO 2/3 ones, which results more accurate. This is probably due to the different stencils used by the schemes: the lower order schemes, F-HC and AF-HC, use more compact stencils ($3\times3$ points, instead of $17\times 17$ points needed to the AF-RKC4 scheme), whereas the WENO 2/3 procedure uses an adaptive stencil, although pretty wide. This implies that the fourth order AF-RKC4 scheme suffers more deeply the presence of singularities, especially in $L^\infty$ norm. 
\begin{table}[h!]
\caption{Test 7a. Errors and orders in $L^\infty$ and  $L^1$ norms. \label{ex3a:table1}}
\centering
\begin{tabular}{c c|c c|c c|c c| c c}
& & \multicolumn{2}{|c|}{\bf F-HC ($20\Delta x$)}&\multicolumn{2}{|c|}{\bf AF-HC}&\multicolumn{2}{|c|}{\bf AF-RKC4}&\multicolumn{2}{|c}{\bf WENO 2/3}\\
\hline
\hline
$N_x$ & $N_t$ &$L^\infty$ Err &Ord  & $L^\infty$ Err & Ord& $L^\infty$ Err & Ord & $L^\infty$ Err & Ord \\
\hline
$30$ &	$12$	& $2.31$e-$01$	&	&	$2.08$e-$01$	& &$1.96$e-$01$ & &$2.03$e-$01$			\\

$60$ &	$24$	& $8.84$e-$02$ &	$1.38$ & $6.45$e-$02$ &	$1.69$ &$1.06$e-$01$ &	$0.88$    & $5.38$e-$02$ &	$1.92$ \\

$120$ &	$48$	& $5.42$e-$02$ &	$0.71$ & $4.93$e-$02$ &	$0.39$ & $6.02$e-$02$ &	$0.82$	& $2.84$e-$02$ &	$0.92$ \\

$240$ &	$96$	& $4.71$e-$02$ &	$0.20$ & $1.78$e-$02$ &	$1.47$ & $3.34$e-$02$ &	$0.85$   & $1.80$e-$02$ &	$0.66$\\
\hline
\hline
$N_x$ & $N_t$ &$L^1$ Err &Ord  & $L^1$ Err & Ord& $L^1$ Err & Ord & $L^1$ Err & Ord \\
\hline
$30$ &	$12$ &  $1.29$e+$00$  &	&	$1.20$e+$00$	& & $9.25$e-$01$	& & $1.05$e+$00$\\

$60$ &	$24$ & $4.07$e-$01$	& 	$1.67$ &$3.41$e-$01$	& 	$1.82$  & $1.57$e-$01$	& 	$2.55$   & $2.13$e-$01$	& 	$2.31$\\

$120$ &	$48$ & $1.27$e-$01$	& 	$1.67$  & $8.78$e-$02$	& 	$1.96$ & $5.68$e-$02$	& 	$1.47$	& $5.85$e-$02$	& 	$1.86$ \\

$240$ &	$96$ & $5.24$e-$02$	& 	$1.28$ & $3.35$e-$02$	& 	$1.39$	&  $3.39$e-$02$	& 	$0.75$  &$3.21$e-$02$	& 	$0.86$\\
\hline
\end{tabular}
\vspace{-0.1cm}
\end{table}
\begin{table}[h!]
\caption{Test 7a. CPU times in seconds. \label{ex3a:table2}}
\centering
\begin{tabular}{c c| p{1.7cm}  | p{1.7Cm} | p{1.7cm}  | c }$N_x$ & $N_t$ & \centering {\bf  F-HC } & \centering {\bf AF-HC }& \centering {\bf AF-RKC4}& {\bf WENO 2/3}\\
\hline
\hline
$30$ &	$12$	&\centering 	$0.005\ s$ 	&\centering 	$0.014\ s$ 	&\centering $0.017\ s$  &$0.027\ s$   \\

$60$ &	$24$ &\centering 	$0.020\ s$  &\centering 	$0.109\ s$   &\centering $0.129\ s$  	& $0.204\ s$    \\

$120$ &	$48$ &\centering 	$0.119\ s$  &\centering  	$0.851\ s$   &\centering $0.958\ s$   	& $1.475\ s$ \\

$240$ &	$96$ &\centering 	$0.814\ s$  &\centering  	$6.674\ s$   &\centering $7.383\ s$   	 & $12.21\ s$ 	\\
\hline
\end{tabular}
\end{table}
\begin{figure}[h!]
	\centering
	\includegraphics[width=0.29\textwidth]{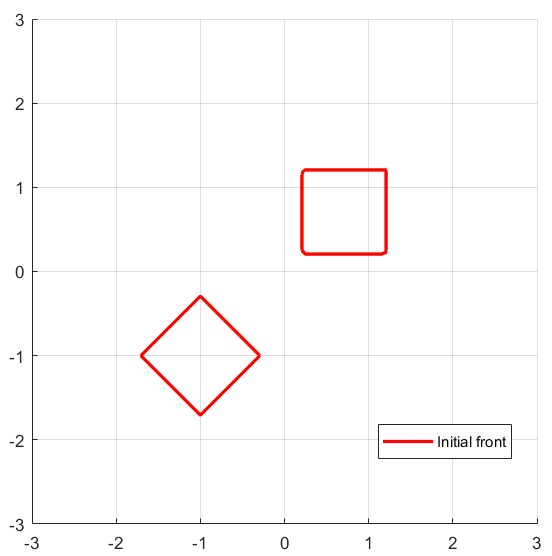}
	\qquad\quad
	\includegraphics[width=0.29\textwidth]{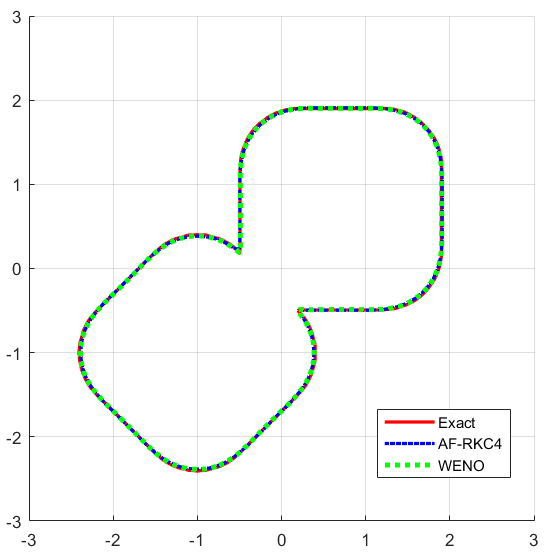}
	\caption{\small{Test 7b. Initial front (left) and fronts at $T=0.7$ using WENO 2/3 and AF-RKC4 schemes (right).}\label{ex3b:fig1}}
\end{figure}
\begin{figure}[h!]
	\centering
	\includegraphics[width=0.32\textwidth]{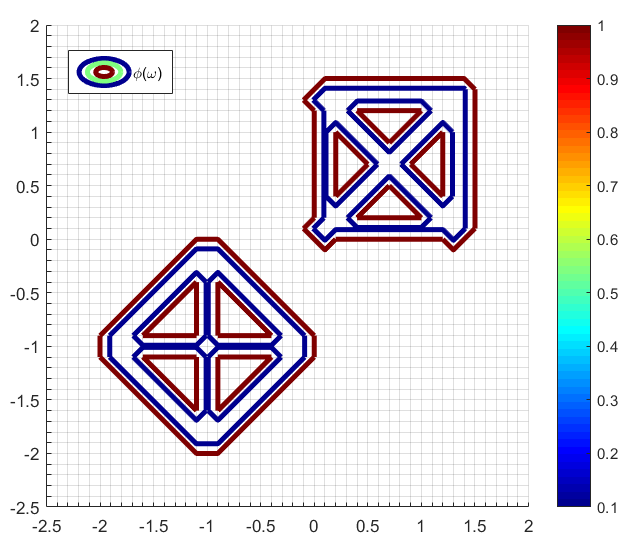}
	\quad\quad
	\includegraphics[width=0.32\textwidth]{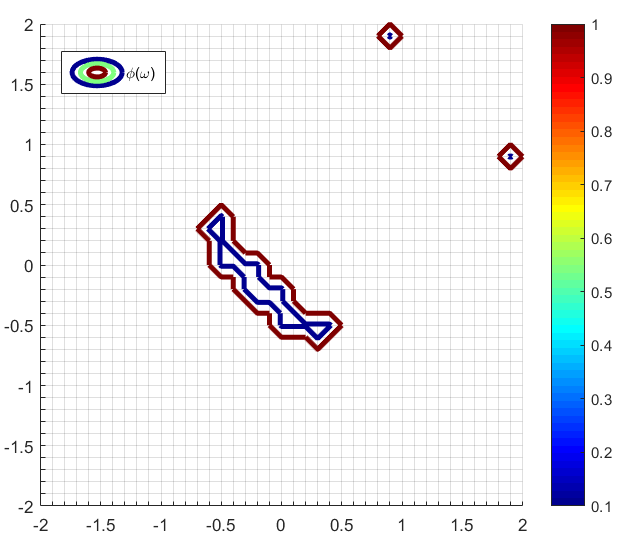}
	\caption{\small{Test 7b. Contour plots of the function $\phi$ at the initial (left) and final (right) time, using the smoothness indicators $g(\omega^F_{2D})$, for $\Delta x=\Delta y=0.1$.
		} \label{ex3b:fig3}}
\end{figure}

This behavior is more evident in Case b, in which the singularities are already present in the initial data, as visible in Fig. \ref{ex3b:fig1}, in which we report also the merging through the $0$-level sets for the WENO 2/3 and the AF-RKC4 schemes.  
Looking at Tab. \ref{ex3b:table1}, we can note that the AF-RKC4 is able to keep the sharpness of the edges also with respect to the WENO 2/3 scheme, although the latter performs better in both norms. 
Moreover, the stabilizing properties of the adaptive filtering, provided by the indicator function $\phi$ visible in Fig. \ref{ex3b:fig3}, are visible. 
In fact, the AF-HC prevents the oscillating behavior (of the numerical order of convergence) of its non-adaptive version with $\varepsilon=20\Delta x$, although in this particular situation it produces often bigger errors. 
Finally, in Fig. \ref{ex3b:fig3} we present the results given by the smoothness indicator $g(\omega_{2D}^F)$ used in the simulation. It is evident that our final choice of the indicators is able to localize the position of the singularities in the initial condition with good precision, and to detect clearly the merging area of the representation functions, also using a rather coarse grid. 
As last remark, looking at Tabs. \ref{ex3a:table2} and \ref{ex3b:table2}, note that in both cases the CPU times of the filtering schemes are lower than those of the WENO 2/3 scheme, especially for finer grids. 
\begin{table}[h!]
\caption{Test 7b. Errors and orders in $L^\infty$ and  $L^1$ norms. \label{ex3b:table1}}
\centering
\begin{tabular}{c c|c c |c c |c c  | c c}
& & \multicolumn{2}{|c|}{\bf F-HC ($20\Delta x$) }&\multicolumn{2}{|c|}{\bf AF-HC }&\multicolumn{2}{|c|}{\bf AF-RKC4}&\multicolumn{2}{|c}{\bf WENO 2/3}\\
\hline
\hline
$N_x$ & $N_t$ & $L^\infty$ Err & Ord & $L^\infty$ Err & Ord  & $L^\infty$ Err & Ord  & $L^\infty$ Err & Ord  \\
\hline
$30$ &	$14$	& $1.64$e-$01$	& 	&	$1.06$e-$01$	&	&  	$1.06$e-$01$	& 	& $7.79$e-$02$ &  	\\

$60$ &	$28$	& $6.93$e-$02$ &	$1.24$	& $8.35$e-$02$ &	$0.35$ & 	 $7.31$e-$02$	 &	$0.54$ & $4.59$e-$02$ &	$0.76$	 \\

$120$ &	$56$	& $6.73$e-$02$ &	$0.04$	& $4.76$e-$02$ &	$0.81$ & 	 $4.24$e-$02$	&	$0.79$  & $2.50$e-$02$ &	$0.88$	\\

$240$ &	$112$	& $8.57$e-$02$ &	$-0.35$	& $2.79$e-$02$ &	$0.77$ & 	 $2.41$e-$02$ &	$0.82$ & $1.72$e-$02$ &	$0.54$	\\
\hline
\hline
$N_x$ & $N_t$ & $L^1$ Err & Ord  & $L^1$ Err & Ord&	 $L^1$ Err & Ord &	 $L^1$ Err & Ord 	\\
\hline
$30$ &	$14$	&$6.11$e-$01$	 &	 &  $8.33$e-$01$	&	& $5.20$e-$01$	&  	& $5.38$e-$01$  	&  	\\

$60$ &	$28$	& $3.18$e-$01$	& 	$0.94$ 	& $3.98$e-$01$	& 	$1.07$ & $3.60$e-$01$	& 	$0.53$ 	& $2.01$e-$01$	& 	$1.42$   \\

$120$ &	$56$	& $1.65$e-$01$	& 	$0.95$	 & $1.89$e-$01$	& 	$1.08$ &	$1.73$e-$01$	& 	$1.05$  & $8.31$e-$02$	& 	$1.28$	\\

$240$ &	$112$	& $1.22$e-$01$	& 	$0.43$ 	& $8.48$e-$02$	& $1.15$ & 	$7.57$e-$02$	& 	$1.19$  &  $4.59$e-$02$	& 	$0.85$	\\
\hline
\end{tabular}
\end{table}
\begin{table}[h!]
\caption{Test 7b. CPU times in seconds. \label{ex3b:table2}}
\centering
\begin{tabular}{c c| p{1.7cm}  | p{1.7Cm} | p{1.7cm}  | c }
$N_x$ & $N_t$ & \centering {\bf  F-HC } & \centering {\bf AF-HC }& \centering {\bf AF-RKC4}& {\bf WENO 2/3}\\
\hline
\hline
$30$ &	$14$&\centering 	$0.006\ s$ 	&\centering 	$0.019\ s$ 	&\centering $0.021\ s$  &$0.031\ s$   \\

$60$ &	$28$ &\centering 	$0.030\ s$  &\centering 	$0.140\ s$   &\centering $0.159\ s$  	& $0.233\ s$    \\

$120$ &	$56$ &\centering 	$0.176\ s$  &\centering  	$1.038\ s$   &\centering $1.156\ s$   	& $1.747\ s$ \\

$240$ &	$112$ &\centering 	$0.924\ s$  &\centering  	$7.932\ s$   &\centering $8.984\ s$   	 & $13.58\ s$ 	\\
\hline
\end{tabular}
\end{table}

\noindent It is important to point out that for producing Fig. \ref{ex3b:fig3} a ``fix'' has been needed to obtain the correct outcome by the indicators. That is because in the critical case in which two singularity curves intersect in one point, as for the two pyramids considered in Case b, it can happen that all the eight cases for formula \eqref{beta_zeta} are such that $\beta_k^{\zeta_1\zeta_2}=O(1)$. This implies that the information given by the function $\phi$ can not really be trusted. In order to solve the problem it is enough to consider the information contained in the function $\phi$ computed on the eight points of the stencil $\mathcal S_0$ around the considered point $(x_j,y_i)$ and to control if there are at least two consecutive values different from $0$. In this way we can be sure that there is at least one sector which is not crossed by any singularity curve and the resulting information can be used. Another possibility, slightly less formal, is to control that at least one of the values of $\phi$ on the corners of the cell $I_{i,j}$ is equal to $1$. 

\vspace{0.2cm}\noindent
{\bf Test 8.} 
In this last test we consider a problem similar to the Burgers' equation in two dimensions,
\begin{equation}
\label{ex4:eq1}
\left\{
\begin{array}{l}
v_t+(v_x+1)^2+(v_y+1)^2=0\qquad \textrm{ in } (0,T)\times\Omega, \\
v(0,x,y)=-0.5\left(\cos(\pi x)+\cos(\pi y)\right),
\end{array}
\right.
\end{equation}
with $\Omega=[0,2]^2$ and periodic boundary conditions. For that problem (\ref{ex4:eq1}), we consider the final time $T=\frac{3}{4\pi^2}$, when the solution is still smooth, and then $T=\frac{3}{2\pi^2}$, time at which an interesting set of singularities develops. We use the same schemes of the previous example and a slightly more restrictive CFL number in order to use coarser grids, which is set to $\lambda=\frac{3}{4\pi^2}\approx 0.076$ for both tests.
The exact solution is computed by the \emph{Hopf-Lax} formula and is as
\begin{equation}
v(t,x,y)=\left(\min_{a\in A} -\frac{1}{2}\cos(\pi(x-at))+\frac{1}{4}a^2-a + \min_{b\in A} -\frac{1}{2}\cos(\pi(y-bt))+\frac{1}{4}b^2-b\right),
\end{equation}
with $A=[-6,6]$, where we used the fact that the evolution can be seen as the sum of one-dimensional evolutions. 

This test summarizes all the behaviors already seen in the previous cases. In fact, 
if the solution is still regular (see Tab. \ref{ex4:table1}) the AF-RKC4 scheme gives the best results and achieves the optimal order in both norms, whereas when the singularity appears (see Tab. \ref{ex4:table2}) it has the usual problems in the $L^\infty$ errors, but better orders in the $L^1$ norm with respect to the second order filtered schemes. Regarding this second case, note that here the WENO 2/3 scheme gets the best results in terms of both errors and orders with respect to all the three filtered schemes. 
\begin{table}[h!]
	\caption{Test 8. $T=3/(4\pi^2)$. Errors and orders in $L^\infty$ and  $L^1$ norms. \label{ex4:table1}}
	\centering
	\begin{tabular}{c c|c c|c c|c c| c c}
		& & \multicolumn{2}{|c|}{\bf F-HC ($20\Delta x$) }&\multicolumn{2}{|c|}{\bf AF-HC }&\multicolumn{2}{|c|}{\bf AF-RKC4}&\multicolumn{2}{|c}{\bf WENO 2/3}\\
		\hline
		\hline
		$N_x$ & $N_t$ &$L^\infty$ Err &Ord  & $L^\infty$ Err & Ord& $L^\infty$ Err & Ord & $L^\infty$ Err & Ord \\
		\hline
		$20$ &	$10$	& $4.66$e-$02$	&	&	$3.22$e-$02$	& & $4.22$e-$02$ & & $1.13$e-$02$	\\
		
		$40$ &	$20$	& $2.40$e-$02$ &	$0.96$  & $9.44$e-$03$ &	$1.77$ & $1.65$e-$03$ &	$4.68$    & $2.57$e-$03$ &	$2.14$ \\
		
		$80$ &	$40$	& $1.38$e-$02$ &	$0.79$ & $2.54$e-$03$ &	$1.89$ & $1.73$e-$04$ &	$3.25$	& $4.24$e-$04$ &	$2.60$  \\
		
		$160$ &	$80$	& $7.21$e-$03$ &	$0.94$ & $6.20$e-$04$ &	$2.04$	& $1.19$e-$05$ &	$3.86$  & $6.30$e-$05$ &	$2.75$\\
		\hline
		\hline
		$N_x$ & $N_t$ &$L^1$ Err &Ord  & $L^1$ Err & Ord& $L^1$ Err & Ord & $L^1$ Err & Ord \\
		\hline
		$20$ &	$10$ & $2.66$e-$02$  &	&	$2.36$e-$02$	& & $1.22$e-$02$	& &  $6.66$e-$03$\\
		
		$40$ &	$20$ & $6.62$e-$03$	& 	$2.01$ & $6.45$e-$03$	& 	$1.87$  & $2.69$e-$04$	& 	$5.50$   & $8.76$e-$04$	& 	$2.93$ \\
		
		$80$ &	$40$ & $1.81$e-$03$	& 	$1.87$  & $1.66$e-$03$	& 	$1.96$ &$2.00$e-$05$	& 	$3.75$	& $1.20$e-$04$	& 	$2.86$ \\
		
		$160$ &	$80$ & $5.15$e-$04$	& 	$1.81$ & $4.15$e-$04$	& 	$2.00$	&  $1.43$e-$06$	& 	$3.81$  & $1.56$e-$05$	& 	$2.95$ \\
		\hline
	\end{tabular}
\end{table}
\begin{table}[h!]
	\caption{Test 8.$T=3/(2\pi^2)$. Errors and orders in $L^\infty$ and  $L^1$ norms. \label{ex4:table2}}
	\centering
	\hspace*{-0.3cm}
	\begin{tabular}{c c|c c|c c|c c| c c}
		& & \multicolumn{2}{|c|}{\bf F-HC ($20\Delta x$) }&\multicolumn{2}{|c|}{\bf AF-HC }&\multicolumn{2}{|c|}{\bf AF-RKC4}&\multicolumn{2}{|c}{\bf WENO 2/3}\\
		\hline
		\hline
		$N_x$ & $N_t$ &$L^\infty$ Err &Ord  & $L^\infty$ Err & Ord& $L^\infty$ Err & Ord & $L^\infty$ Err & Ord \\
		\hline
		$20$ & $20$ & $1.12$e-$01$	&	&	$1.30$e-$01$	& & $1.51$e-$01$ & & $5.81$e-$02$	\\
		
		$40$ & $40$ &$6.03$e-$02$ &	$0.89$ & $6.65$e-$02$ &	$0.96$ &$7.22$e-$02$ &	$1.06$    & $2.37$e-$02$ &	$1.29$ \\
		
		$80$ & $80$ & $2.82$e-$02$ &	$1.10$ & $2.95$e-$02$ &	$1.17$ &  $3.13$e-$02$ &	$1.21$	& $8.07$e-$03$ &	$1.55$  \\
		
		$160$ & $160$ &$1.10$e-$02$ &	$1.36$ & $1.12$e-$02$ &	$1.40$ & $1.17$e-$02$ &	$1.41$   & $1.91$e-$03$ &	$2.08$\\
		\hline
		\hline
		$N_x$ & $N_t$ &$L^1$ Err &Ord  & $L^1$ Err & Ord& $L^1$ Err & Ord & $L^1$ Err & Ord \\
		\hline
		$20$ & $20$ &  $4.17$e-$02$  &	&	$4.37$e-$02$	& & $3.71$e-$02$	& &  $2.10$e-$02$\\
		
		$40$ & $40$ & $1.17$e-$02$	& 	$1.84$ & $1.04$e-$02$	& 	$2.07$  & $8.94$e-$03$	& 	$2.05$   & $4.28$e-$03$	& 	$2.30$\\
		
		$80$ & $80$ & $2.84$e-$03$	& 	$2.04$  & $2.56$e-$03$	& 	$2.02$ & $1.97$e-$03$	& 	$2.19$	& $8.49$e-$04$	& 	$2.33$ \\
		
		$160$ & $160$ & $6.41$e-$04$	& 	$2.15$ & $6.05$e-$04$	& 	$2.08$	& $4.03$e-$04$	& 	$2.29$  & $1.28$e-$04$	& 	$2.73$\\
		\hline
	\end{tabular}
\end{table}
\begin{figure}[h!]
\centering
\includegraphics[width=0.3\textwidth]{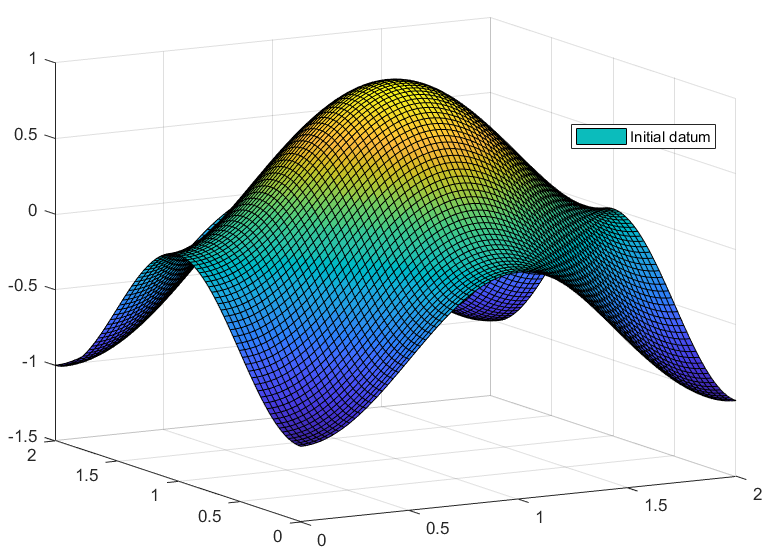}
\qquad
\includegraphics[width=0.3\textwidth]{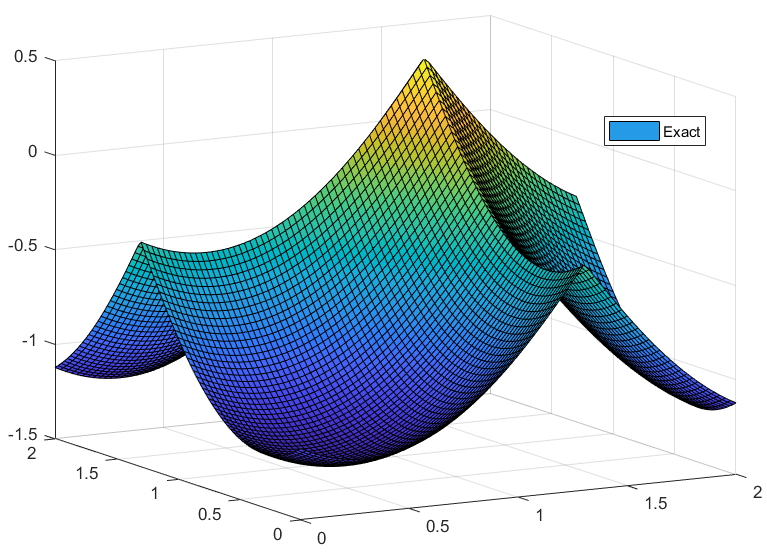}\\
\includegraphics[width=0.3\textwidth]{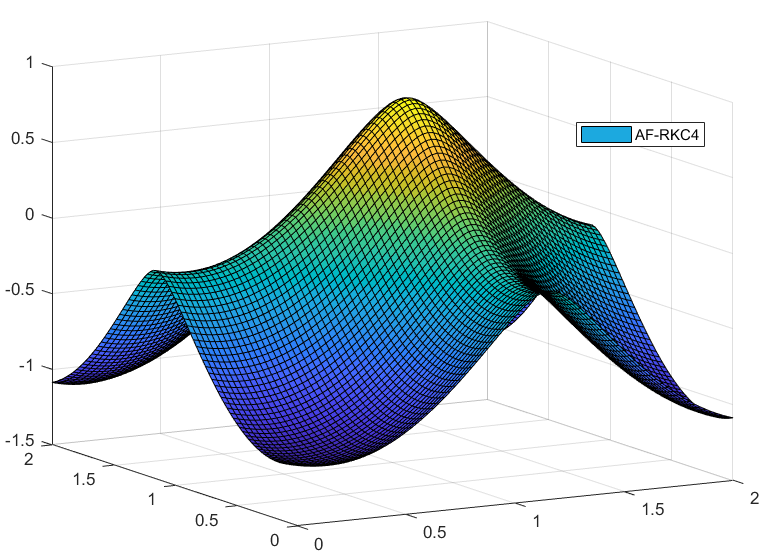}
\qquad
\includegraphics[width=0.3\textwidth]{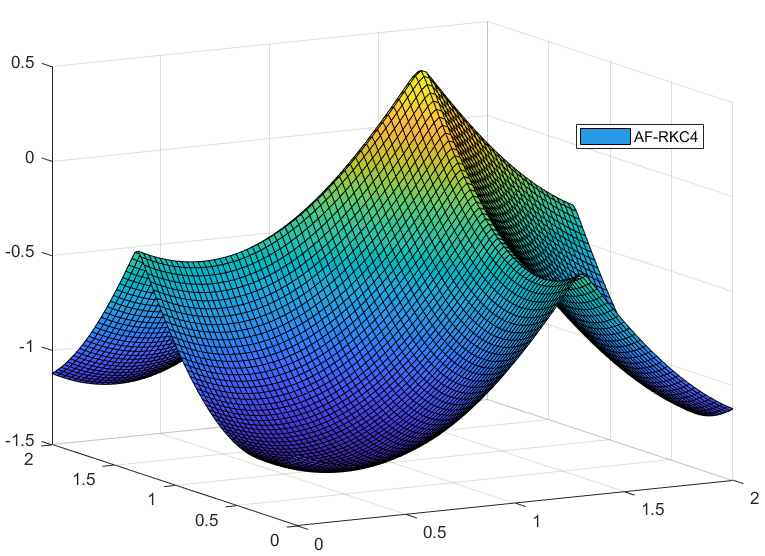}
\caption{\small{Test 8.} Top: Initial data (left) and exact solution at $T=3/(2\pi^2)$ (right). Bottom: solution at $T=3/(4\pi^2)$ (left) and $T=3/(2\pi^2)$ (right) computed by the AF-RKC4 scheme with $\Delta x=\Delta y=0.025$. \label{ex4:fig1}}
\vspace*{-0.2cm}
\end{figure}
\begin{figure}[h!]
\centering
\begin{subfigure}{0.3\textwidth}
{\includegraphics[width=\textwidth]{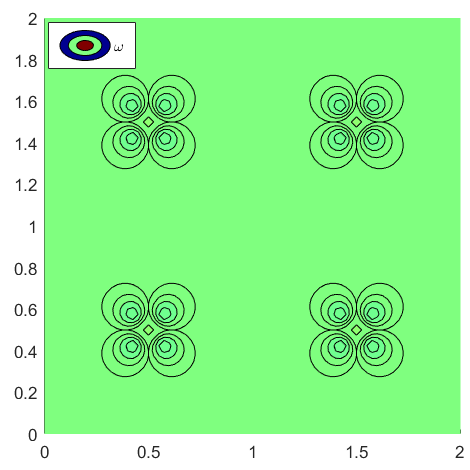}}
\end{subfigure}
\begin{subfigure}{0.3\textwidth}
{\includegraphics[width=\textwidth]{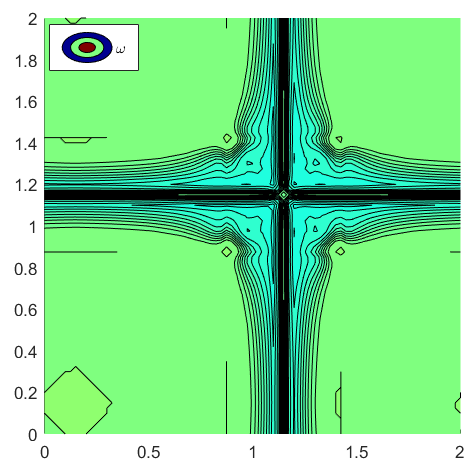}}
\end{subfigure}
\begin{subfigure}{0.36\textwidth}
{\includegraphics[width=\textwidth]{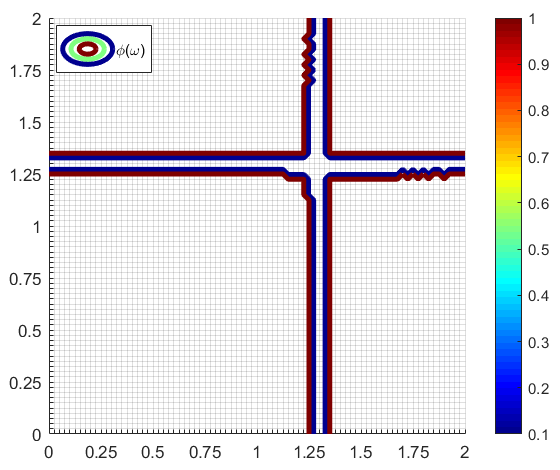}}
\end{subfigure}
\vspace{-0.1cm}
\caption{\footnotesize{ Test 8. Results of the smoothness indicators (obtained using $g(\omega^F_{2D})$), for $\Delta x=\Delta y=0.025$. Plots at initial time (left), at time $T=3/(4\pi^2)$, when the solution is still smooth (middle), and at the final time $T=3/(2\pi^2)$, when the singularity is fully developed (right).}\label{ex4:fig2}}
\end{figure}
\begin{figure}[h!]
\centering
\begin{subfigure}{0.32\textwidth}
{\includegraphics[width=\textwidth]{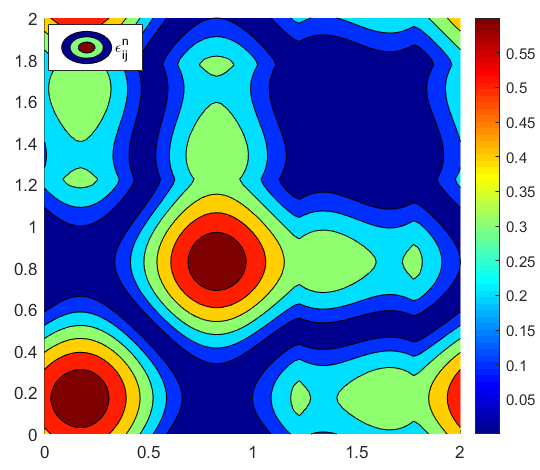}}
\end{subfigure}
\begin{subfigure}{0.32\textwidth}
{\includegraphics[width=\textwidth]{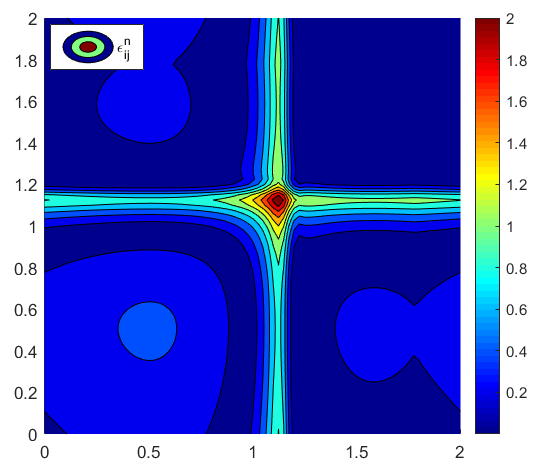}}
\end{subfigure}
\begin{subfigure}{0.32\textwidth}
{\includegraphics[width=\textwidth]{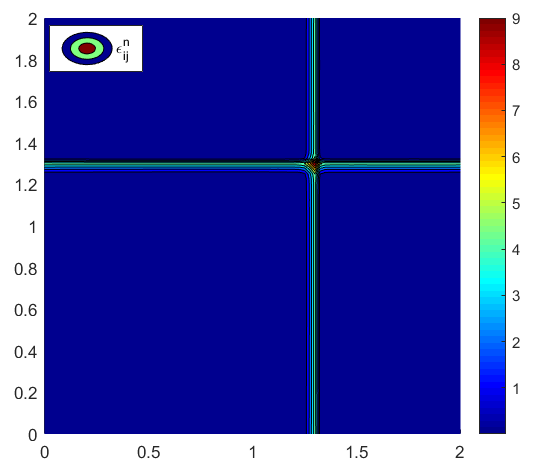}}
\end{subfigure}
\vspace{-0.1cm}
\caption{\footnotesize{Test 8. Results of $\varepsilon^n_{ij}$ (obtained using formula \eqref{eps_2D}, without computing the maximum on $\mathcal R^n$), for $\Delta x=\Delta y=0.025$. Plots at initial time (left), at time $T=3/(4\pi^2)$, when the solution is still smooth (middle), and at the final time $T=3/(2\pi^2)$, when the singularity is fully developed (right).}\label{ex4:fig3}}
\end{figure}
Moreover, we can clearly see that the basic filtered scheme depends heavily on the choice of $\varepsilon$. In fact, choosing $\varepsilon=20\Delta x$ leads to very good results in the singular case, whereas it has clear problems at exploiting the full accuracy of the high-order scheme when the solution is still regular. This is the main advantage of the adaptive $\varepsilon^n$, which is able to tune itself depending on the local (in time) regularity of the solution, as shown in Fig. \ref{ex4:fig3}. There we reported the results obtained by formula \eqref{eps_2D} without computing the maximum over $\mathcal R^n$, in the case of the initial condition and the two considered final times. We can notice that the value of $\varepsilon^n$ can consistently vary when the solution is regular, whereas it drops when the singularity appears, since all the high values are concentrated inside the region detected by the function $\phi$. 
In Fig. \ref{ex4:fig2} we have collected some results given by the smoothness indicators using the AF-RKC4 scheme with  $\Delta x=0.025$, which testify the reliability of our definition. The indicators are able to recognize the regularity of the solution in the first two cases (the $\phi$ function is identically equal to $1$) and to precisely locate the singularity after its development (right). Note that the formation of a singularity is already visible in the behavior of $g(\omega^F_{2D})$ (middle). 
The conclusions of these last two examples are rather promising, testifying the good properties of the Adaptive Filtered Schemes also in more space dimensions (see \cite{FPT18a, PaolucciPhD} for the one-dimensional version). It is interesting to notice that, although the wideness of the stencil seems to limit excessively the accuracy in the $L^\infty$ norm in presence of some singularity, the simple AF-RKC4 gave very good responses, especially in terms of sharpness of the representation and of accuracy in regions of regularity, achieving the optimal order in most of the simulations. 
Also in this last example, similar comments on CPU times can be made looking at Tab. \ref{ex4:table3}. 
\begin{table}[h!]
\caption{Test 8.$T=3/(2\pi^2)$. CPU times in seconds. \label{ex4:table3}}
\centering
\begin{tabular}{c c| p{1.7cm}  | p{1.7Cm} | p{1.7cm}  | c }
$N_x$ & $N_t$ & \centering {\bf  F-HC } & \centering {\bf AF-HC }& \centering {\bf AF-RKC4}& {\bf WENO 2/3}\\
\hline
\hline
$20$ &	$20$&\centering 	$0.004\ s$ 	&\centering 	$0.011\ s$ 	&\centering $0.012\ s$  &$0.019\ s$   \\

$40$ &	$40$ &\centering 	$0.012\ s$  &\centering 	$0.082\ s$   &\centering $0.084\ s$  	& $0.194\ s$    \\

$80$ &	$80$ &\centering 	$0.077\ s$  &\centering  	$0.636\ s$   &\centering $0.723\ s$   	& $1.796\ s$ \\

$160$ &	$160$ &\centering 	$0.531\ s$  &\centering  	$4.883\ s$   &\centering $5.091\ s$   	 & $8.363\ s$ 	\\
\hline
\end{tabular}
\end{table}

\section{Conclusions}\label{sec:conclusions}
In this paper we presented a novel formulation for smoothness indicators for non-differentiable functions in two spatial dimensions generalizing the definition of \cite{JP00}, originally devised for the construction of the well-known WENO schemes for Hamilton-Jacobi equations. We give a very compact explicit formula for the computation of the indicators, which is straightforward to implement. Moreover, we compared the results of two slightly different formulations and shown the improvements with respect to simple indicators based on dimensional splitting.\\
We have applied the 2D indicators to the construction  of our multidimensional AF schemes. The filtering process is able to stabilize an otherwise unstable (high-order) scheme, still preserving its accuracy. The schemes in this paper are improvements of those presented in \cite{BFS16} where the switching parameter $\varepsilon$ has to be tuned by hand, moreover we  generalize the adaptive construction presented in \cite{FPT18a,FPT17} to more spatial dimensions. The adaptive AF scheme is able to reduce the oscillations which may appear choosing a constant $\varepsilon$ and, as shown by the numerical tests, gives always better results.\\
The main advantage of the filtered schemes relies in their simple implementation and in the extreme generality allowed for the choice of the high-order scheme. In the latter context, the adaptive definition of the parameter $\varepsilon^n$ and the stabilization properties of the function $\phi$, represent a relevant improvement with respect to the basic scheme \cite{BFS16}. That is why no further limiting correction is needed, thus preventing the risk of losing accuracy when using schemes of order of accuracy higher than $2$. The general applicability of our procedure has been testified by the successful implementation of simple and efficient fourth-order schemes in two space dimensions.
\section*{Acknowledgments}
\noindent
The authors are members of the INdAM Research Group GNCS, and gratefully acknowledge its financial support to this research. 
This work has also been funded by MIUR national project PRIN 2017, grant n. 2017KKJP4X.

\bibliographystyle{plain}
\bibliography{refs}

\appendix
\section{Brief review of one-dimensional indicators}\label{sec:ind_1D}
\noindent Let us begin by considering the 1D case and show how to define a \emph{smoothness indicator function} $\phi$, such that 
\begin{equation}\label{eq:phi_1D}
\phi^n_j=\phi(\omega^n_j):=\left\{
\begin{array}{ll}
1\qquad&\textrm{ if the solution } u^n\textrm{ is regular in }I_j,  \\
0&\textrm{ if }I_j \textrm{ contains a point of singularity,}
\end{array}
\right.
\end{equation}
where $I_j=(x_{j-1},x_{j+1})$ and $\omega^n_j$ is the \emph{smoothness indicator} at the node $x_j$ depending on the values of the approximate solution $u^n$. We consider the indicators of \cite{JP00}, which basically exploit the properties of the divided differences of the solution $u^n$ as singularity detectors. More precisely, we first define the main ingredients of the indicators, i.e. 
\begin{equation}
\label{beta}
\beta_k = \beta_k(u^n)_j := \sum_{l=2}^r\int_{x_{j-1}}^{x_{j}} \Delta x^{2l-3} \left( P_k^{(l)}(x)\right)^2 dx,
\end{equation}
for $k=0,\dots,r-1$, where $P_k$ is the Lagrange polynomial of degree $r$ interpolating the values of $u^n$ on the stencil $\mathcal S_{j+k}=\{x_{j+k-r},\dots,x_{j+k}\}$.
Then, before proceeding with the construction of $\phi$, let us state a fundamental result on the behavior of the indicators (\ref{beta}), proved in \cite{FPT18a}, considering a generic function $f$ and dropping the superscript $n$ in order to simplify the notation.
\begin{proposition}
\label{prop_beta}
Assume $f\in C^{r+1}\left(\Omega\setminus\{\xi\}\right)$, with $\Omega$ a neighborhood of $\xi$, and $f^\prime(\xi^-)\not = f^\prime(\xi^+)$. Moreover, let $f^{\prime\prime}(x)\not=0$, $\forall x\in (\Omega\setminus\{\xi\})$. Then, for $k=0,\dots,r-1$ and $j\in \mathbb Z$, the followings are true:
\begin{enumerate}[i)]
\item If $x_s \in \Omega\ \setminus \stackrel{\circ}{\mathcal S}_{j+k}\quad \Rightarrow \quad\beta(f)_{k}=O(\Delta x^2)$;
\item If $x_s \in \ \stackrel{\circ}{\mathcal S}_{j+k}\quad \Rightarrow \quad \beta(f)_k=O(1)$,
\end{enumerate}
where $\mathcal S_{j+k}=\left\{x_{j-r+k},\dots,x_{j+k}\right\}$ and $\stackrel{\circ}{\mathcal S}_{j+k}=(x_{j-r+k},x_{j+k})$.
\end{proposition}
\begin{remark}
\label{def_sigma}
Notice that we could avoid the restrictions on $f$ in the points of regularity by adding a small quantity $\sigma_h=\sigma \Delta x^2$, for some constant $\sigma>0$, to the indicators $\beta_k$ and considering instead
\begin{equation}
\label{beta_k_sigma}
\widetilde{\beta}_k:=\beta_k+\sigma_h,
\end{equation}
as it has been done in \cite{ABM10}. We will use this assumption in the sequel, choosing $\sigma=1$.
\end{remark}
Our aim is to identify the points (or the intervals) in which the approximate solution $u^n$ presents a singularity in the first derivative. To be precise, here with $u^n$ we mean any continuous function with nodal values $u^n_j$, $j\in \mathbb Z$. Let us focus the attention on a point $x_j$ of the grid and consider the simplest case of $r=2$, which is enough for our purpose.
Let us consider separately the intervals $(x_{j-1},x_{j}]$ and $[x_{j},x_{j+1})$ defining
\begin{equation}\label{beta_k_m}
\beta^-_k=\Delta x\int_{x_{j-1}}^{x_{j}} ( P_k^{\prime\prime}(x))^2 dx = \left(\frac{f_{j-k}-2f_{j-1-k}+f_{j-2-k}}{\Delta x}\right)^2,
\end{equation}
for $k=0,1$, where $P_0$, $P_1$ are the polynomials interpolating the solution, respectively, on the stencils $\{x_{j-2},x_{j-1},x_{j}\}$ e $\{x_{j-1},x_{j},x_{j+1}\}$; and symmetrically
\begin{equation}\label{beta_k_p}
\beta^+_k=\Delta x\int_{x_{j}}^{x_{j+1}} ( P_k^{\prime\prime}(x))^2 dx = \left(\frac{f_{j+k-1}-2f_{j+k}+f_{j+k+1}}{\Delta x}\right)^2,
\end{equation}
for $k=0,1$, where now $P_0$, $P_1$ are the interpolating polynomials on $\{x_{j-1},x_{j},x_{j+1}\}$ and $\{x_{j},x_{j+1},x_{j+2}\}$, respectively. From the definition it is clear that $(\beta^+)_j=(\beta^-)_{j+1}$ so we have to compute the quantities just once.
Then, we define as in \cite{JP00}
\begin{equation}
\label{def_alpha_pm}
\alpha^\pm_k=\frac{1}{(\beta^\pm_k+\sigma_h)^2},
\end{equation}
with $\sigma_h=\sigma \Delta x^2$ the parameter we introduced in Remark \ref{def_sigma}, and focus on the information given by the interpolating polynomial on $\{x_{j-1},x_{j},x_{j+1}\}$ defining
\begin{equation}\label{omegaPM}
\omega_+=\frac{\alpha^+_0}{\alpha^+_0+\alpha^+_1} \qquad\textrm{ and }\qquad \omega_-=\frac{\alpha^-_1}{\alpha^-_0+\alpha^-_1}
\end{equation}
to inspect the regularity on $[x_j,x_{j+1})$ and on $(x_{j-1},x_j]$, respectively. Then, using Prop. \ref{prop_beta} and Remark \ref{def_sigma}, it can be shown that 
if the solution is regular enough in both stencils, then
\begin{equation}
\label{relazione_omega}
\omega_\pm=\frac{\alpha_k^\pm}{\alpha_0^\pm+\alpha^\pm_1}=\frac{1}{2}+O(\Delta x),
\end{equation}
with $k=0$ for the superscript ``$+$'' and $k=1$ for ``$-$''.
On the other hand, if there is a singularity in at least one of the stencils then
\begin{equation}
\alpha^\pm_k=
\left\{
\begin{array}{ll}
O(1)\qquad&\textrm{ if } f \textrm{ is not smooth in }\mathcal S_{j+k} \\
O(\Delta x^{-4})&\textrm{ if } f \textrm{ is smooth in }\mathcal S_{j+k}.
\end{array}
\right.
\label{alpha_01}
\end{equation}
Consequently, it is easy to verify that, defining $\omega_j=\min\{\omega_-,\omega_+\}$ we get 
\begin{equation}
\omega_j=
\left\{
\begin{array}{ll}
O(\Delta x^4)\qquad &\textrm{ if }x_{j-1}<x_s<x_{j+1} \\
\frac{1}{2}+O(\Delta x)& \textrm{ otherwise.}
\end{array}
\right.
\end{equation}
See \cite{FPT17,FPT18a} for more details and the proofs.
Unfortunately, we noticed through numerical tests that the $O(\Delta x)$ term in regular regions may produce heavy oscillations around the optimal value $\overline\omega= 1/2$. To increase the accuracy, we can use higher order smoothness indicator ($r>2$), but we would need a bigger reconstruction stencil, or we can use the \emph{mappings} defined in \cite{HAP05},
\begin{equation}
g(\omega)=\frac{\omega(\overline \omega+\overline \omega^2-3 \overline \omega \omega +\omega^2)}{\overline \omega^2 +\omega(1-2\overline \omega)},\qquad \overline \omega \in (0,1)
\label{map}
\end{equation}
which have the properties that $g(0)=0$, $g(1)=1$, $g(\overline \omega)=\overline \omega$, $g^\prime(\overline \omega)=0$ and $g^{\prime\prime}(\overline \omega)=0$. Using Taylor expansion around $\overline \omega$, this directly implies that, if we compute $\omega_\pm^*=g(\omega_\pm)=\overline \omega +O(\Delta x^3)$, then the amplitude of the oscillations is drastically reduced.
Notice that with respect to the definition in \cite{HAP05} we avoided the second weighting which seems unnecessary in our case. More explicitly, the mapping we use is
\begin{equation}
 \label{map2}
g(\omega)=4\omega\left(\frac{3}{4}-\frac{3}{2}\omega+\omega^2\right).
\end{equation}

Another useful technique to reduce the oscillations, which in particular does not require any mapping, has been proposed in the context of hyperbolic conservation laws in \cite{BCCD08} and further generalized in \cite{CCD11}, leading to the definition of the so-called \emph{WENO-Z} schemes. This procedure can be applied for our purpose without any relevant change. For example, for $r=2$, it is implemented first defining
\begin{equation}
\label{tau_z}
\tau^{\pm}:=\left|\beta_0^\pm-\beta_1^\pm\right|,
\end{equation}
which has the properties:
\begin{itemize}
\item If $f$ is smooth in, respectively, $\mathcal S^\pm:=\mathcal S_0^\pm\bigcup \mathcal S_1^\pm$, then $\tau^\pm=O(\Delta x^3)$;
\item If $f$ is smooth in some $\mathcal S_k^\pm$, but not in $\mathcal S^\pm$, then $\tau^\pm\gg \beta_k^\pm$;
\item $\tau^\pm\leq \max_k \beta_k^\pm$.
\end{itemize}
Then, analogously to the usual WENO procedure, the final indicator is obtained computing
\begin{equation}
\label{WENO_Z}
\alpha_k^{Z,\pm}=\frac{1}{2}\left(1+{\left(\frac{\tau^\pm}{\beta_k^\pm +\sigma_h}\right)}^p\right),\qquad \omega_\pm^Z=\frac{\alpha^{Z,\pm}_\nu}{\alpha^{Z,\pm}_0+\alpha^{Z,\pm}_1},
\end{equation}
for $k=0,1$, with $p=2$, $\nu=1$ for the superscript `$-$' and $\nu=0$ for `$+$'. Unfortunately, applying directly the `WENO-Z' procedure to Hamilton-Jacobi equations, in particular in our context, we do not achieve improvements comparable to those obtained through the mapping (\ref{map2}). In fact, using similar arguments of the previous lines, it is straightforward to show that the resulting smoothness indicators are such that
\begin{equation}
\omega_\pm^Z=\frac{1}{2}+O(\Delta x^2),
\end{equation}
and thus produce slightly wider oscillations around the optimal value. This problem would probably suggest to increase to power $p$ in (\ref{WENO_Z}), but in doing so also the dependence on the magnitudes of successive derivatives of $f$ is increased, producing even wider oscillations for coarser grids, thus discouraging such an approach.

\begin{remark}\label{remark_tau_new}
It is worth of notice that, in the case $r=2$, making use of the full stencil $\{x_{j-2},\dots,x_{j+2}\}$ in defining $\tau^\pm$ in (\ref{tau_z})-(\ref{WENO_Z}) we can obtain better results in regular regions with respect to those of the mapped indicators. More precisely, if for both cases we define
\begin{equation}
\label{tau_new}
\tau_{new}=\left|\beta_0^--2\beta_1^-+\beta_1^+\right|,
\end{equation}
using Taylor expansions we have that
\begin{equation}
\label{sviluppo_tau_new}
\tau_{new}=2\Delta x^4\left((f_j^{\prime \prime\prime})^2+f_j^{\prime\prime}f_j^{(4)}\right)+O(\Delta x^5)=O(\Delta x^4),
\end{equation}
if the function is smooth in the full stencil. Consequently, it can be shown that the indicators computed through (\ref{WENO_Z})-(\ref{tau_new}), are such that
\begin{equation}
\omega^Z_{new}=\frac{1}{2}+O(\Delta x^4),
\end{equation}
in regions of regularity and have similar behavior with respect to the previous constructions in the remaining cases.
\end{remark}
For more details about the computations as well as other possible constructions, e.g. indicators with $r>2$, we refer the interested reader to \cite{PaolucciPhD}.
\end{document}